\documentclass[12pt]{amsart}

\usepackage{amscd,verbatim}
\usepackage{mathrsfs,yfonts,times}
\usepackage{amssymb,array}
\usepackage[all]{xy}
\usepackage{url}
\usepackage{enumitem}
\usepackage{xcolor}
\usepackage[stable]{footmisc}
\usepackage[colorlinks,linkcolor=blue,citecolor=blue,urlcolor=red]{hyperref}

\newcommand{\eq}[2]{\begin{equation}\label{#1}#2 \end{equation}}
\newcommand{\ml}[2]{\begin{multline}\label{#1}#2 \end{multline}}
\newcommand{\mlnl}[1]{\begin{multline*}#1 \end{multline*}}

\newcommand{\xr}[1] {\xrightarrow{#1}}

\newcommand{\lra}{\longrightarrow}
\newcommand{\inj}{\hookrightarrow}
\newcommand{\surj}{\twoheadrightarrow}

\newcommand{\A}{\mathbb{A}}

\newcommand{\G}{\mathbb{G}}

\renewcommand{\P}{\mathbb{P}}

\newcommand{\Z}{\mathbb{Z}}

\newcommand{\C}{\mathbb{C}}

\newcommand{\cA}{\mathcal{A}}

\newcommand{\cH}{\mathcal{H}}
\newcommand{\cK}{\mathcal{K}}
\newcommand{\sO}{\mathcal{O}}
\newcommand{\cO}{\mathcal{O}}
\newcommand{\cP}{\mathcal{P}}
\newcommand{\cU}{\mathcal{U}}
\newcommand{\cV}{\mathcal{V}}

\newcommand{\cX}{\mathcal{X}}
\newcommand{\cY}{\mathcal{Y}}

\newcommand{\sA}{\mathscr{A}}

\newcommand{\tF}{\widetilde{F}}
\newcommand{\tG}{\widetilde{G}}
\newcommand{\tH}{\widetilde{H}}

\newcommand{\fm}{\mathfrak{m}}

\newcommand{\Ab}{\operatorname{\mathbf{Ab}}}
\newcommand{\Sch}{\operatorname{\mathbf{Sch}}}

\newcommand{\Hom}{\operatorname{Hom}}

\newcommand{\Ker}{\operatorname{Ker}}
\newcommand{\Coker}{\operatorname{Coker}}

\renewcommand{\Im}{\operatorname{Im}}

\newcommand{\id}{\operatorname{id}}

\newcommand{\uHom}{\operatorname{\underline{Hom}}}

\newcommand{\Mod}{\operatorname{Mod}}
\newcommand{\Frac}{\operatorname{Frac}}

\newcommand{\eff}{{\operatorname{eff}}}

\newcommand{\op}{{\operatorname{op}}}

\newcommand{\codim}{{\operatorname{codim}}}
\newcommand{\ch}{{\operatorname{ch}}}
\newcommand{\CH}{{\operatorname{CH}}}
\newcommand{\pro}{{\operatorname{pro}}}

\renewcommand{\lim}{\operatornamewithlimits{\varprojlim}}

\newcommand{\ol}{\overline}
\newcommand{\wt}{\widetilde}

\newcommand{\Tr}{\operatorname{Tr}}
\newcommand{\tr}{{\operatorname{tr}}}
\newcommand{\Nm}{\operatorname{Nm}}

\newcommand{\Spec}{\operatorname{Spec}}

\newcommand{\Div}{\operatorname{Div}}
\renewcommand{\div}{\operatorname{div}}

\newcommand{\Res}{\operatorname{Res}}

\newcommand{\Lin}{\operatorname{Lin}}
\newcommand{\dlog}{\operatorname{dlog}}

\newcommand{\Zar}{{\operatorname{Zar}}}
\newcommand{\Nis}{{\operatorname{Nis}}}

\newcommand{\e}{{\epsilon}}

\renewcommand{\o}{{\operatorname{o}}}
\newcommand{\ul}{\underline}
\newcommand{\ot}{\widetilde}

\newcommand{\Sm}{\operatorname{\mathbf{Sm}}}
\newcommand{\Cor}{\operatorname{\mathbf{Cor}}}
\newcommand{\PST}{{\operatorname{\mathbf{PST}}}}

\newcommand{\NST}{\operatorname{\mathbf{NST}}}
\newcommand{\DM}{\operatorname{\mathbf{DM}}}

\newcommand{\MSm}{\operatorname{\mathbf{MSm}}}
\newcommand{\uMSm}{\operatorname{\mathbf{{\underline{M}}Sm}}}

\newcommand{\MCor}{\operatorname{\mathbf{MCor}}}
\newcommand{\uMCor}{\operatorname{\mathbf{{\underline{M}}Cor}}}
\newcommand{\ulMCor}{\operatorname{\mathbf{{\underline{M}}Cor}}}
\newcommand{\bcube}{{\ol{\square}}}

\newcommand{\uomega}{{\ul{\omega}}}

\newcommand{\Pro}{\operatorname{\mathbf{Pro}}}

\newcommand{\MPST}{\operatorname{\mathbf{MPST}}}
\newcommand{\uMPST}{\operatorname{\mathbf{{\underline{M}}PST}}}
\newcommand{\ulMPST}{\operatorname{\mathbf{{\underline{M}}PST}}}
\newcommand{\CI}{{\operatorname{\mathbf{CI}}}}

\newcommand{\RSC}{{\operatorname{\mathbf{RSC}}}}
\newcommand{\HI}{{\operatorname{\mathbf{HI}}}}

\newcommand{\MDM}{\operatorname{\mathbf{MDM}}}

\newcommand{\te}{{\otimes}}
\newcommand{\tte}{{\underline{\otimes}}}

\newcommand{\ten}[1]{{\overset{#1}{\otimes}}}
\newcommand{\tten}[1]{{\overset{#1}{\underline{\otimes}}}}

\newcommand{\bu}[1]{{\overset{#1}{\bullet}}}
\newcommand{\tw}[1]{{\langle{#1}\rangle}}

\theoremstyle{plain}

\newtheorem{prop}{Proposition}[section]
\newtheorem{lem}[prop]{Lemma}
\newtheorem{cor}[prop]{Corollary}
\newtheorem{thm}[prop]{Theorem}

\theoremstyle{definition}
\newtheorem{defn}[prop]{Definition}

\newtheorem{problem}[prop]{Problem}
\newtheorem{para}[prop]{}

\theoremstyle{remark}
\newtheorem{rmk}[prop]{Remark}
\newtheorem{remark}[prop]{Remark}

\numberwithin{equation}{prop}

\begin{document}

\title[Tensor structures]{Tensor structures in the theory of \\ 
modulus presheaves with transfers}
\author{Kay R\"ulling, Rin Sugiyama \and Takao Yamazaki}
\date{\today}
\address{Bergische Universit\"at Wuppertal, Gau\ss{}str 20, 42119 Wuppertal, and}
\address{Technische Universit\"at M\"unchen,Boltzmannstr. 3, 85748 Garching}
\email{ruelling@uni-wuppertal.de}
\address{Department of Mathematical and Physical sciences, Japan Women's University, 2-8-1 Mejirodai, Bunkyo-ku Tokyo, 112-8681 Japan}
\email{sugiyamar@fc.jwu.ac.jp}
\address{Graduate School of Science, Mathematical Institute, Tohoku University,
  Aoba, Sendai 980-8578, Japan}
\email{ytakao@math.tohoku.ac.jp}


\begin{abstract}
The tensor product of $\A^1$-invariant sheaves with
transfers introduced by Voevodsky is generalized to reciprocity sheaves
via the theory of modulus presheaves with transfers.
We prove several general properties of this construction and compute it in
some cases. In particular we obtain new (motivic) presentations
of the absolute K\"ahler differentials and the first infinitesimal neighborhood of the
diagonal.
\end{abstract}

\thanks{The first author was supported by the DFG Heisenberg Grant RU 1412/2-2.
The second author is supported by JSPS KAKENHI Grant (JP16K17579).
The third author is supported by JSPS KAKENHI Grant (JP18K03232).}
\maketitle

\tableofcontents

\section{Introduction}

The tensor structure on
Voevodsky's category $\HI$ of homotopy invariant presheaves with transfers can be generalized 
to the category $\RSC$ of reciprocity sheaves.  We study some properties of this construction.

Our motivation partly comes from 
the category $\MDM^\eff(k)$ of \emph{motives with modulus} \cite{KMSY3}.
This is a non-$\A^1$-invariant generalization of 
Voevodsky's triangulated category of effective motives $\DM^\eff(k)$.
The latter category admits a $t$-structure whose center is precisely $\HI$.
For this reason,
the computation of the tensor product in $\HI$ 
embodies some important cases of motivic cohomology
(see \eqref{eq:tenHI-intro1}, \eqref{eq:tenHI-intro2} below).
We will carry out a similar computation in $\RSC$
(see Theorems \ref{intro-thm1}, \ref{intro-thm2}).
As with $\HI$, $\RSC$ is expected to be (a good approximation of)
the heart of the homotopy $t$-structure of $\MDM^\eff(k)$.
Similarly, $\RSC$ also has close relationship with 
\emph{logarithmic motives} studied in \cite{BPO}.
We expect that, in due course, our results will be interpreted 
in terms of motivic cohomology with modulus.

We now give a more precise description of our work.

\begin{para}\label{para1a0}
Let $\sA$ be an additive category with a symmetric monoidal structure $\te$.
It is well-known, that there exists a monoidal structure on the category of presheaves
(of abelian groups) on $\sA$ which is uniquely characterized by the fact that it is right exact and 
that the Yoneda functor is monoidal. 
\end{para}
\begin{para}\label{para1a}
We recall Voevodsky's tensor structure on $\HI$.
Denote by $\Cor$  the category of finite correspondences
whose objects are smooth varieties over a perfect base field $k$
and by $\PST$ the category of presheaves on it.
The  product $X\times Y$ defines a monoidal structure on $\Cor$.
The tensor structure induced via \ref{para1a0} on $\PST$ is denoted by $\ten{\PST}$.
Denote by  $\HI$ the full subcategory of $\A^1$-invariant presheaves.  
It becomes a symmetric monoidal category via $F \ten{\HI} G:=h_0^{\A^1}(F \ten{\PST} G)$,
where $h_0^{\A^1}(F)$ denotes the maximal $\A^1$-invariant quotient of $F\in \PST$.
It is a basic property, that for proper smooth varieties $X$ and $Y$ this construction gives
\begin{equation}\label{eq:tenHI-intro1}
(h_0^{\A^1}(X)\ten{\HI} h^{\A^1}_0(Y))(k)= \CH_0(X\times Y),
\end{equation}
where $h_0^{\A^1}(X)=h_0^{\A^1}(\Z_\tr(X))$, with $\Z_\tr: \Cor\to \PST$ the Yoneda functor, and where
$\CH_0$ denotes the Chow group of zero dimensional cycles.
(We abbreviate $F(R)=F(\Spec R)$ 
for $F \in \PST$ and a $k$-algebra $R$.) 
Furthermore, by the work of many authors we have for every regular local $k$-algebra $R$
\begin{equation}\label{eq:tenHI-intro2}
\G_m^{\ten{\HI} n}(R)\cong \cH^n(\Z(n))(R)=\cK^M_n(R),
\end{equation}
where $\G_m$ is the multiplicative group, $\Z(n)$ is Voevodsky's motivic complex, and $\cK^M_n$ is 
the improved Milnor $K$-theory \cite[1.]{Ke10}.
Note also that $F\ten{\HI} G(k)$ can be  explicitly  described in terms of generators and relations
in the style of Somekawa's $K$-groups, see  \cite{KY}.
\end{para}

\begin{para}\label{para1b}
The generalization of $\ten{\HI}$
is based on the theory of {\em modulus pairs},
developed in \cite{KMSY1, KMSY2, KSY3}.
Recall that a modulus pair $M=(\ol{M}, M^\infty)$
is a pair of a separated $k$-scheme of finite type
and an effective Cartier divisor $M^\infty$ on $\ol{M}$
such that $M^\o := \ol{M} \setminus M^\infty$ is smooth over $k$
(see \S \ref{sec:umcor} for details).
We say $M$ is {\em proper} if $\ol{M}$ is.
We have the category $\MCor$ at our disposal
whose objects are proper modulus pairs 
(see \S \ref{sec:mcor}). Denote by $\MPST$ the category of presheaves on it.
For $M, N \in \MCor$ we put
\begin{equation}\label{eq:oldmonoidstr}
M \otimes N 
:=(\ol{M} \times \ol{N}, \ol{M} \times N^\infty + M^\infty \times \ol{N}).
\end{equation}
The tensor structure induced via \ref{para1a0} on $\MPST$ is denoted by $\te$ in the following.
Denote by $\CI$ the full subcategory of $\bcube$-invariant presheaves, where $\bcube=(\P^1, (\infty))$.
Any $F \in \MPST$ admits 
a maximal quotient $h_0^\bcube(F)$ that belongs to $\CI$.
By \cite[Proposition 2.1.9]{KSY3}, 
$F\ten{\CI} G:= h_0^{\bcube}(F\te G)$ defines a symmetric monoidal
structure on $\CI$.  Consider the functor
\[ \omega : \MCor \to \Cor, 
\qquad \omega(\ol{M}, M^\infty) =\ol{M} \setminus M^\infty.
\]
It induces 
$\omega^* : \PST \to \MPST$  by pre-composition with $\omega$,
and $\omega_! : \MPST \to \PST$ by the left Kan extension.
The essential image of $\CI$ via $\omega_!$ agrees with 
the full subcategory $\RSC$ of $\PST$
consisting of $F$ having {\em SC-reciprocity},
introduced in \cite{KSY3} (see \S \ref{subsec:RSC}).

In this article we study the local sections of 
\[h_0(F\te G):= \omega_!h_0^{\bcube}(F\te G),\]
for certain $F,G\in \MPST$.  
A first technical  result, which is used in all the applications below,
is a presentation in terms of generators and relations of
$h_0(F\te G)(K)$, for any $k$-field $K$, see Theorem \ref{thm:CI-te-fields}.
This presentation is a modulus-version of the one found in the homotopy invariant case 
in \cite[Theorem 6.2]{KY} (for $K=k$) and in \cite[Theorem 5.1.8]{IR} (for general $K$).
For example, using that for $F\in \RSC$ the Nisnevich - and the Zariski sheafification agree, see 
\cite[Corollary 3.2.2]{KSY3},
Theorem \ref{thm:Ga-te-Ga-arbCh} and Theorem \ref{thm:DR-arbCh}, give the following:

\begin{thm}\label{intro-thm1}
Set $\G_a^\#:=\ker(h_0^\bcube(\P^1, 2(\infty)) \to h_0^\bcube(\Spec k, \emptyset)=\Z)$
and 
$\G_m^\#:=\ker(h_0^\bcube(\P^1, (0)+(\infty)) \to h_0^\bcube(\Spec k, \emptyset))$.
Let $R$ be a regular local $k$-algebra. 
\begin{enumerate}[label= (\arabic*)]
\item\label{intro-thm1.1} Assume $\ch(k)\neq 2$. Then
\[h_0(\G_a^\# \te \G_a^\#)(R)\cong (R\otimes_\Z R)/I_{\Delta_R}^2\ (\cong R\oplus \Omega^1_{R/\Z}),\]
where $I_{\Delta_R}$ is the kernel of $R\otimes_\Z R \to R$, $a\otimes b\to ab$.
\item\label{intro-thm1.2} Assume $\ch(k)\neq 2,3,5$. Then 
\[h_0(\G_a^\#\te (\G_m^\#)^{\te n})(R)\cong \Omega^n_{R/\Z},\]
where $\Omega^n_{-/\Z}$ denotes the absolute K\"ahler differentials of degree $n$.
\end{enumerate}
\end{thm}
As a corollary of \ref{intro-thm1.1} we obtain for any field $K$ of $\ch(K)\neq 2$
\[ \CH_0(\P^1_K \times \P^1_K, 2(\P^1_K \times (\infty) + (\infty) \times \P^1_K))
\cong \Z \oplus K^{\oplus 3} \oplus  \Omega_{K/\Z}^1,
\]
where the left hand side denotes the Chow groups of zero-cycles of Kerz-Saito.
Moreover,  \ref{intro-thm1.2} and \cite{Bloch-Esnault} yield (see Corollary \ref{Cor:hCH-GaM})
\[h_0((\P^1, 2(\infty))\otimes (\G_m^\#)^{\otimes n})(K)\cong \CH_0((\P^1_K, 2(\infty)), n),\]
for any field $K$ of $\ch(K)\neq 2,3,5$,
where the right hand side denotes the higher Chow-groups with modulus introduced by Binda-Saito in 
\cite{Binda-Saito} (generalizing the additive higher Chow groups of Bloch-Esnault).
Note that the left hand side is more in the spirit of Voevodsky's definition of motivic cohomology as defined in  \cite{MVW}.
\end{para}

\begin{para}\label{para1c}
The sheaves $\G_a^\#$ and $\G_m^\#$ from Theorem \ref{intro-thm1} 
are lifts to $\CI$ of $\G_a$ and $\G_m$, respectively, in the sense that 
\[\omega_!\G_a^\#=\G_a, \qquad \omega_!\G_m^\#=\G_m.\]
On the other hand, there is also a functorial way to lift any $F\in \RSC$ to $\CI$. Indeed,
the restriction of $\omega_!$ to $\CI$  
has a right adjoint $\omega^\CI: \RSC\to \CI$ and we have $\omega_!\omega^{\CI} F=F$.
We obtain a functor
\eq{para1c1}{\RSC^{\times n}\to \RSC, \quad 
 (F_1,\ldots, F_n)\mapsto h_{0}(\omega^{\CI}F_1\te\ldots\te \omega^{\CI}F_n).}
We do not know whether this defines a symmetric monoidal structure on $\RSC$ since
associativity is not clear;  see Corollaries \ref{cor:assoc} and \ref{cor:assocII} for something in this direction.
Note that even if $F$ is representable, $\omega^{\CI}F$ in general will not be and there is also no
simple way to find a resolution in terms of representable presheaves. This makes it potentially more difficult
to compute $\omega^{\CI}F\otimes \omega^{\CI} G$, compared to the cases considered 
in Theorem \ref{intro-thm1}.
However, the presheaf on the right in \eqref{para1c1} has a universal property in the sense that it represents
certain multi-linear maps, see Theorem \ref{thm:ten-rep}. A key point in the proof of this theorem, is 
the fact stemming from \cite{RS}, that a section of $F(\ol{M}\setminus M_\infty)$ lies in 
$\omega^{\CI}F(M)$ if and only if
it lies in $F(S, M_{\infty|S})$, for all morphisms $S\to \ol{M}$,
with $S$ the spectrum of a henselian discrete valuation field over $k$.
Hence we only have to understand $\omega^{\CI}F\otimes \omega^{\CI}G$ on pro-smooth modulus pairs of
dimension 1, which is accomplished by Proposition \ref{prop:ten-rep-1dim}.

Using  Theorem \ref{thm:ten-rep} we get the natural maps  in the following theorem.
\begin{thm}[cf Theorems \ref{thm:tHI-tRSC}, \ref{thm:Ga-te-Ga}, \ref{thm:DR}]\label{intro-thm2}
Let $R$ be a regular local $k$-algebra.
\begin{enumerate}[label = (\arabic*)]
\item\label{intro-thm2.1} Assume $F_i \in \HI$, then 
      \[h_0(\omega^{\CI}F_1\te\ldots\te \omega^{\CI}F_n)(R)\xr{\simeq} (F_1\ten{\HI}\ldots\ten{\HI} F_n)(R).\]
\end{enumerate}

In the following we assume $\ch(k)=0$.
\begin{enumerate}[label=(\arabic*), start=2]
\item\label{intro-thm2.3} 
\[h_0(\omega^{\CI}\G_a\te \omega^{\CI}\G_a)(R)\xr{\simeq} (R\otimes_\Z R)/I_{\Delta_R}^2.\]
\item\label{intro-thm2.4}
 \[h_0(\omega^{\CI}\G_a\te (\omega^{\CI}\G_m)^{\te n})(R)   \xr{\simeq} \Omega^n_{R/\Z}.\]
\end{enumerate}
\end{thm}
All the isomorphisms above can be enhanced to isomorphisms in the category of
Nisnevich sheaves in $\RSC$. In particular, \ref{intro-thm2.1} says that \eqref{para1c1} is the extension to $\RSC$
of the tensor structure on $\HI$ mentioned at the beginning. An important ingredient in the proofs
is Saito's result that  Nisnevich sheafification preserves $\RSC$, which generalizes the analogue statement for $\HI$
by Voevodsky, see \cite{S}.  Another key point is the  injectivity property of reciprocity sheaves proved in
\cite{KSY}, which implies that a morphism between reciprocity sheaves is an isomorphism if it is so on all function fields.
Computations from \cite{RS} are used to define the natural maps in the above theorem.

We want to stress, that \ref{intro-thm2.3} and \ref{intro-thm2.4} are not true in positive characteristic.
For example, if $\ch(k)=p>0$ the left hand side in \ref{intro-thm2.3} vanishes for $R=k$
(see Corollary \ref{ga-ga-p}); also in this case 
${\rm Frob}_{\G_a}\otimes \id$ acts on the left hand side in \ref{intro-thm2.4},
but it does not on the right hand side.

As an application  of these methods, we obtain that certain higher Chow groups  of zero-cycles
with modulus condition defined via either the {\em sum} or the {\em ssup} convention do agree,
see Corollary \ref{Cor:hCH-GaM}.
\end{para}

\begin{para}
We are going to study another tensor structure 
$\tte$ on $\MPST$.
This arises as follows.
Recall that we employed the monoidal structure on $\MCor$
defined by \eqref{eq:oldmonoidstr}.
This is {\em not} a (categorical) product on the category of proper modulus pairs $\MSm$.
As we will see in Proposition \ref{prop:prod-msm},
there does exist a product, which we denote by $\tte$.
This in turn 
induces another tensor structure $\tte$ on $\MPST$.
Although it does not behave very well 
with $\CI$ (see Remark \ref{rmk:interval}),
we will also explore this second tensor product.
For instance, we will prove an analogue of the
Theorems \ref{intro-thm1} and \ref{intro-thm2} for $\tte$.
We also show, that on function fields we get in this way the $K$-groups of reciprocity functors from \cite{IR} 
back, see Theorem \ref{cor:tenT}.
In fact, the isomorphism from Theorem \ref{intro-thm2}\ref{intro-thm2.4} in the case $\tte$ was proven for 
function fields in \cite{IR}. Note however, that the left hand side of  Theorem \ref{intro-thm2}\ref{intro-thm2.3}
vanishes in the case $\tte$, as a consequence of \cite[Thm 5.5.1]{IR}.
Finally we want to mention, that the comparison with the $K$-groups of  reciprocity functors,
together with \cite{RY14} imply the following (see Corollary \ref{cor:prop-tten}):
Assume $\ch(k)=0$. Let $A$ be an abelian variety over $k$ and $F\in \RSC$. Then for every regular local $k$-algebra
\[h_0(\omega^{\CI}A\te \omega^{\CI}\G_a\te \omega^{\CI}F)(R)=0.\]
\end{para}

\noindent
{\bf Acknowledgment.}
Part of this work was done while 
the first author was a visiting professor at the TU M{\"u}nchen.
He thanks Eva Viehmann for the invitation and the support.
The project started when the third author was visiting 
Freie Universit\"at Berlin.
He thanks their hospitality and support.

Many of the results in 
\S \ref{sec:prod-modpair} and \S \ref{sec:mono-str}
are obtained in the collaboration of the third author with
Bruno Kahn and Shuji Saito.
He would like to acknowledge special thanks to them.
The authors  thank the referee for his/her careful reading and valuable comments.

\section{Products of modulus pairs}\label{sec:prod-modpair}
We fix once and for all a perfect field $k$.

\begin{para}\label{sect:cor}
Let $\Sch$ be the category of 
separated schemes of finite type over $k$.
Denote by $\Sm$ the full subcategory of $\Sch$
consisting of smooth schemes over $k$.
Both $\Sch$ and $\Sm$ have finite products
given by the (scheme-theoretic) fiber product over $k$.
We write them simply by $\times$ (instead of $\times_k$).
For $X \in \Sm$ and $Y \in \Sch$,
an integral closed subscheme of $X \times Y$
that is finite and surjective over a component of $X$
is called a prime correspondence from $X$ to $Y$.
The category $\Cor$ of finite correspondences
has the same objects as $\Sm$,
and for $X, Y \in \Sm$ the 
group of morphisms $\Cor(X, Y)$ is
the free abelian group on the set of
all prime correspondences from $X$ to $Y$.
We consider $\Sm$ as a subcategory of $\Cor$
by regarding a morphism in $\Sm$ as its graph in $\Cor$.
The product $\times$ on $\Sm$ 
yields a symmetric monoidal structure on $\Cor$.
\end{para}

\begin{para}\label{sec:umcor}
We recall the definition of the category 
$\uMCor$ from \cite[Definition 1.3.1]{KMSY1}.
A pair $M=(\ol{M}, M^\infty)$
of $\ol{M} \in \Sch$ and 
an effective Cartier divisor $M^\infty$ on $\ol{M}$ 
such that $M^\o:=\ol{M} \setminus |M^\infty| \in \Sm$
is called a modulus pair.
Let $M, N$ be modulus pairs
and $\Gamma \in \Cor(M^\o, N^\o)$ be a prime correspondence.
Let $\ol{\Gamma} \subset \ol{M} \times \ol{N}$
be the closure of $\Gamma \subset M^\o \times N^\o$,
and let
$\ol{\Gamma}^N \to \ol{\Gamma}$ be the normalization.
We say $\Gamma$ is admissible 
if $M^\infty|_{\ol{\Gamma}^N} \ge N^\infty|_{\ol{\Gamma}^N}$,
where $(-)|_{\Gamma^N}$ denotes the pull-back
of a Cartier divisor to $\ol{\Gamma}^N$.
We say $\Gamma$ is left proper
if $\ol{\Gamma}$ is proper over $\ol{M}$.
Let $\uMCor(M, N)$ be the subgroup of $\Cor(M^\o, N^\o)$
generated by all admissible left proper prime correspondences.
The category $\uMCor$ has modulus pairs as objects
and $\uMCor(M, N)$ as the group of morphisms from $M$ to $N$.
\end{para}

\begin{para}
We recall the category $\uMSm$ from \cite[Definition 1.3.2]{KMSY1}.
It has the same objects as $\uMCor$,
and the morphisms are intersection of those of $\uMCor$
and $\Sm$.
This means that for modulus pairs $M, N$,
a morphism $f: M \to N$ in $\uMSm$ 
is given by a morphism $f^\o : M^\o \to N^\o$ in $\Sm$
such that its graph $\Gamma_{f^\o} \in \Cor(M^\o, N^\o)$
is admissible and left proper.
We regard $\uMSm$ as a subcategory of $\uMCor$.
\end{para}

\begin{para}\label{sect:omega-lambda}
We have functors
\begin{equation}\label{eq:omega-lambda}
 \lambda : \Cor \to \uMCor, \,\, \lambda(X)=(X, \emptyset);
\quad
 \ul{\omega} : \uMCor \to \Cor, \,\, \ul{\omega}(M)=M^\o,
\end{equation}
and $\lambda$ is a left adjoint of $\ul{\omega}$
\cite[Lemma 1.5.1]{KMSY1}.
They restrict to a corresponding adjoint pair
\begin{equation}\label{eq:omega-lambda2}
 \lambda : \Sm \to \uMSm, \,\, \lambda(X)=(X, \emptyset);
\quad
 \ul{\omega} : \uMSm \to \Sm, \,\, \ul{\omega}(M)=M^\o.
\end{equation}
\end{para}

\begin{para}\label{sec:te}
There is a symmetric monoidal structure
$\te$ on $\uMSm$ and $\uMCor$
given by
\begin{equation}\label{eq:old-tensor}
(\ol{M}, M^\infty) \otimes
(\ol{N}, N^\infty) 
= (\ol{M} \times \ol{N}, M^\infty \times \ol{N} + \ol{M} \times N^\infty).
\end{equation}
which is used in \cite{KSY3}.
All of the functors 
in \eqref{eq:omega-lambda} and \eqref{eq:omega-lambda2}
are monoidal with respect to $\te$ and $\times$.
However, $\otimes$ is {\it not} a product in $\uMSm$.
(For instance, the diagonal map on $M^\o$
is not admissible for $M \to M \otimes M$
unless $M^\infty=\emptyset$.)
Nevertheless, it is shown in \cite[Corollary 1.10.8]{KMSY1}
that the category $\uMSm$ has finite products,
and it yields another symmetric monoidal structure on $\uMCor$,
see Proposition \ref{prop:prod-msm} below.
(See also Remark \ref{rmk:interval} for a drawback.)
For this, we need a preparation.
\end{para}

\begin{para}
Let $M$ be a modulus pair.
A {\it modulus replacement} of $M$ is 
a morphism $f : M' \to M$ in $\uMSm$ such that
(i) $f^\o : {M'}^\o \to M^\o$ is an isomorphism, 
(ii) $f^\o$ extends to a proper morphism
$\ol{f} : \ol{M'} \to \ol{M}$ in $\Sch$, and
(iii) ${M'}^\infty=\ol{f}^*M^\infty$.
\end{para}

\begin{lem}\label{lem:replacement}
\begin{enumerate}
\item 
Any modulus replacement $f : M' \to M$
is an isomorphism in $\uMSm$.
\item 
Let $f : M' \to M$ be a modulus replacement
and let $N \in \uMSm$.
Then, a prime correspondence
$\Gamma \in \Cor(M^\o, N^o)$ defines an element of  $\uMCor(M, N)$
if and only if
 $\Gamma' :=
(f^\o \times \id_{N^\o})^{-1}(\Gamma) \in \Cor({M'}^\o, N^\o)$
belongs to $\uMCor(M', N)$.
\item
Let $M, N$ be modulus pairs and let
$g^\o : M^\o \to N^\o$ be a morphism in $\Sm$.
Suppose that the closure 
$\ol{\Gamma} \subset \ol{M} \times \ol{N}$
of the graph of $g^\o$
is proper over $\ol{M}$.
Then there is a modulus replacement
$f : M' \to M$ such that
${g'}^{\o}:= g^\o \circ f^\o : {M'}^\o \to N^\o$
extends to a morphism 
$\ol{g'} : \ol{M'} \to \ol{N}$ in $\Sch$.
Moreover, 
$g^\o$ defines a morphism 
$g : M \to N$ in $\uMSm$
if and only 
${g'}^{\o}$ defines a morphism 
$g' : M' \to N$ in $\uMSm$.
\end{enumerate}
\end{lem}
\begin{proof}
The closure of the graph $\Gamma_{(f^\o)^{-1}}$ of
$(f^\o)^{-1} : M^\o \to {M'}^\o$ in $\ol{M} \times \ol{M'}$
is the transpose ${}^t \Gamma_{\ol{f}}$ 
of the graph of $\ol{f} : \ol{M'} \to \ol{M}$,
which is proper over $\ol{M}$ as $\ol{f}$ is.
The condition (iii) in the definition of modulus replacement
immediately shows that 
$\Gamma_{(f^\o)^{-1}}$ is admissible as well.
Thus $(f^\o)^{-1}$ defines the inverse of $f$ in $\uMSm$.
This proves (1), and (2) is an immediate consequence.

We prove (3).
Let $\ol{M'} \to \ol{\Gamma}$ be the normalization.
The composition 
$\ol{f} : \ol{M'} \to \ol{\Gamma} \hookrightarrow
\ol{M} \times \ol{N} \to \ol{M}$
is a proper morphism restricting to an isomorphism
$f^\o : \ol{f}^{-1}(M^\o) \cong M^\o$. 
Hence it defines a modulus replacement
$f : M'=(\ol{M'}, \ol{f}^*M^\infty) \to M$,
and $g^\o \circ f^\o$ extends to 
a morphism 
$\ol{g'} : \ol{M'} \to \ol{\Gamma} \hookrightarrow
\ol{M} \times \ol{N} \to \ol{N}$.
This proves the first statement.
The rest follows from (2).
\end{proof}

\begin{para}\label{ss:const-newprod}
Take $M_1, M_2 \in \uMSm$.
Let $\pi : \ol{M}\to\ol{M}_1 \times \ol{M}_2$
be the blow-up along $M_1^\infty \times M_2^\infty$.
Define closed subschemes $E, D_1$ and $D_2$ of $\ol{M}$ 
as the exceptional divisor,
the strict transform of $M_1^\infty \times \ol{M}_2$,
and that of $\ol{M}_1 \times M_2^\infty$, respectively.
They can be locally described as follows.
Suppose $\ol{M}_i = \Spec A_i$
and $M_i^\infty$ is defined by the ideal generated by $t_i \in A_i$.
Then $\ol{M} = \Spec R_1 \cup \Spec R_2$, where
\[ R_1 = A_1 \otimes A_2 [u]/(t_1 \otimes 1 - u(1 \otimes t_2))
\text{ and }
   R_2 =A_1 \otimes A_2 [v]/(1 \otimes t_2 - v(t_1 \otimes 1)),
\]
and patching is given by $uv=1$.
The divisor $E$ is defined by the ideals generated 
by $1 \otimes t_2 \in R_1$ and $t_1 \otimes 1 \in R_2$.
Similarly, $D_1$ (resp. $D_2$) is defined 
by $u \in R_1$ and $1 \in R_2$
(resp. $1 \in R_1$ and $v \in R_2$).
From this description, we see that
$D_i$ are effective Cartier divisors
and that
\begin{equation}\label{eq:prod-disjoint}
|D_1| \cap |D_2|=\emptyset, \quad
\ol{p}_i^*M_i^\infty = E + D_i,
\end{equation}
where $\ol{p}_i : \ol{M} \to \ol{M}_i$ is the canonical map.
Since $E$ is also a Cartier divisor, we can define
\begin{equation}\label{eq1.12a}
M_1 \tte M_2 :=
(\ol{M}, E+D_1+D_2) \in \uMSm. 
\end{equation}
Put $M:=M_1 \tten{} M_2$. 
Note that $\ol{p}_i$ defines a morphism in $\uMSm$
\begin{equation}\label{eq:proj-map}
p_i : M \to M_i \qquad \text{for}~ i=1, 2.
\end{equation}
On the other hand, 
the blow-up $\pi$ defines a modulus replacement
\[ M':=(\ol{M}, 2E+D_1+D_2) \overset{\cong}{\longrightarrow}
 M_1 \otimes M_2.
\]
Since ${M'}^\infty \ge M^\infty$,
the identity map on $M_1^\o \times M_2^\o$
defines a canonical morphism
\begin{equation}\label{eq1.12}
M_1\otimes M_2 \to M_1\tten{} M_2.
\end{equation}
This is an isomorphism if and only if 
$M_1^\infty=\emptyset$ or $M_2^\infty=\emptyset$.
Note also that this becomes an isomorphism in $\Sm$
after applying $\uomega$ \eqref{eq:omega-lambda2}.
\end{para}

We now prove that
the category $\uMSm$ has finite products.
Part (1) of the following proposition is obtained in
\cite[Corollary 1.10.8]{KMSY1},
but we include its proof to keep this article self-contained.

\begin{prop}\label{prop:prod-msm}
\begin{enumerate}
\item 
Let $M_1, M_2 \in \uMSm$
and let $M=M_1 \tten{} M_2 $ be
the object defined in \eqref{eq1.12a}.
Then $M$ is a categorical product of $M_1, M_2$ in $\uMSm$.
In particular, $\tte$ is bifunctorial in $M_1$ and $M_2$.
\item 
Let $M_i, N_i \in \uMCor$,
$\Gamma_i \in \uMCor(M_i, N_i)$ for $i=1, 2$
and let $M := M_1 \tten{} M_2, ~N := N_1 \tten{} N_2$.
Then $\Gamma_1 \times \Gamma_2 \in 
\Cor(M_1^\o \times M_2^\o, N_1^\o \times N_2^\o)$
belongs to  $\uMCor(M, N)$.
Consequently,
$\tten{}$ 
yields a symmetric monoidal structure on $\uMCor$.
\item
All of the functors 
in \eqref{eq:omega-lambda} and \eqref{eq:omega-lambda2}
are monoidal with respect to $\tte$ and $\times$.
\end{enumerate}
\end{prop}
\begin{proof}
(1) 
Take $N \in \uMSm$ and $f_i : N \to M_i ~(i=1, 2)$ in $\uMSm$.
We must show that there exists a unique $f : N \to M$
such that $f_i = p_i \circ f$,
where $p_i$ is from \eqref{eq:proj-map}.
By construction,
we find the functor $\ul{\omega}$ from \eqref{eq:omega-lambda2}
commutes with $\tte$ and $\times$
(as it should be,
since $\ul{\omega}$ has a left adjoint $\lambda$,
and $\tte$ will be seen to be a product in $\uMSm$).
Hence $f$ should be given by the product map
$f^\o : N^\o \to M^\o_1 \times M^\o_2$ of $f_1^\o$ and $f_2^\o$,
showing uniqueness.

Now we turn to the existence. 
Let $\ol{\Gamma} \subset \ol{N} \times \ol{M}$
be the closure of the graph of
$f^\o:=f_1^\o \times f_2^\o : N^\o \to M_1^\o \times M_2^\o$.
Let $\ol{\Gamma}_i$ be the closure of the graph of 
$f_i^\o : N^\o \to M_i^\o$ in $\ol{N} \times \ol{M}_i$
for $i=1, 2$.
We then have a commutative diagram
\[
\xymatrix{
\ol{\Gamma} \ar[rr] \ar@{}[d]|-*{\cap} & &
\ol{\Gamma}_1 \times \ol{\Gamma}_2 \ar@{}[d]|-*{\cap}
\ar@/^70pt/[dd]
\\
\ol{N} \times \ol{M} 
\ar[rr]^-{\delta \times \pi} \ar[d] 
&&
\ol{N} \times \ol{N} \times \ol{M}_1 \times \ol{M}_2  \ar[d]
\\
\ol{N} \ar[rr]^-{\delta} & & \ol{N} \times \ol{N},
}
\]
where 
$\pi : \ol{M} \to \ol{M}_1 \times \ol{M}_2$ is the blow-up and 
$\delta : \ol{N} \to \ol{N} \times \ol{N}$ is the diagonal
closed immersion.
The right winding arrow is proper because 
$\ol{\Gamma}_1$ and $\ol{\Gamma}_2$ are proper over $\ol{N}$.
We conclude that $\ol{\Gamma}$ is proper over $\ol{N}$.

Therefore, by Lemma \ref{lem:replacement},
we may assume that
$f^\o:=f_1^\o \times f_2^\o : N^\o \to M_1^\o \times M_2^\o$
extends to a morphism $\ol{f} : \ol{N} \to \ol{M}$.
In particular, $f_i^\o : N^\o \to M_i^\o$ extends to
a morphism $\ol{f}_i: \ol{N} \to \ol{M}_i$ for $i=1,2$. 
The admissibility of $f_i$ means that 
\[ N^\infty \ge \ol{f}_i^* M_i^\infty
= \ol{f}^* \ol{p}_i^* M_i^\infty = \ol{f}^*E + \ol{f}^*D_i,
\]
where the last equality follows from \eqref{eq:prod-disjoint}.
This proves the desired inequality
\begin{equation}\label{ineq:prod-div}
N^\infty \ge \ol{f}^*E + \ol{f}^*D_1 + \ol{f}^*D_2 = \ol{f}^*M^\infty
\end{equation}
restricted to $\ol{M} \setminus |\ol{f}^*D_i|$ for $i=1, 2$,
but we have
$|\ol{f}^* D_1| \cap |\ol{f}^* D_2| 
= \emptyset$ by \eqref{eq:prod-disjoint},
hence \eqref{ineq:prod-div} holds on the whole $\ol{M}$.
This completes the proof of (1).

(2)
We may suppose $\Gamma_i$ are prime correspondences.
Let $\ol{\Gamma}_i$ (resp. $\ol{\Gamma}$) 
be the closure of $\Gamma_i$ (resp. $\Gamma_1 \times \Gamma_2$)
in $\ol{M}_i \times \ol{N}_i$ (resp. $\ol{M} \times \ol{N}$).
We have a commutative diagram
\[
\xymatrix{
\ol{\Gamma} \ar[rr] \ar@{}[d]|-*{\cap} \ar@/_70pt/[dd] &&
\ol{\Gamma}_1 \times \ol{\Gamma}_2 \ar@{}[d]|-*{\cap} 
\ar@/^70pt/[dd]
\\
\ol{M} \times \ol{N} \ar[rr]^-{\pi_M \times \pi_N} \ar[d]&&
\ol{M}_1 \times \ol{M}_2 \times \ol{N}_1 \times \ol{N}_2  \ar[d]
\\
\ol{M} \ar[rr]^-{\pi_M}&&
\ol{M}_1 \times \ol{M}_2 
}
\]
where $\pi_M : \ol{M} \to \ol{M}_1\times \ol{M}_2$
and  $\pi_N : \ol{N} \to \ol{N}_1\times \ol{N}_2$
are the blow-up maps.
By assumption $\ol{\Gamma}_i$ are proper over $\ol{M}_i$,
hence so is the right winding arrow.
It follows that the left winding arrow is proper too.

Let $\ol{\Gamma}^N$ be the normalization of $\ol{\Gamma}$
and let $\nu : \ol{\Gamma}^N \to \ol{M} \times \ol{N}$
be the canonical map.
It remains to prove 
\begin{equation}\label{eq:bifunc}
 \nu^*(M^\infty \times \ol{N}) \geq \nu^*(\ol{M} \times N^\infty). 
\end{equation}
Let $\ol{\Gamma}_i^N$ be the normalization of $\ol{\Gamma}_i$.
The composition map 
$\ol{\Gamma}^N \to \ol{M} \times \ol{N} \to \ol{M}_i \times \ol{N}_i$
factors through $\mu_i : \ol{\Gamma}^N \to \ol{\Gamma}_i^N$.
By assumption, we have 
$\nu_i^*(M_i^\infty \times \ol{N}_i) 
\geq \nu_i^*(\ol{M}_i \times N_i^\infty)$,
where 
$\nu_i : \ol{\Gamma}_i^N \to \ol{M}_i \times \ol{N}_i$
is the canonical map.
Using a commutative diagram
\[
\xymatrix{
\ol{\Gamma}^N \ar[d]_{\mu_i} \ar[r]^-{\nu} & 
\ol{M} \times \ol{N} \ar[d]^-{\ol{p}_i} 
\ar[rr]^-{\pi_M \times \pi_N} & &
\ol{M}_1 \times \ol{M}_2 \times \ol{N}_1 \times \ol{N}_2  
\ar[dll]^{q_i}
\\
\ol{\Gamma}_i^N \ar[r]_-{\nu_i} &
\ol{M}_i \times \ol{N}_i, & &
}
\]
where $q_i$ is the projection and 
$\ol{p}_i=q_i \circ (\pi_M \times \pi_N)$, we obtain
\begin{equation}\label{eq:prod-ineq}
\nu^* \ol{p}_i^*(M_i^\infty \times \ol{N}_i)
\geq \nu^* \ol{p}_i^*(\ol{M}_i \times N_i^\infty)
\quad \text{for $i=1,2$}.
\end{equation}

We write $D_1^M,~ D_2^M$ and $E^M$ for
the strict transform of $M_1^\infty \times \ol{M}_2$,
that of $\ol{M}_1 \times M_2^\infty$ and
the exceptional divisor, respectively.
They are effective Cartier divisors on $\ol{M}$,
and we have $M^\infty = D_1^M+D_2^M+E^M$.
We use similar notations for 
$D_1^N,~ D_2^N,~ E^N \subset \ol{N}$.
Then we have
$\ol{p}_i^*(M_i^\infty \times \ol{N}_i) = (E^M + D_i^M) \times \ol{N}$
and 
$\ol{p}_i^*(\ol{M}_i \times N_i^\infty) = \ol{M} \times (E^N + D_i^N)$.
Hence
\eqref{eq:prod-ineq} can be rewritten as
\begin{equation}\label{eq:prod-ineq2}
\nu^*((E^M + D_i^M) \times \ol{N}) \geq
\nu^*(\ol{M} \times (E^N + D_i^N)) \quad \text{for $i=1,2$}.
\end{equation}
On the other hand, we also have
$|D_1^M| \cap |D_2^M| = |D_1^N| \cap |D_2^N| = \emptyset.$
Hence it suffices to show \eqref{eq:bifunc}
after restricting to
\[
U_{i, j}:=
\nu^{-1}((\ol{M} \setminus |D_i^M|) \times (\ol{N} \setminus |D_j^N|))
\quad \text{for $i, j =1, 2$}.
\]

Suppose first $(i,j)=(1,1)$.
We have
\[
\label{ineq:n}
M^\infty|_{\ol{M} \setminus |D_1^M|} 
= (E^M + D_2^M)|_{\ol{M} \setminus |D_1^M|}, \quad
N^\infty|_{\ol{N} \setminus |D_1^N|} 
= (E^N + D_2^N)|_{\ol{N} \setminus |D_1^N|}.
\]
Hence \eqref{eq:bifunc} on $U_{11}$
follows from \eqref{eq:prod-ineq2} for $i=2$.
The case $(i,j)=(2,2)$ is shown in the same way.
Suppose now $(i,j)=(2,1)$.
We have
\[
 M^\infty|_{\ol{M} \setminus |D_2^M|} 
= (E^M + D_1^M)|_{\ol{M} \setminus |D_2^M|}
\geq E^M|_{\ol{M} \setminus |D_2^M|}.
\]
The left hand side of \eqref{eq:prod-ineq2} for $i=2$
restricts to $\nu^*(E^M \times \ol{N})$ on $U_{21}$.
Hence we get
\begin{align*}
&\nu^*(M^\infty \times \ol{N})|_{U_{21}}
\geq \nu^*(E^M \times \ol{N})|_{U_{21}}
\\
&\geq \nu^*(\ol{M} \times (E^N + D_2^N))|_{U_{21}}
=\nu^*(\ol{M} \times N^\infty)|_{U_{21}},
\end{align*}
as desired.
The case $(i,j)=(1,2)$ is shown in the same way.
(2) is proved.
(3) is obvious.
\end{proof}

\begin{rmk}\label{rmk:interval}
Let $\bcube := (\P^1, \infty) \in \uMSm$
so that $\bcube^\o=\A^1$.
The map $\A^1 \times \A^1 \to \A^1, (x, y) \mapsto xy$
is {\it not} admissible for 
$\bcube \tte \bcube \to \bcube$.
In particular, 
the interval structure on $\A^1$ used in \cite{MV}
cannot be lifted to $\uMCor$
with respect to the monoidal structure $\tte$.
On the other hand,
the same map is admissible for 
$\bcube \te \bcube \to \bcube$,
and it yields such a lift with respect to $\te$.
This is the main advantage of $\te$ over $\tte$.
\end{rmk}

\begin{para}\label{sec:mcor}
Recall from \cite[Definition 1.1]{KMSY1}
that a modulus pair $M$ is called proper if $\ol{M}$ is.
Denote by $\MSm$ and $\MCor$
the full subcategory of $\uMSm$ and $\uMCor$
consisting of proper modulus pairs, respectively.
Our construction of $\tte$ from \eqref{eq1.12a}
restricts to a product on $\MSm$
and to a symmetric monoidal structure on $\MCor$.
(The construction of $\te$ from \eqref{eq:old-tensor}
restricts to  symmetric monoidal structures on $\MSm$ and $\MCor$
as well.)
We use the same letter $\tau$ 
\begin{equation}\label{eq:tau}
\tau : \MSm \to \uMSm, \qquad \tau : \MCor \to \uMCor
\end{equation}
for the inclusion functors,
which are thus monoidal for both $\te$ and $\tte$.
It follows that $\omega := \uomega \tau$ is monoidal 
for both $(\te, \times)$ and $(\tte, \times)$:
\begin{equation}\label{eq:omega} 
\omega : \MSm \to \Sm, \qquad \omega : \MCor \to \Cor.
\end{equation}
\end{para}

\section{Monoidal structures on modulus presheaves with transfers}\label{sec:mono-str}
Tensor products of
(modulus) presheaves with transfers are introduced.

\begin{para}\label{sec:mod-a}
Let $\sA$ be an additive category 
with a symmetric monoidal structure $\bullet$,
and let $\Mod(\sA)$  be the abelian category 
of all additive functors $\sA^\op \to \Ab$.
Recall, that there exists a symmetric monoidal structure
$\bullet$ on $\Mod(\sA)$ (unique up to equivalence)
such that 
the (additive) Yoneda functor $y : \sA \to \Mod(\sA)$
is monoidal, and
such that $\bullet$ is right exact (see, e.g., \cite[Lecture 8]{MVW}, or \cite[A8, A9]{KY}).
The monoidal structure $\bullet$ on $\Mod(\sA)$
is closed, that is, $(-)\bullet F$ has a right adjoint $\uHom(F,-)$, for all $F\in \Mod(\sA)$.
For a representable object, it is given by
\begin{equation}\label{eq:uhom}
\uHom(y(A), F')(B)=F'(A \bullet B),
\quad A, B \in \sA, ~F' \in \Mod(\sA).
\end{equation}
\end{para}

\begin{para}\label{sec:pst}
The standard closed symmetric monoidal structure on 
the category of presheaves with transfers
$\PST := \Mod(\Cor)$ is obtained by
applying the general machinery in \S \ref{sec:mod-a}
to $\sA=\Cor$ and $\bullet=\times$ from \S \ref{sect:cor}.
We denote it by $\ten{\PST}$, or
simply by $\te$ when no confusion can occur.
For $X \in \Sm$,
we use the standard notation
$\Z_\tr(X)=y(X)$ for the Yoneda embedding.
\end{para}

\begin{para}\label{sec:mpst}
Let us apply the results of \S \ref{sec:mod-a}
to $\sA=\uMCor$ from \S \ref{sec:umcor}.
The category
$\uMPST :=\Mod(\uMCor)$ of modulus presheaves with transfers
carries two closed symmetric monoidal structures 
$\te=\ten{\uMPST}$ and $\tte=\tten{\uMPST}$,
which are deduced by $\te$ from \S \ref{sec:te}
and by $\tte$ from \eqref{eq1.12a}, respectively.
For $M \in \uMCor$,
we write $\Z_\tr(M)=y(M) \in \uMPST$ for the Yoneda embedding.
The left Kan extension of $\uomega$ from \eqref{eq:omega-lambda}
\begin{equation}\label{eq:uomega-sh}
\uomega_! : \uMPST \to \PST
\end{equation}
is monoidal with respect to both $(\ten{\uMPST}, \ten{\PST})$ 
and $(\tten{\uMPST}, \ten{\PST})$.
It is a localization and is exact, as it has both left and right adjoints
(see \cite[Proposition 2.3.1]{KMSY1}).
\end{para}

\begin{lem}\label{lem:comp-te-tte}
There is a natural transformation of bifunctors 
$\Psi : \te \to \tte$
which is compatible with \eqref{eq1.12} through the Yoneda embedding.
Moreover, for any $F, F' \in \uMPST$ we have
\[ \uomega_!(\Coker(\Psi_{F, F'} : F \te F' \to F \tte F'))=0.
\] 
\end{lem}
\begin{proof}
It suffices to consider the case
$F=\Z_\tr(M), ~F'=\Z_\tr(M')$ for $M, M' \in \uMCor$.
Then the morphism \eqref{eq1.12} yields an injective morphism
\begin{multline*}
 \Psi_{\Z_\tr(M), \Z_\tr(N)} : 
\Z_\tr(M) \te \Z_\tr(N)=\Z_\tr(M \te N) \\
\to \Z_\tr(M \tte N)=\Z_\tr(M) \tte \Z_\tr(N)
\end{multline*}
in $\uMPST$
with desired property $\uomega_!(\Coker(\Psi_{\Z_\tr(M), \Z_\tr(N)}))=0$.
\end{proof}

\begin{para}\label{sec:tte-MPST}
Similarly,
we apply \S \ref{sec:mod-a} to $\sA=\MCor$ from \S \ref{sec:mcor}
to obtain two closed symmetric monoidal structures
$\te=\ten{\MPST}$ and $\tte=\tten{\MPST}$
on $\MPST:=\Mod(\MCor)$.
The functor $\tau : \MCor \to \uMCor$ from \eqref{eq:tau}
induces a functor
\begin{equation}\label{eq:tau-sh}
\tau_! : \MPST \to \uMPST,
\end{equation}
which is monoidal for both
$(\ten{\MPST}, \ten{\uMPST})$ and $(\tten{\MPST}, \tten{\uMPST})$.
It is also exact by \cite[Proposition 2.4.1]{KMSY1}.
Hence the same holds for
the left Kan extension
\begin{equation}\label{eq:omega-sh}
\omega_!=\uomega_! \tau_! : \MPST \to \PST,
\end{equation}
of $\omega : \MCor \to \Cor$ from \eqref{eq:omega}.
We use the same notation
$\Z_\tr(M)=y(M) \in \MPST$ for the Yoneda embedding
of $M \in \MCor$.
\end{para}

\begin{para}\label{subsec:CI}
We say $F \in \MPST$ is $\bcube$-invariant
if $p^* : F(M) \to F(M \otimes \bcube)$ is an isomorphism
for any $M \in \MCor$,
where $\bcube:=(\P^1, \infty) \in \MCor$
and $p : M \otimes \bcube \to M$ is the projection.
Let $\CI$ be the full subcategory of $\MPST$
consisting of all $\bcube$-invariant objects.
Recall from \cite[Theorem 2.1.8]{KSY3}
that $\CI$ is a Serre subcategory of $\MPST$,
and that
the inclusion functor $i^\bcube : \CI \to \MPST$
has a left adjoint $h_0^\bcube$ and a right adjoint $h^0_\bcube$:
\begin{align}\label{eq:h0-bcube} 
&h_0^\bcube(F)(M)
=\Coker(i_0^* - i_1^* : F(M \te \bcube) \to F(M)),
\\
\notag
&h^0_\bcube(F)(M)=\Hom(h_0^\bcube(M), F),
\end{align}
for $F \in \MPST$ and $M \in \MCor$.
Here 
$i_\epsilon : (\Spec k, \emptyset) \to \bcube$ denotes
the embedding with image $\epsilon$, for $\epsilon \in \{ 0, 1 \}$.
Note that $h_0^\bcube$ is right exact and a localization.
We write $h_0^\bcube(M):=h_0^\bcube(\Z_\tr(M))$ for $M \in \MCor$.
\end{para}

\begin{remark}
We stress that the monoidal structure involved here is $\te$ and not $\tte$.
Indeed, if we simply replace $\te$ by $\tte$ in the above discussion,
then we may loose the $\bcube$-invariance of $h_0^\bcube(F)$,
because $\bcube$ does not admit an interval structure
with respect to $\tte$ (see Remark \ref{rmk:interval}).
\end{remark}

\begin{remark}\label{rmk:te-ci}
In \cite[Proposition 2.1.9]{KSY3},
it is proved that $\CI$ has 
a symmetric monoidal structure $\te=\ten{\CI}$
for which $h_0^\bcube$ is monoidal.
In particular, for $F, F' \in \CI$, 
one can compute $F \te F'$ by taking
$\ot{F}, \ot{F'} \in \MPST$ such that
$h_0^\bcube(\ot{F})=F$ and $h_0^\bcube(\ot{F'})=F'$
(one may simply take $\ot{F}=i^\bcube(F), \ot{F'}=i^\bcube(F')$),
and then
\eq{tenCI}{
F \ten{\CI} F' = h_0^\bcube(\ot{F} \ten{\MPST} \ot{F'}).
}
It is also closed:
\[
\uHom_{\CI}(F, F') = h^0_\bcube(\uHom_{\MPST}(F, F')).
\]

Unfortunately, a similar construction for $\tte$ does not work,
since our definition of $\CI$ is based on $\te$.
\end{remark}

\begin{lem}\label{lem:hcube-ten}
We have
\[h_0^{\bcube}(F\ten{\MPST} G)\cong h_0^{\bcube}(F\ten{\MPST} h_0^{\bcube}(G)), \quad F,G\in \MPST.\]
\end{lem}
\begin{proof}
This follows, e.g.,  from the fact that \eqref{tenCI} is well-defined.
\end{proof}

\begin{para}\label{subsec:RSC}
We define $h_0(F) := \omega_!(h_0^\bcube(F)) \in \PST$, for $F \in \MPST$,
and $h_0(M) := h_0(\Z_\tr(M))$, for $M \in \MCor$
(see \eqref{eq:omega-sh}, \eqref{eq:h0-bcube}).
It comes equipped with a canonical surjection
$\omega_!(\Z_\tr(M))=\Z_\tr(M^\o) \twoheadrightarrow h_0(M)$.
We say $F \in \PST$ has {\em SC-reciprocity} if
for any $X \in \Sm$ and $a \in F(X)$,
there exists a {\em modulus} $M$, i.e., $M\in \MCor$ such that $M^\o=X$
and the Yoneda map $a: \Z_\tr(X) \to F$ defined by $a$ factors through $h_0(M)$.
We define $\RSC$ to be the full subcategory of $\PST$
consisting of objects having SC-reciprocity.
It is an abelian subcategory of $\PST$.
Denote by $i^\natural : \RSC \to \PST$ the inclusion functor,
which is exact and has a right adjoint $\rho : \PST \to \RSC$.
For $F \in \PST$, the counit map $i^\natural \rho F \to F$
is an isomorphism if and only if $F \in \RSC$
(see \cite[Remark 2.2.5, Proposition 2.2.6, Corollary 2.2.7]{KSY3}).
\end{para}

\begin{para}\label{subsec:omegaCI}
Recall from \cite[Proposition 2.3.7]{KSY3} that
$\omega_!$ from \eqref{eq:omega-sh} restricts to 
an exact localization functor
\begin{equation}\label{eq:omega-ci}
\omega_{\CI} : \CI \to \RSC.
\end{equation}
Thus we obtain a functor
\eq{subsec:omegaCI1}{h_0: \MPST\to \RSC,\quad F\mapsto h_0(F)=\omega_!(h_0^{\bcube}(F)).}
Note that by the exactness and monoidality of $\omega_!$,  the natural surjection $F\to h^\bcube_0(F)$, for $F\in \MPST$,
induces a natural surjection
\eq{subsec:omegaCI2}{
(\omega_!F_1)\ten{\PST}\ldots \ten{\PST} (\omega_!F_n)\surj h_0(F_1\bullet\ldots\bullet F_n),}
where $F_i\in\MPST$ and $\bullet\in \{\ten{\MPST}, \tten{\MPST}\}$.

The functor $\omega_{\CI}$ has a right adjoint $\omega^\CI : \RSC \to \CI$,
and the counit map $\omega_\CI \omega^\CI \Rightarrow \id_\RSC$ 
is an isomorphism. Concretely, for $F\in \RSC$ and $M\in \MCor$  we have 
\eq{subsec:omegaCI3}{\omega^{\CI}F(M)
=\{a\in F(M^\o)\mid a \text{ has modulus }M\} \subset F(M^o).}
For brevity we set
\eq{subsec:omegaCI4}{\tF:=\omega^{\CI}F\in \CI, \quad F\in \RSC.}
\end{para}

\begin{para}\label{para:Chow-mod}
For a proper modulus pair $M=(\ol{M}, M_\infty)\in \MCor$,
we denote by $\CH^i(M; j)$ the Chow group with modulus 
introduced in \cite{Binda-Saito}.
When $\ol{M}$ is of pure dimension $d$,
we also write $\CH_0(M):=\CH^d(M):=\CH^d(M; 0)$.
Let $K$ be a $k$-field, then we have an isomorphism
\[\CH^d(M_K) \cong h_0^\bcube(M)(K),\]
where $M_K=(\ol{M}\times_{\Spec k}\Spec K, M_\infty\times_{\Spec k}\Spec K)$, see \cite[3.5 (3)]{RY}.
\end{para}

\begin{para}\label{sec:omega-h}
Let $\HI$ be the full subcategory of $\PST$
consisting of $F \in\PST$ such that
$p^* : F(X) \to F(X \times \A^1)$ is an isomorphism
for any $X \in \Sm$,
where $p : X \times \A^1 \to X$ is the projection.
We have $\HI \subset \RSC$, by \cite[{Corollary 2.3.4}]{KSY3}.

The inclusion functor $i^\flat : \HI \to \PST$ has a left adjoint
\begin{equation}\label{eq:h0-a1} 
h_0^{\A^1} : \PST \to \HI, \quad
F \mapsto \Coker(i_0^* - i_1^* : \uHom(\Z_\tr(\A^1), F) \to F)
\end{equation}
with the similar notation as \eqref{eq:h0-bcube}.
We write $h_0^{\A^1}(X):=h_0^{\A^1}(\Z_\tr(X))$ for $X \in \Sm$.
The symmetric monoidal structure $\te=\ten{\HI}$ is defined in such a way that
$h_0^{\A^1}$ becomes monoidal:
\[
F \ten{\HI} G = h_0^{\A^1}((i^\flat F) \ten{\PST} (i^\flat G)),
\quad F, G \in \HI.
\]
\end{para}

\begin{para}
Let $\NST \subset \PST$ be the category of Nisnevich sheaves with transfers,
and define $\HI_\Nis := \HI \cap \NST, ~\RSC_\Nis := \RSC \cap \NST$.
Recall that the inclusion functor $\PST \to \NST$
admits a left adjoint $F \mapsto F_\Nis$.
We will need the following important results.
\begin{thm}[{\cite[{Theorem 3.1.12}]{VoTmot}, \cite[{Theorem 0.1}]{S}}]\label{thm:sheafification}
\
\begin{enumerate}
\item 
We have $F_\Nis \in \HI_\Nis$, for any $F \in \HI$.
\item 
We have $F_\Nis \in \RSC_\Nis$, for any $F \in \RSC$.
\end{enumerate}
\end{thm}
It follows that $\HI_\Nis$ and $\RSC_\Nis$ are abelian  categories,
and that 
$\HI_\Nis$ has a symmetric monoidal structure
$\ten{\HI_\Nis}$ given by 
\[
F \ten{\HI_\Nis} G =(F \ten{\HI} G)_\Nis,
\qquad F, G \in \HI_\Nis.
\]
For $F\in \MPST$ we set $h_{0, \Nis}(F):=(h_0(F))_{\Nis}$ and we obtain a functor
\eq{h0nis}{h_{0,\Nis}: \MPST\to \RSC_{\Nis}.}
For $M\in \MCor$ we also write $h_{0,\Nis}(M):=h_{0,\Nis}(\Z_\tr(M))$.
\end{para}

\section{Lax monoidal\footnote{The term {\em lax monoidal} is used in a loose sense;
it seems a correct mathematical notion which appears in the literature is {\em unbiased oplax monoidal category}.}
 structure for reciprocity sheaves}

\begin{para}\label{subsec: MCorpro}
Recall  the category $\uMCor^{\rm pro}$ from \cite[3.7]{RS}:
The objects are pairs $\cX=(\ol{X}, X^\infty)$, where 
\begin{enumerate}
\item $\ol{X}$ is a separated noetherian scheme over $k$;
\item $\ol{X}=\varprojlim_{i\in I} \ol{X}_i$, with $(\ol{X}_i)_{i\in I}$  a projective system of 
separated  finite type $k$-schemes
indexed by a partially ordered set with affine transition maps $\tau_{i,j}: \ol{X}_i\to \ol{X}_j$, $i\ge j$, and 
$X^\infty=\varprojlim_{i\in I} X_i^\infty$, with $X_i^\infty$ an effective Cartier divisor on $\ol{X}_i$,
such that $\ol{X}_i\setminus |X_i^\infty|$ is smooth, for all $i$, and 
$\tau_{i,j}^*X_{j,\infty}= X_{i,\infty}$, $i\ge j$;
\item $\cX^\o=\ol{X}\setminus |X^\infty|$ is regular.
\end{enumerate}
The morphisms are given by the admissible left proper correspondences which are verbatim defined as in 
\ref{sec:umcor}; the composition is defined as for $\uMCor$, cf. \cite[Proposition 1.2.3]{KMSY1}.

We denote by $\Cor^{\rm pro}$ the full subcategory consisting of pairs $(X,\emptyset)=X$.
Note, if $X=\Spec A$, with $A$ a regular ring over $k$, then $X\in \Cor^\pro$, by \cite[Proposition 1.8]{MR868439}.
By \cite[Lemma 2.8]{RS}, there is a faithful functor
\eq{sec: MCorpro1}{\uMCor^{\rm pro}\to \Pro\ulMCor, \quad \varprojlim_i(\ol{X}_i, X^\infty_i)\mapsto (\ol{X}_i, X^\infty_i)_i.}
In this way we can extend any $F\in \uMPST$ to $\uMCor^{\pro}$ by the formula
\eq{sec: MCorpro2}{F(\cX):=\varinjlim_i F(\cX_i), \quad\text{where } \cX=\varprojlim_i \cX_i,\quad \cX_i=(\ol{X}_i, X^\infty_i).}
\end{para}

\begin{lem}\label{lem:rep-pro}
Let $M\in \uMCor$ and $\cX=\varprojlim_i \cX_i\in \uMCor^{\pro}$. Then the natural map
\[\ulMCor^{\pro}(\cX, M)\xr{\simeq} \Pro\uMCor((\cX_i)_i, M)=\varinjlim_i \uMCor(\cX_i, M),\]
induced by \eqref{sec: MCorpro1}, is an isomorphism, i.e., the extension of  
 $\Z_\tr(M)$ to $\uMCor^\pro$ is representable by $M\in \uMCor^\pro$.
\end{lem}
\begin{proof}
The morphism in the statement is defined as follows:
Let $V\subset \cX^\o\times M^\o$ be an integral prime correspondence; 
we find an index $i_0$ and a closed subscheme  $V_{i_0}\subset \cX^\o_{i_0}\times M^\o$ such that 
$V_{i_0}\in \uMCor(\cX_{i_0}, M)$ and $V_{i_0}\times_{\cX^\o_{i_0}} \cX^\o= V$; then the map from the statement 
maps $V$ to the class of $V_{i_0}\in \varinjlim_i\uMCor(\cX_i, M)$, see \cite[Lemma 3.8]{RS}.

We define a map in the other direction: Let $\rho_i : \cX\to \cX_i$ be the transition map.
Let $W\subset \uMCor(\cX_i, M)$ be a prime correspondence. In particular, $W\to \cX_i^\o$ is
 universally equidimensional of relative dimension. Hence  the cycle-theoretic inverse image $\rho_i^* V$
from \cite[V, C), 7.]{SerreAL} is defined; it is supported on the irreducible components of 
$\cX^\o\times_{\cX^\o_i}W$. Let $W'\in\Cor(\cX^\o, M^\o)$
 be such a component. By, e.g., \cite[Proposition 2.3, (4)]{RY}, $W'$ is an admissible correspondence
from $\cX$ to $M$; further  since $W$ is left proper (see \ref{sec:umcor}), so is $W'$. 
Hence $\rho_i^* V\in \uMCor^{\pro}(\cX, M)$. This gives a well-defined map
\[\varinjlim_i \uMCor(\cX_i, M)\to \uMCor^\pro(\cX, M).\]
It is direct to check that it is inverse to the map from the statement.
\end{proof}

\begin{lem}\label{lem:tensorMPST}
Let $\bullet\in\{\te, \tte\}$ denote one of the symmetric monoidal structures on $\uMCor$ defined in 
\ref{sec:te} and Proposition \ref{prop:prod-msm}, respectively; we also denote by $\bullet$ their extension
to $\uMPST$, see \ref{sec:tte-MPST}. Let $F_1,\ldots, F_n\in \uMPST$; 
we extend $F_1\bullet\ldots\bullet F_n\in \uMPST$ to $\uMCor^{\pro}$ as in 
\eqref{sec: MCorpro2}. Then for all $\cX\in \uMCor^\pro$, the group 
$(F_1\bullet\ldots\bullet F_n)(\cX)$ is equal to the quotient of
\[\bigoplus_{M_1,\ldots, M_n\in \uMCor}F_1(M_1)\te_{\Z} \ldots \te_{\Z}F_n(M_n)\te_{\Z}
\uMCor^\pro(\cX, M_1\bullet\ldots\bullet M_n)\]
by the subgroup generated by elements of the form
\ml{lem:tensorMPST1}{a_1\te \ldots\te a_n\te
(\id_{M_{1}\bullet \ldots\bullet M_{i-1}}\bullet f\bullet\id_{M_{i+1}\bullet \ldots\bullet M_n})\circ h\\
- a_1\te \ldots \te f^*a_i\te\ldots \te a_n\te h,}
where $a_j\in F_j(M_j)$,  $i\in \{1,\ldots, n\}$,  $f\in \uMCor(M',M_i)$, and 
$h\in \uMCor^\pro(\cX, M_1\bullet\ldots\bullet M_{i-1}\bullet M'\bullet M_{i+1}\bullet\ldots \bullet M_n)$.
\end{lem}
\begin{proof}
First assume $\cX\in \uMCor$. Then the formula holds using the following  presentation of $F\in \uMPST$
\[\bigoplus_{f\in \uMCor(M',M)} F(M)\te_\Z \Z_\tr(M')\to \bigoplus_{M\in\uMCor} F(M)\te_\Z \Z_\tr(M)\to F\to 0,\]
 cf. \cite[\S 2]{SuVo00b}. 
For general $\cX=\varprojlim_i\cX_i\in \uMCor^\pro$, with $\cX_i\in\uMCor$, we have by definition
$(F_1\bullet\ldots\bullet F_n)(\cX)=\varinjlim_i ((F_1\bullet\ldots\bullet F_n) (\cX_i))$; 
hence the formula follows from the exactness of filtered direct limits, their compatibility
with (usual) tensor products, and Lemma \ref{lem:rep-pro}.
\end{proof}

\begin{lem}\label{lem:ten-cond}
We continue to use the notation $\bullet\in\{\te,\tte\}$ from Lemma \ref{lem:tensorMPST}.
Let  $M_i\in \uMCor$, $i=1,2$, and set $M:=M_1\bullet M_2$.
Let $Y\to \ol{M}$ be a morphism from a locally  factorial scheme $Y$, such that the image of no component of $Y$
is contained in $M^\infty$. Then
\[(M^\infty)_{|Y}=
\begin{cases} 
(M^\infty_1)_{|Y} + (M^\infty_2)_{|Y}, & \text{if }\bullet=\te,\\
\max\{(M^\infty_1)_{|Y}, (M^\infty_2)_{|Y}\}, &\text{if }\bullet=\tte,
\end{cases}\]
where the restrictions are along the morphisms $Y\to \ol{M}\to\ol{M}_i$ induced by the projections.
Note, the right hand side makes sense as Weil divisors and so it does as Cartier divisors, by the assumptions on $Y$.
\end{lem}
\begin{proof}
For $\bullet=\te$, this is immediate from the definition. For $\bullet=\tte$ recall that
$\ol{M}$ is the blow-up of $\ol{M_1}\times \ol{M_2}$ in $M_1^\infty\times M_2^\infty$ 
and that $M^{\infty}=D_1+D_2+E$, with $D_i$ the strict transform of $p_i^*M_i^\infty$ and $E$ the exceptional divisor.
Set $U_i:= Y\times_{\ol{M}} (\ol{M}\setminus  D_i)$. Then
\[(M^\infty)_{|U_1}= (D_2+E)_{|U_1}= \max\{(M^\infty_1)_{|U_1}, (M^\infty_2)_{|U_1}\}\]
and similar with $(M^\infty)_{|U_2}$. The statement  follows from  $Y=U_1\cup U_2$.
\end{proof}

\begin{lem}\label{lem:fin-sur-pro}
Let $\cX=(X, D)\in \uMCor^{\pro}$.
Assume $\cX=\varprojlim_i \cX_i$ with $\cX_i=(X_i, D_i) \in \uMCor$ such that
the projection maps $X\to X_i $ are {\em flat}.
Let $Y$ be a regular scheme with a finite surjective morphism $\pi: Y\to X$ and let $E$ be 
an effective Cartier divisor on $Y$.
Then $(Y, E)\in \uMCor^{\pro}$. In particular the graph of $\pi$ and its transpose define morphisms
\[\Gamma_{\pi}\in \uMCor^\pro((Y,\pi^* D),\cX ), \quad \Gamma_{\pi}^t\in \uMCor^\pro(\cX, (Y, \pi^*D)).\]
\end{lem}
\begin{proof}
Since $X$ is noetherian, the $\sO_X$-algebra $\cA=\pi_*\sO_Y$ is a coherent $\sO_X$-module.
It follows that there exists an index $i_0$ and a coherent $\sO_{X_{i_0}}$-algebra $\cA_{i_0}$ such that $\cA_{i_0| X}=\cA$. 
Set $Y_{i_0}= \Spec \cA_{i_0}$ and $Y_i=X_i\times_{X_{i_0}} Y_{i_0}$ for $i\ge i_0$.
It follows that for all $i\ge i_0$ we have a cartesian diagram
\[\xymatrix{ Y\ar[r]\ar[d]^{\pi} & Y_i\ar[d]\\ X\ar[r] & X_i,}\]
and $\varprojlim_i Y_i \cong Y$. The horizontal maps in this diagram are flat by assumption.
Since $Y$ is regular, we find regular open neighborhoods $U_i\subset Y_i$ around the images of the projections $Y\to Y_i$,
see \cite[Corollaire (6.5.2)]{EGAIV2}. Since the $U_i$ are of finite type over the perfect ground field $k$, they are smooth, and we can
arrange them into a projective system $(U_i)_i$ with $Y=\varprojlim_i U_i$. If $E$ is an effective Cartier divisor on $Y$,
we clearly find
a large enough index $i_0$ and an effective Cartier divisor $E_{i_0}$ on $U_{i_0}$, such that $E= E_{i_0|Y}$.
Hence $(Y,E)=\varprojlim_{i\ge i_0} (U_i, E_{i_0|U_i})\in\uMCor^\pro$.
\end{proof}

The following proposition will be the main tool, which makes  the purely categorical 
defined tensor product on $\uMPST$ accessible via symbolic computations. 
Note that this is not completely formal since the category of regular and excellent at most 1-dimensional $k$-schemes
does not have a reasonable monoidal structure.
\begin{prop}\label{prop:ten-rep-1dim}
Let $\cX=(X,D)=\varprojlim_i \cX_i\in \uMCor^{\pro}$, where $\cX_i=(X_i, D_i)\in \uMCor$. Assume 
\begin{enumerate}
\item $\dim X\le 1$;
\item $X$ is excellent and connected;
\item the projection $X\to X_i$ is flat for all $i$.
\end{enumerate}
Let $F_1,\ldots,F_n\in \uMPST$, and for $\bullet\in \{\te, \tte\}$ extend 
$F_1\bullet\ldots\bullet F_n$ to $\uMCor^\pro$ as in \eqref{sec: MCorpro2}.
Then
\eq{prop:ten-rep-1dim0}{
(F_1\bullet\ldots\bullet F_n)(\cX)=
\left(\bigoplus_{(Y, E_1,\ldots, E_n)\in \Lambda(\cX)} F_1(Y, E_1)\te_\Z \ldots \te_\Z F_n(Y, E_n)\right)/ R,}
where 
\[\Lambda(\cX):=\left\{(Y\to X, E_1,\ldots, E_n)\,\middle\vert\, 
\begin{minipage}{0.48\textwidth}
$Y$ is regular,
$Y\to X$  is finite surjective, 
$E_i$ are effective Cartier divisors on $Y$,
 with
$D_{|Y}\ge E_1\bullet\ldots\bullet E_n$
\end{minipage}
\right\}
,\]
here we use the notation 
\[E_1\bullet \ldots \bullet E_n= \begin{cases} E_1+\ldots + E_n, & \text{if } \bullet=\te,\\ 
\max\{E_1,\ldots, E_n\},&\text{if } \bullet=\tte,\end{cases}\] 
and  $R$ is the subgroup generated by elements
\eq{prop:ten-rep-1dim1}{(a_1\te \ldots \te f_* b_i\te a_{i+1}\te \ldots \te a_n) 
- (f^*a_1\te \ldots \te b_i\te f^*a_{i+1}\te \ldots \te f^*a_n) ,}
where $(Y, E_1, \ldots,E_n)$, $(Y', E_1',\ldots, E_n')\in \Lambda(\cX)$,
$i\in \{1,\ldots, n\}$, $a_j\in F_j(Y, E_j)$, $b_i\in F_i(Y', E_i')$
$f: Y'\to Y$ is finite surjective,  $E_j'\ge f^* E_j$, for $j\neq i$, and $f^*E_i\ge E_i'$.
\end{prop}
\begin{proof} The proof is similar to the one of \cite[Proposition 5.1.3]{IR} (there for $\PST$).
We will give a detailed proof, since we have to take care of the moduli. 
For a morphism $g$ we also denote by $g$ the finite correspondence determined by the graph of $g$;
if $g$ is finite and surjective, we denote by $g^t$ the finite correspondence determined by the transpose of the graph of $g$.
Let $f:Y'\to Y$ be as in \eqref{prop:ten-rep-1dim1}. Then $f\in \ulMCor^{\pro}((Y', E_j'), (Y,E_j))$, for $j\neq i$,
and $f^t\in\ulMCor^{\pro}((Y,E_i), (Y', E_i'))$, see Lemma \ref{lem:fin-sur-pro}.
Hence $R$ is well-defined. 
Denote the right hand side of \eqref{prop:ten-rep-1dim0} by $T(\cX)$.
The rest of the proof consists of three steps.

\vspace{2mm}
\noindent
{\bf Step 1.}
We construct a map 
\begin{equation}\label{eq:bar-theta}
\bar{\theta}: (F_1\bullet\ldots\bullet F_n)(\cX)\to T(\cX).
\end{equation}
For $\ul{M}=(M_1, \ldots, M_n)\in\ulMCor^{\times n}$ set $\ul{M}^\bullet:=M_1\bullet\ldots\bullet M_n$ and define
\eq{prop:ten-rep-1dim3}{\theta_{\ul{M}}: 
F_1(M_1)\te_\Z\ldots \te_\Z F_n(M_n)\te \ulMCor^{\pro}(\cX, \ul{M}^\bullet )\to T(\cX)}
as follows:
Let $V\in\ulMCor^{\pro}(\cX, \ul{M}^\bullet)$ be a prime correspondence and denote by 
$\ol{V}\subset X\times \ol{\ul{M}^\bullet}$ its closure.
In particular, $V\to X^\o:=X\setminus D$ is finite surjective; hence $\dim V=\dim X\le 1$ 
and $\ol{V}\to X$ is also finite and surjective. Let $\wt{V}:=\ol{V}^N$ be the normalization of $\ol{V}$; it is a regular scheme.
Since $X$ is excellent the induced map $\wt{V}\to X$ is finite surjective, hence $(\wt{V},E)\in \ulMCor^{\pro}$, for any effective
Cartier divisor $E$ on $\wt{V}$, by Lemma \ref{lem:fin-sur-pro}.
Denote by $p_i: \wt{V}\to \ol{M}_i$  the maps induced by projections and
set $E_i:= p_i^*M_{i,\infty}$.
By Lemma \ref{lem:ten-cond} we have $D_{|\wt{V}}\ge E_1\bullet\ldots\bullet E_n$. Hence
$(\wt{V}\to X, E_1,\ldots, E_n)\in \Lambda(\cX)$. 
For $a_i\in F_i(M_i)$,  we define 
\[\theta_{\ul{M}}(a_1\otimes\ldots\te a_n\otimes V):= \text{class of } (p_1^*a_1\otimes\ldots\te p_n^*a_n) \text{ in }T(\cX).\]
We extend this additively to obtain the map \eqref{prop:ten-rep-1dim3}.
Define 
\[\theta:=\oplus \theta_{\ul{M}} : 
\bigoplus_{\ul{M}} F_1(M_1)\otimes_\Z\ldots\te_\Z F_n(M_n)\otimes\ulMCor^{\pro}(\cX, \ul{M}^\bullet)\to T(\cX).\]
We claim 
\eq{prop:ten-rep-1dim4} {\theta(\eqref{lem:tensorMPST1})=0.}
We show this in the case $i=1$ in  \eqref{lem:tensorMPST1}. (The proof for general $i$ is similar.)
Thus we have to show the following: Let $a_j\in F_j(M_j)$, 
$f\in \ulMCor(M',M_1)$, and $h\in \ulMCor^{\pro}(\cX, M'\bullet N)$, with $N:=M_2\bullet\ldots\bullet M_n$;  then 
\eq{{prop:ten-rep-1dim4}}
{\theta(a_1\otimes\ldots \te a_n\otimes (f\bullet\id_N)\circ h)= \theta(f^*a_1\otimes a_2\te \ldots\te a_n\otimes h).}
We may assume that $f=V\subset {M'}^{\o}\times M_1^\o$ and 
$h=W\subset X^\o\times {M'}^{\o}\times N^\o$ are prime correspondences.
We have 
\eq{prop:ten-rep-1dim6}{(f\bullet\id_N)\circ h=\sum_U m_U\cdot d_U\cdot \sigma(U),}
where the sum is over the irreducible components $U$ of $W\times_{{M'}^{\o}\times N^{\o}} V\times N^{\o}$, 
$\sigma$ is the natural map induced by projection
\[\sigma: W\times_{{M'}^{\o}\times N^{\o}} V\times N^{\o}= W\times_{{M'}^\o} V\to X^\o\times M_1^\o\times N^\o,\]
$d_U=[U:\sigma(U)]=$ degree of $U/\sigma(U)$, and 
\eq{prop:ten-rep-1dim6.5}{m_U=\sum_{j\ge 0} (-1)^j 
\text{length}({\rm Tor}_j^{\sO_{{M'}^\o\times N^\o, \eta_U}}(\sO_{W,\eta_U}, \sO_{V\times N^\o, \eta_U})),}
where by abuse of notation we denote by $\eta_U$ the images of the generic point of $U$ in the various schemes. Denote by
$\wt{\sigma(U)}$, $\wt{U}$, $\wt{W}$ the normalizations of the closures $\ol{\sigma(U)}$, $\ol{U}$, $\ol{W}$, respectively.
We obtain the following commutative diagram
\[\xymatrix{
      &   & \wt{\sigma(U)}\ar[rr]^{p_{1}}\ar[dr]\ar[drr]^{p_{j}}\ar@/_1pc/[dll]  
                                                         &  &\ol{M_1} \\
X & \wt{U}\ar[l]\ar[ur]^{\sigma_U}\ar[dr]^s &      & \ol{N}\ar[r] & \ol{M}_j, &  j\ge 2\\
            &          &  \wt{W}\ar@/^1pc/[ull]\ar[ur]\ar[urr]_{q_{j}}\ar[rr]_{q_{1}}       &   &\ol{M'},
}\]
in which all maps are induced projections. Set 
\[ E_{\sigma(U), j}:=p_j^*M_{j,\infty}, \quad \text{for }j\ge 1,\]
\[E_{W,1}:=q_1^* M'_\infty, \quad E_{W,j}:= q_j^*M_{j,\infty}, \quad \text{for  }j\ge 2,\]
and
\[E_{U,1}:=\sigma_U^*E_{\sigma(U),1}, \quad  E_{U,j}:=\sigma_U^*E_{\sigma(U),j}=s^*E_{W,j}, \quad
\text{for } j\ge 2.\]
Note that the tuples
$(\wt{\sigma(U)}\to X, E_{\sigma(U)})$, $(\wt{U}\to X, E_{U})$, 
$(\wt{W}\to X, E_{W})$ are in $\Lambda(\cX)$, where $E_{U}=(E_{U,1},\ldots, E_{U,n})$, etc.
Furthermore, on $\wt{U}$ we have
\eq{prop:ten-rep-1dim7}{s^*E_{W,1}\ge E_{U,1}. }
Indeed, the natural maps $\ol{U}\to \ol{M_1}$, and $\ol{U}\to \ol{M'}$ factor via the natural maps
$\ol{V}\to \ol{M_1}$ and $\ol{V}\to \ol{M'}$, respectively.
Since $M'_{\infty|\wt{V}}\ge M_{1,\infty|\wt{V}}$ by the modulus condition which $V$ satisfies, the analog inequality
also holds on $\wt{U}$, see e.g. \cite[Proposition 2.3(4)]{RY}.

In the following we will use the notation 
\eq{prop:ten-rep-1dim7.5}{[a]_i:=a_i\te\ldots\te a_n, \quad [p^*a]_i:= p_i^*a_i\te\ldots\te p_n^*a_n,\quad \text{etc.}}
We compute in $T(\cX)$:
\begin{align*}
\theta([a]_1\otimes(f\bullet\id_N)\circ h) & =
\sum_U m_U\cdot d_U\cdot  [p^*a]_1, 
& &\text{by defn},\\
 &= \sum_U m_U\cdot  (\sigma_{U*}\sigma^*_U p_1^*a_1)\te [p^*a]_2 \\
&=\sum_U m_U\cdot [\sigma^*_U p^*a]_1,
& & \text{by \eqref{prop:ten-rep-1dim1}},\\
&=\sum_U m_U\cdot  (\sigma^*_U p_1^*a_1)\te [s^* q^*a]_2\\
&=\left(\sum_U m_U\cdot  s_*\sigma^*_U p_1^*a_1\right)\otimes  [q^*a]_2,
& & \text{by \eqref{prop:ten-rep-1dim1}},
\end{align*}
for the last equality we  use \eqref{prop:ten-rep-1dim7} to apply \eqref{prop:ten-rep-1dim1}
in the situation 
\[(f, E, E', b_1, a)= 
(s, E_{W}, E_{U},  \sigma_U^* p_1^*a_1, q^*a ).\]
We have 
\[q_1\in \ulMCor^{\pro}((\wt{W}, E_{W,1}), M'), \quad p_1\circ \sigma_U\in \ulMCor^{\pro}((\wt{U}, E_{U,1}), M_1)\]
and by \eqref{prop:ten-rep-1dim7} also
\[s^t\in\ulMCor^{\pro}((\wt{W}, E_{W,1}), (\wt{U}, E_{U,1})).\]
By \eqref{prop:ten-rep-1dim6.5} and flat base change for Tor we have 
\[m_U=\sum_{j\ge 0} (-1)^j 
\text{length}({\rm Tor}_j^{\sO_{{M'}^\o, \eta_U}}(\sO_{W,\eta_U}, \sO_{V, \eta_U})).\]
Hence in $\ulMCor^{\pro}((\wt{W}, E_{W,1}), M_1)$
\eq{prop:ten-rep-1dim8}{
\sum_U m_U \cdot (p_1\circ\sigma_U)\circ s^t  = \sum_{U} m_U \cdot U' = f\circ q_1,
}
where $U'\subset (\wt{W}\setminus|E_{W,1}|)\times M_1^\o$ is the restriction  of 
the image of $\wt{U}\to \wt{W}\times\ol{M}_1$.
Together with the above we obtain
\begin{align*}
\theta([a]_1\otimes(f\bullet\id_N)\circ h)& = 
\left(\sum_U m_U\cdot   (p_1\circ \sigma_U\circ s^t)^*a_1\right)\otimes  [q^*a]_2\\
                                                     &= q_1^*f^*a_1\otimes  [q^*a]_2\\
                                                    & = \theta( f^*a_1\te [a]_2\te h).
\end{align*}
This shows that $\theta$ satisfies \eqref{prop:ten-rep-1dim4}.
Thus $\theta$ factors to give a well-defined map \eqref{eq:bar-theta}.

\vspace{2mm}
\noindent
{\bf Step 2.}
We construct a map 
\begin{equation}\label{eq:bar-psi}
\bar{\psi}: T(\cX)\to (F_1\bullet\ldots \bullet F_n)(\cX).
\end{equation}
Let $(\pi: Y\to X, E_1,\ldots, E_n)\in \Lambda(\cX)$ and $a_i\in F_i(Y, E_i)$.
By Lemma \ref{lem:fin-sur-pro} we find a smooth $k$-scheme $U$ with a $k$-morphism 
$Y\to U$, effective Cartier divisors $E_{U,i}$, and elements $\tilde{a}_i\in F_i(U, E_{U,i})$, 
such that $E_{U,i|Y}=E_i$, $\tilde{a}_{i|Y}=a_i$. Denote by $\Gamma\subset (Y\setminus|\pi^*D|)\times U^{\times n}$ the
graph of the diagonal map $Y\setminus|\pi^*D|\to U^{\times n}$. 
By definition of $\Lambda(\cX)$ and Lemma \ref{lem:ten-cond}
we have 
\[\Gamma\in \ulMCor^{\pro}((Y, \pi^*D), (U, E_{U,1})\bullet\ldots\bullet (U, E_{U,n})).\]
Using the description of $F_1\bullet\ldots\bullet F_n$ from Lemma \ref{lem:tensorMPST} 
(and the notation from \eqref{prop:ten-rep-1dim7.5}) we set 
\eq{prop:ten-rep-1dim8.5}{\psi_Y([a]_1):= 
\text{class of } ([\tilde{a}]_1\otimes \Gamma) \in (F_1\bullet\ldots\bullet F_n)(Y, \pi^*D)}
and 
\[\psi([a]_1):=\pi_*\psi_Y([a]_1)= \text{class of }([\tilde{a}]_1\otimes (\Gamma\circ \pi^t) ) 
\in (F_1\bullet\ldots \bullet F_n)(\cX).\]
We claim that $\psi_Y([a]_1)$ (and hence $\psi([a]_1)$) is well-defined, i.e.,
independent of the choices of $Y\to U$, $E_{U,i}$, and $\tilde{a}_i$.
Indeed, let $Y\to V$, $E_{V,i}$, and  $a'_i$ be different choices; it suffices to consider the case in which
we have a commutative diagram 
\[\xymatrix{ & V\ar[d]^f\\ Y\ar[ur]\ar[r] & U,}\]
and $E_{V,i}= f^*E_{U,i}$, $a'_i=f^*\tilde{a}_i$.
Set $\cU_i=(U, E_{U,i})$ and $\cV_i=(V, E_{V,i})$ and 
denote by $\Gamma_U$  and $\Gamma_V$ the graphs of
$Y\setminus |\pi^*D|\to U^{\times n}$ and $Y\setminus |\pi^*D|\to V^{\times n}$, respectively. 
We have in $\ulMCor^\pro((Y, \pi^*D), \cU_1\bullet \ldots\bullet \cU_n)$
\[\Gamma_U= \underbrace{f\bullet\ldots\bullet f}_{n\text{-times}}\circ \Gamma_V= 
(f\bullet\id_{\cV_2\bullet\ldots \bullet\cV_n})\circ\ldots \circ (\id_{\cV_1\bullet\ldots\bullet \cV_{n-1}}\bullet f)\circ \Gamma_V.\]
Thus $[\tilde{a}]_1\otimes \Gamma_U = [a']_1\otimes \Gamma_V$ in $(F_1\bullet\ldots\bullet F_n)(Y, \pi^*D)$,
by \eqref{lem:tensorMPST1}.
Hence $\psi([a]_1)$ is well-defined and we can extend it additively to obtain a map
\[\psi: \bigoplus_{(Y, E_1, \ldots,E_n)\in \Lambda(\cX)} F(Y, E_1)\te_\Z\ldots\te_\Z F_n(Y, E_n)\to 
(F_1\bullet\ldots\bullet F_n)(\cX).\]
We claim 
\eq{prop:ten-rep-1dim9}{\psi(\eqref{prop:ten-rep-1dim1})=0.}
We show this in the case $i=1$ in  \eqref{prop:ten-rep-1dim1}. (The proof for general $i$ is similar.)
To this end, let $\alpha:=(f, E, E', a, b_1)$ be as in \eqref{prop:ten-rep-1dim1}.
We find a finite surjective map  between smooth $k$-schemes $\tilde{f}: U'\to U$ which fits in a cartesian diagram
\begin{equation}\label{cd-yu}
\xymatrix{Y'\ar[r]^{u'}\ar[d]^{f}& U'\ar[d]^{\tilde{f}}\\ Y\ar[r]^{u} & U,}
\end{equation}
effective Cartier divisors $\wt{E}_{i}$ on $U$ and $\wt{E}'_i$ on $U'$, elements $\tilde{a}_j\in F_j(U, \wt{E}_j)$, 
$j\ge 2$, and $\tilde{b}_1\in F_1(U', \wt{E}'_1)$, such that $(\tilde{f}, \wt{E}, \wt{E}', \tilde{a}, \tilde{b}_1)$ 
restricts to $\alpha$; furthermore we can assume
\[\tilde{f}^*\wt{E}_1\ge \wt{E}'_1, \quad \wt{E}'_j\ge \tilde{f}^* \wt{E}_j,\quad j\ge 2.\]
Let $\Gamma$ and $\Gamma'$ be the graphs of $Y\setminus |D_{|Y}|\to U^{\times n}$ and 
$Y'\setminus |D_{|Y'}|\to (U')^{\times n}$, respectively.
For $j\ge 1$, set 
\[\cU_j:=(U, \wt{E}_j), \quad \cU'_j:=(U', \wt{E}'_j), \quad \cY:=(Y, D_{|Y}), \quad \cY':=(Y', D_{|Y'}).\]
We write $\cU_j^\o= U\setminus |\wt{E}_j|$, etc. By assumption we have 
\[\Gamma\in \ulMCor^{\pro}(\cY, \cU_1\bullet\ldots\bullet\cU_n), \quad 
\Gamma'\in\ulMCor^{\pro}(\cY', \cU_1'\bullet\ldots\bullet\cU_n')\]
\[f\in \ulMCor^\pro(\cY', \cY), \quad f^t\in \ulMCor^\pro(\cY, \cY')\]
\[\tilde{f}^t\in \ulMCor(\cU_1, \cU'_1), \quad \tilde{f}\in\ulMCor(\cU'_j, \cU_j),\quad j\ge 2.\]
Set
\[\cV:=\cU_2\bullet\ldots\bullet \cU_n,\quad \cV':=\cU'_2\bullet\ldots\bullet\cU'_n.\]
We obtain
\[\tilde{f}^{\bullet n-1}=\underbrace{\tilde{f}\bullet\ldots\bullet \tilde{f}}_{(n-1)\text{-times}}\in 
\ulMCor(\cV',\cV).\]
In $\ulMCor(\cY, \cU'_1\bullet \cV)$ the following equality holds
\eq{prop:ten-rep-1dim10}{(\id_{U'_1}\bullet \tilde{f}^{\bullet n-1})\circ\Gamma'\circ f^t= (\tilde{f}^t\bullet\id_{\cV})\circ\Gamma.}

Indeed, this is a purely cycle theoretic question; 
it suffices to show the equality in 
$\Cor(Y \setminus D|_Y, (\cU_1^\o) \times \cV^\o)$.
We may therefore suppose  $D|_Y, E, E'$ are all trivial.
Then \eqref{prop:ten-rep-1dim10} follows from the following identity of cycles on $Y$
\[f_* ((\id_{U'}\times\tilde{f}^{\times n-1})\circ \Gamma')^*[U'\times U^{\times n-1}]
= \Gamma^*(\tilde{f}\times \id_{U^{\times n-1}})_*[U'\times U^{\times n-1}], \]
which holds, e.g.,  by \cite[2.2. Lemma (4)]{Fu-Rat-Eq}.

We compute in $(F_1\bullet\ldots\bullet F_n)(\cY)$
\begin{align*}
f_*\psi_{Y'}(b_1\te [f^*a]_2)&= \tilde{b}_1\te[\tilde{f}^*\tilde{a}]_2\otimes (\Gamma'\circ f^t), 
& & \text{by defn},\\
                     &=\tilde{b}_1\te [\tilde{a}]_2\otimes ((\id_{\cU_1'}\bullet \tilde{f}^{\bullet n-1})\circ \Gamma'\circ f^t), 
& & \text{by \eqref{lem:tensorMPST1}}\\
                                    & = \tilde{b}_1\te [\tilde{a}]_2\otimes ((\tilde{f}^t\bullet\id_{\cV})\circ\Gamma),
& & \text{by \eqref{prop:ten-rep-1dim10}}\\
                                   & = (\tilde{f}^t)^*\tilde{b}_1\otimes [\tilde{a}]_2\otimes \Gamma,
& &\text{by \eqref{lem:tensorMPST1}}\\
                                  & = \tilde{f}_*\tilde{b}_1\otimes [\tilde{a}]_2\otimes \Gamma\\
                                  &= \psi_Y(f_*b_1\otimes [a]_2), 
& &\text{by defn.}
\end{align*}
Pushing forward to $\cX$ we obtain  $\psi(b_1\te [f^*a]_2)=\psi(f_*b_1\otimes [a]_2)$,

Thus $\psi$ factors to give a well-defined map \eqref{eq:bar-psi}.

\vspace{2mm}
\noindent
{\bf Step 3.}
We show the two maps $\bar{\theta}$ and $\bar{\psi}$ from \eqref{eq:bar-theta} and \eqref{eq:bar-psi}
are inverse to each other.

By definition we have 
\[\bar{\theta}\circ\bar{\psi}=\id_{T(\cX)}.\]
Thus it remains to show that $\bar{\psi}$ is surjective.
To this end, let $M_1,\ldots, M_n\in \ulMCor$, $a_i\in F_i(M_i)$, and  let $V\in \ulMCor^\pro(\cX, M_1\bullet\ldots\bullet M_n)$
be a prime correspondence. Denote by $\wt{V}$ the normalization of $\ol{V}\subset X\times \ol{M_1\bullet\ldots\bullet M_n}$
and denote by $p_i: \wt{V}\to \ol{M_i}$ the maps induced by the projection maps.
By Lemma \ref{lem:fin-sur-pro}, we find a morphism $\wt{V}\to U$ with $U$ a smooth $k$-scheme such that
$p_i$ factors via $\tilde{p}_i:U\to \ol{M_i}$.
Set $E_i:=p_i^* M_{i,\infty}$ and  $E_{U,i}:=\tilde{p}_i^*M_{i,\infty}$.
Then $(\pi:\tilde{V}\to X, E_1, \ldots, E_n)\in \Lambda(\cX)$.
Let $\Gamma$ be the graph of $\wt{V}\setminus |\pi^*D|\to U^{\times n}$. We compute in 
$(F_1\bullet\ldots\bullet F_n)(\cX)$
\begin{align*}
\psi([p^*a]_1)&= [\tilde{p}^*a]_1 \otimes \Gamma\circ \pi^t, 
& &\text{by defn},\\
                               &= [a]_1\te (\tilde{p}_1\bullet\ldots\bullet \tilde{p}_n)\circ \Gamma\circ\pi^t, 
& &\text{by \eqref{lem:tensorMPST1},}\\
                              &= [a]_1\otimes V.
\end{align*}
This shows that $\psi$ is surjective and completes the proof.
\end{proof}

\begin{rmk}\label{rmk:MF}
Taking $\cX=(X, \emptyset)$ in Proposition \ref{prop:ten-rep-1dim}  we see
that 
\[(F_1\bullet\ldots \bullet F_n) (X,\emptyset)= (\ul{\omega}_!F_1\ten{\PST}\ldots \ten{\PST}\ul{\omega}_!F_n)(X)\]
equals the tensor product of Mackey functors evaluated at $X$, see, e.g., \cite[2.8]{KY}, or  \cite[4.1]{IR}.
In particular, the above proposition is a generalization of \cite[Proposition 5.1.3]{IR}.
\end{rmk}

\begin{para}\label{subsec:not}
We introduce some notations: 
Let $F \in \PST$ (resp. $F \in \MPST$).
If $K$ is a field containing $k$, then 
we write (see \ref{subsec: MCorpro})
\[F(K):=F(\Spec K) \quad (\text{resp.}~ F(K):=\omega_!F(\Spec K)). \]
If $a\in F(K)$ and $L/K$ is a field extension inducing $f:\Spec L\to \Spec K$, 
then we write 
\[a_L:=f^*a\in F(L),\] and if $f$ is finite and $b\in F(L)$, then we write
\[\Tr_{L/K}(b):=f_*b:=(f^t)^*b\in F(K).\]
If $C/K$ is a regular projective curve, 
$D$ an effective divisor on $C$,
$a\in F(C \setminus D)$ 
(resp. $a \in F(C, D)$)
and $i:x\inj C \setminus D$ is a closed point,
then we write 
\[a(x)= i^*a\in F(x)=F(K(x)).\] 
\end{para}

\begin{thm}\label{thm:CI-te-fields}
Let $F_1,\ldots, F_n\in\MPST$. Let $K$ be a field containing $k$. Let $\bullet\in \{\te, \tte\}$, 
then (see \eqref{subsec:omegaCI1} and  Notation \ref{subsec:not})
\[
h_0(F_1\bullet\cdots\bullet F_n)(K)
=\left( \bigoplus_{L/K} F_1(L)\otimes_{\Z}\ldots\otimes_\Z F_n(L)\right)/R(K),\]
where the sum is over all finite field extensions $L/K$ (one may restrict to those inside a fixed algebraic closure of $K$)
and $R(K)$ is the subgroup generated by the following elements:
\begin{enumerate}[label=(R\arabic*)]
\item\label{R1}  
$(a_1\otimes \cdots \otimes\Tr_{L'/L}(a_{i})\otimes\cdots \otimes a_n) - 
                   (a_{1,L'}\otimes\cdots \otimes a_i\otimes\cdots \otimes a_{n,L'}),$
where 
$L'/L/K$ is a tower of finite field extensions,
$i\in [1,n]$, $a_j\in F_j(L)$, for $j\neq i$, $a_i\in F_i(L')$;
\item\label{R2} 
\[\sum_{x\in C\setminus |D|} v_x(f)\cdot a_1(x)\otimes\cdots\otimes a_n(x),\]
where $C$ is a regular projective curve over $K$, $D=D_1\bullet\cdots \bullet D_n$, with 
$D_i$ effective divisors on $C$, such that
$a_i\in (\tau_!F_i)(C, D_i)$, 
and $f\in K(C)^\times$ satisfies 
$f\equiv 1$ mod  $D$,
and $v_x$ denotes the normalized valuation of $K(C)$ attached to $x$.
Here we use the notation
\[D_1\bullet\cdots\bullet D_n=\begin{cases} 
 D_1+\ldots+D_n, & \text{if } \bullet=\te,\\
\max\{D_1,\ldots, D_n\}, &\text{if }\bullet=\tte,
\end{cases}\]
and for a divisor $D=\sum_j n_j x_j$ on $C$ the notation $f\equiv 1$ mod $D$ means
$v_{x_j}(f-1)\ge n_j$, for all $j$.
\end{enumerate}
\end{thm}
\begin{proof}
Let $M\in \MCor$ and $G\in \MPST$. With the notations from \ref{sec:tte-MPST}, \ref{subsec:CI}, and \ref{subsec:omegaCI} 
we have
\eq{thm:CI-te-fields1}{
h_0^{\bcube}G(M)\cong \Coker(G(M\te(\P^1, 1))\xr{i_0^*-i_\infty^*} G(M)).}
(Here we use the canonical isomorphism $(\P^1, 1)\cong \bcube$,
which is induced by the unique automorphism of $\P^1$ which switches $1$ and $\infty$ and fixes $0$.)
Now let $G=  F_1\bu{}\ldots \bu{} F_n$.
Combining \eqref{thm:CI-te-fields1} with the fact that $\omega_!=\uomega_!\tau_!$ and that $\tau_!$ is exact and monoidal for
$\bullet$ (see \eqref{eq:tau-sh}), we obtain for $X\in \Sm$
\[h_0(F_1\bullet\ldots\bullet F_n)(X)=
\frac{(\tau_!F_1\bu{\uMPST}\ldots\bu{\uMPST}\tau_!i F_n)(X,\emptyset)}
{(i_0^*-i_\infty^*) (\tau_!F_1\bu{\uMPST}\ldots\bu{\uMPST}\tau_!F_n)(\P^1_X, 1_X)}.\]
We can write $K=\varinjlim_i A_i$ (filtered direct limit) 
with $A_i$ smooth $k$-algebras such that the induced maps $A_i\to K$ are flat.
Hence, taking $X=\Spec A_i$ in the above formula and taking the direct limit yields
\[
h_0(F_1\bu{}\ldots\bu{} F_n)(K)=
\frac{(\tau_!F_1\bu{\uMPST}\ldots\bu{\uMPST}\tau_! F_n)(K)}
{(i_0^*-i_\infty^*) (\tau_! F_1\bu{\uMPST}\ldots\bu{\uMPST}\tau_! F_n)(\P^1_K, 1_K)}.\]
Now the statement follows from Proposition \ref{prop:ten-rep-1dim} and  the observation that $f\in K(C)^\times$ 
in \ref{R2} is the same as a finite surjective map $f: C\to \P^1_K$ such that $f^*(1)\ge D_1\bullet\ldots\bullet D_n$.
\end{proof}

\begin{para}\label{subsec:not-cond}
Denote 
\[\Phi=\{\text{henselian discrete valuation fields of geometric type over }k\},\]
i.e., $\Phi$  is the set of henselian discrete valuation fields of the form
$\Frac(\cO_{U,x}^h)$, where $U\in \Sm$, $x\in U^{(1)}$, and $\cO_{U,x}^h$ denotes the henselization of
the local ring $\cO_{U,x}$. For $L\in \Phi$ we denote by $\cO_L$ its underlying DVR, and by $\fm_L\subset \cO_L$ the 
maximal ideal. Set $S=\Spec \cO_L$ and denote by $s\in S$ the closed point.
For $G\in\ulMPST$ and $n\ge 0$ we set 
\[G(\cO_L, \fm_L^{-n}):=\begin{cases} G(S,\emptyset), & \text{if } n=0,\\ G(S, n\cdot s),&\text{else.}\end{cases}\]
\end{para}

\begin{cor}\label{cor:mod-genRSC}
Let $F_1,\ldots, F_n\in \MPST$ and set $F:=F_1\bullet\ldots \bullet F_n\in \MPST$, with $\bullet\in \{\te, \tte\}$.
Let $L\in \Phi$, then $h_0(F)(L)$
is generated by all the elements of the form 
\[\pi_*([a_1,\ldots, a_n]_{L'}),\]
where $L'/L$ is a finite field extension, 
$[a_1,\ldots, a_n]_{L'}$ denotes the image of $a_1\otimes\ldots\otimes a_n\in F_1(L')\te_\Z\ldots\te_\Z F_n(L')$
under the natural map
\[F_1(L')\te_\Z\ldots\te_\Z F_n(L')\to (F_1\ten{\PST}\ldots\ten{\PST} F_n)(L')
\xr{\eqref{subsec:omegaCI2}}F(L'),\]
and $\pi:\Spec L'\to \Spec L$ is the finite morphism.
Furthermore, if  $a_i\in \Im(\tau_!F_i(\cO_{L'},\fm_{L'}^{-r_i})\to F_i(L'))$, where the $r_i\ge 0$,
then,
 \eq{cor:mod-genRSC1}{\pi_*([a_1,\ldots, a_n]_{L'})\in \Im(\tau_!h^{\bcube}_0(F)(\cO_L,\fm_L^{-n})\to h_0(F)(L)),}
where
\[n=\begin{cases}
 \lceil\frac{r_1+\ldots+r_n}{e(L'/L)}\rceil, & \text{if }\bullet=\te,\\
 \lceil\frac{\max\{r_1,\ldots,r_n\}}{e(L'/L)}\rceil, & \text{if }\bullet=\tte,
\end{cases}\]
with $e(L'/L)$ the ramification index of $L'/L$ and the map in \eqref{cor:mod-genRSC1}
is induced  from  $h_0(F)(L)= \tau_!h^{\bcube}_0(F)(\Spec L, \emptyset)$.
\end{cor}
\begin{proof}
The first statement holds by
Theorem \ref{thm:CI-te-fields}, the second by Proposition \ref{prop:ten-rep-1dim}.
\end{proof}

\begin{para}\label{subsec:cond}
Let $F\in\RSC_{\Nis}$. 
Let  $a\in F(L)$, $L\in \Phi$, (see \ref{subsec:not-cond}).
Recall from \cite[Definition 4.14]{RS} that the {\em motivic conductor}  $c^F_L(a)$  of $a$ at $L$ is defined by:
\[c^F_L(a):=\min\{n\ge 0\mid a\in \tau_!\tF(\cO_L,\fm_L^{-n})\}.\]
\end{para}

\begin{cor}\label{cor:cond-t}
Let $F_1,\ldots, F_n\in \RSC_{\Nis}$.
Set $F:=h_{0,\Nis}(\tF_1\bullet\ldots\bullet\tF_n)$, with $\bullet\in\{\te, \tte\}$ (see \eqref{subsec:omegaCI4}).
Let $L\in \Phi$ and $a_i\in F_i(L)$. Then (with the notation from Corollary \ref{cor:mod-genRSC})
\[c^F_L([a_1,\ldots, a_n]_L)\le \begin{cases}
c^{F_1}_L(a_1)+\ldots+c^{F_n}_L(a_n), & \text{if } \bullet=\te,\\
\max\{c^{F_1}_L(a_1),\ldots, c^{F_n}_L(a_n)\},& \text{if } \bullet=\tte.
\end{cases}\]
\end{cor}
\begin{proof}
By the adjointness of $(\omega_{\CI}, \omega^{\CI})$ we have a natural transformation
$h_0^{\bcube}\to \omega^{\CI}h_0$ of functors $\MPST\to \CI$.  Thus the
 statement follows from  Corollary \ref{cor:mod-genRSC} and the natural map
\[\tau_!\omega^{\CI}(h_0(\tF_1\bullet\ldots\bullet \tF_n))(\cO_L,\fm^{-r}_L)
\to\tau_!\omega^{\CI}(F)(\cO_L,\fm^{-r}_L), \quad r\ge 0.\]
\end{proof}

\begin{defn}[cf. {\cite[4.2.1]{IR}}]\label{defn:Lin}
Let $F_1,\ldots, F_n, H\in \RSC_{\Nis}$. Denote by 
\[\Lin_{\RSC}^{\bullet}(F_1,\ldots, F_n;H),\quad \bullet\in \{\te, \tte\},\]
the set of collections of maps 
\[\phi=
\bigg\{\phi_{X_1,\ldots, X_n}: F_1(X_1)\times \ldots\times F_n(X_n)\to 
H(X_1\times\ldots\times X_n )\bigg\}_{X_1,\ldots,  X_n\in \Sm},\] 
satisfying the following properties:
\begin{enumerate}[label = (L\arabic*)]
\item\label{L1} $\phi_{X_1,\ldots, X_n}$ is a multilinear morphism of $\Z$-modules, for all $X_i\in \Sm$.
\item\label{L2} For all $i\in [1,n]$, $X_1,\ldots, X_n$, $X_i'\in \Sm$,  $f\in \Cor(X_i', X_i)$,
                        and all $a_j\in F_j(X_j)$, $j=1,\ldots, n$,  we have 
\[\phi_{X_1,\ldots X_i', X_{i+1},\ldots, X_n}(a_1,\ldots,f^*a_i, a_{i+1},\ldots, a_n)=
 f_i^*\phi_{X_1,\ldots, X_n}(a_1,\ldots, a_n),\]
where $f_i=\id_{X_1\times\ldots\times X_{i-1}}\times f\times \id_{X_{i+1}\times\ldots\times X_n}$.
\item\label{L3} For all $L\in \Phi$ and all $a_i\in F_i(L)$ we have (see \ref{subsec:cond} for notation)
\[c^{H}_L(\phi_L(a_1,\ldots, a_n))\le \begin{cases}
c_L^{F_1}(a_1)+\ldots+ c_L^{F_n}(a_n),& \text{if } \bullet=\te,\\
\max\{c_L^{F_1}(a_1),\ldots, c_L^{F_n}(a_n)\}, &\text{if } \bullet=\tte.
\end{cases}\]
Here $\phi_L=\phi_{\Spec L}$, where $\phi_S$, for $S\in \Cor^{\pro}$,  is defined as the composition
\ml{defn:Lin1}{
\phi_S: F_1(S)\times\ldots \times F_n(S) = \varinjlim_U (F_1(U)\times \ldots F_n(U))\\
                                      \xr{\phi} \varinjlim_U H(U\times\ldots\times U)\xr{\Delta_U^*} \varinjlim_U H(U)= H(S),}
where $U$ runs through all smooth $k$-models of $S$,  $\Delta_U: U\to U\times\ldots\times U$ denotes the diagonal,
and we use \ref{L2} to extend $\phi$ to the colimit.
\end{enumerate} 
\end{defn}

\begin{lem}\label{lem:L3}
Let $\phi=\{\phi_{X_1,\ldots, X_n}\}$ be a collection of maps as in Definition \ref{defn:Lin}.
Assume $\phi$ satisfies \ref{L1}, \ref{L2}.
Then $\phi$ satisfies \ref{L3} if and only if for all $a_i\in\tF_i(\cX_i)$, with proper modulus pairs $\cX_i$ we have 
\[\phi_{\cX_1^\o,\ldots,\cX_n^\o}(a_1,\ldots, a_n)\in \tH(\cX_1\bullet\ldots\bullet \cX_n).\]
\end{lem}
\begin{proof}
Set $\alpha:=\phi_{\cX_1^\o,\ldots,\cX_n^\o}(a_1,\ldots, a_n)$,
and $\cX_1\bullet\ldots\bullet \cX_n=:(\ol{X}, X_\infty)$.
By \cite[Theorem 4.15, (4)]{RS} we have  
\[\alpha\in \tH(\cX_1\bullet\ldots\bullet \cX_n)\Longleftrightarrow c_{L}^{H}(\rho^*\alpha)\le v_L(X_\infty),\quad
 \forall L\in \Phi,\, \rho \in \ol{X}(\sO_L),   \]
where on the right hand side $v_L(X_\infty)$ is the multiplicity of $\rho^*X_\infty$. 
Observe  that $\rho^*\alpha= \phi_L(\rho_1^*a_1,\ldots, \rho_n^*a_n)$, where
$\rho_i:\Spec L\to \cX_i^\o$ are induced by $\rho$ and  projection. 
Thus the {\em only if} direction follows directly from Lemma \ref{lem:ten-cond}.
For the other direction let $b_i\in F_i(L)$, $L\in \Phi$, and set $r_i:=c_L^{F_i}(b_i)$.
We find smooth models $(U,Z)$ of $(\Spec \sO_L, \fm_L)$ with compactification
$\cU=(\ol{U}, \ol{Z}+B)$ and elements $\tilde{b}_i\in F_i(\cU_{r_i})$ restricting to $b_i$,
where $\cU_{r_i}=(\ol{U}, r_i \ol{Z}+B)$. Note that the map $\Spec \sO_L\to U$ 
(coming from the structure as model of $\sO_L$)
maps the closed point to the generic point of $\ol{Z}$ and hence is not contained in $B$.
By assumption $\phi_{\cU_1^\o,\ldots, \cU_n^\o}(\tilde{b}_1,\ldots\tilde{b}_n)\in \tH(\cU_1\bullet\ldots\bullet\cU_n)$.
Denote by $\rho: \Spec \sO_L\to \ol{\cU_1\bullet\ldots\bullet\cU_n}$ the diagonal map.
Thus (by Lemma \ref{lem:ten-cond})
\mlnl{c_L^H(\phi_L(b_1,\ldots, b_n))=c_L^H(\rho^*\phi_{\cU_1^\o,\ldots, \cU_n^\o}(\tilde{b}_1,\ldots\tilde{b}_n))
\le  v_L(\cU_{r_1}\bullet\ldots\bullet \cU_{r_n})=\\
\begin{cases}
v_L(\cU_{r_1})+\ldots+ v_L(\cU_{r_n})= r_1+\ldots r_n, &\text{if }\bullet=\te,\\
\max\{v_L(\cU_{r_1}),\ldots,v_L(\cU_{r_n})\}=\max\{r_1,\ldots, r_n\}, & \text{if }\bullet=\tte.
\end{cases}}
Hence $\phi$ satisfies \ref{L3}.
\end{proof}

\begin{lem}\label{lem:ffc}
Let $\varphi: F\to G$ be a morphism in $\RSC_{\Nis}$. Then $\varphi$ is an isomorphism if and only if
$\varphi(K): F(K)\to G(K)$ is an isomorphism, for all finitely generated $k$-fields $K$.
\end{lem}
\begin{proof}
By Theorem \ref{thm:sheafification}, 
the category $\RSC_{\Nis}$ is abelian. Hence it suffices to show that a sheaf  
$F\in\RSC_{\Nis}$ is zero if and only if $F(K)=0$, for all $K$.
This follows from \cite[Theorem 6]{KSY} and \cite[Corollary 3.2.3]{KSY3}.
\end{proof}

\begin{thm}\label{thm:ten-rep}
Let $F_1,\ldots, F_n, H\in \RSC_{\Nis}$.
Then there is a natural isomorphism
\[\Hom_{{\RSC}}(h_0(\tF_1\bullet\ldots\bullet\tF_n), H)= \Lin^{\bullet}_{\RSC}(F_1,\ldots, F_n;H), \quad \bullet\in \{\te,\tte\}.\]
\end{thm}
\begin{proof}
Denote by $\Lin_{\PST}(F_1,\ldots, F_n;H)$ the set of those $\phi=\{\phi_{X_1,\ldots, X_n}\}$ 
satisfying only \ref{L1} and \ref{L2}. We have 
\eq{thm:ten-rep1}{\Lin_{\RSC}^\bullet(F_1,\ldots, F_n; H)\subset \Lin_{\PST}(F_1,\ldots, F_n; H).}
Set 
\[F:=h_0(\tF_1\bullet\ldots\bullet \tF_n).\]
The surjection \eqref{subsec:omegaCI2} induces an inclusion
\eq{thm:ten-rep2}{\Hom_{\RSC}(F, H)\subset \Hom_{\PST}(F_1\ten{\PST}\ldots\ten{\PST}F_n, H).}
Since $H$ is a Nisnevich sheaf  we can replace $F$ by $F_{\Nis}$ and $\RSC$ by $\RSC_{\Nis}$
on the left hand side and  on the right hand side $\PST$ by $\NST$.
By \cite[Lemma 2.1]{SuVo00b} we have a canonical identification
\eq{thm:ten-rep3}{\Lin_{\PST}(F_1,\ldots, F_n;H)= \Hom_{\PST}(F_1\ten{\PST}\ldots\ten{\PST}F_n, H).}
Observe that if $\varphi: H_1\to H_2$ is a morphism in $\RSC_\Nis$ and $a\in H_1(L)$, $L\in \Phi$,
then $c_L^{H_2}(\varphi(a))\le c_L^{H_1}(a)$, by definition of the motivic conductor, see \ref{subsec:cond}.
Hence it follows  directly from Corollary \ref{cor:cond-t} that \eqref{thm:ten-rep3}
induces an inclusion 
\[ \Hom_{{\RSC}}(F, H)\subset \Lin^{\bullet}_{\RSC}(F_1,\ldots, F_n;H).\]
For the other inclusion, let $\phi\in \Lin^\bullet_{\RSC}(F_1,\ldots, F_n;H)$ and 
denote  by 
\[\hat{\phi}: F_1\ten{\NST}\ldots\ten{\NST} F_n\to H\]
the induced map. Set
\[G:=\Ker(F_1\ten{\NST}\ldots\ten{\NST} F_n\xr{\eqref{subsec:omegaCI2}_\Nis} F_{\Nis}).\]
It remains to show 
\eq{thm:ten-rep4}{\hat{\phi}(G)_{\Nis}=0.}
By definition $\hat{\phi}(G)_{\Nis}$ is a sub-$\NST$ of $H$, hence it is in $\RSC_{\Nis}$.
By Lemma \ref{lem:ffc} it suffices to show $0=\hat{\phi}(G)_{\Nis}(K)=\hat{\phi}(G(K))$, for all function fields $K$ over $k$.
By Theorem \ref{thm:CI-te-fields},  the group $G(K)$ is generated by elements of the form 
\[\alpha:=\sum_{x\in C\setminus |D|} v_x(f) \cdot a_1(x)\otimes\ldots\otimes a_n(x)\]
with $C/K$, $D=D_1\bullet\ldots\bullet D_n$, $a_i\in \tau_!\tF_i(C,D_i)$ and $f\equiv 1$ mod $D$ as in 
\ref{thm:CI-te-fields}\ref{R2}.
We compute
\begin{align*} 
\hat{\phi}(\alpha) & =\sum_{x\in C\setminus|D|}v_x(f) \Tr_{K(x)/K}\big(\phi_{K(x)}(a_1(x),\ldots, a_n(x))\big), 
&& \\
                        & =\sum_{x\in C\setminus|D|}v_x(f) \Tr_{K(x)/K}\big(\phi_{C\setminus|D|}(a_1,\ldots, a_n)(x)\big), 
&& \text{by }\ref{L2},\\
                       & =\Div_C(f)^* \phi_{C\setminus|D|}(a_1,\ldots, a_n),
\end{align*}
where  we view $\Div_C(f)\in \Cor^{\rm pro}(\Spec K, C\setminus|D|)$.
For the first equality, note, if $K'/K$ is a finite field extension, then the composition 
\mlnl{F_1(K')\otimes_\Z\ldots\otimes_\Z F_n(K')\xr{\psi_{K'}} (F_1\ten{\NST}\ldots\ten{\NST}F_n)(K')\\
\xr{\Tr_{K'/K}} (F_1\ten{\NST}\ldots\ten{\NST}F_n)(K)\xr{\hat{\phi}} H(K)}
maps an element $b_1\otimes\ldots \otimes b_n$ to $\Tr_{K'/K}(\phi_{K'}(b_1,\ldots, b_n))$,
with $\phi_{K'}=\phi_{\Spec K'}$ defined as in \eqref{defn:Lin1} and $\psi_{K'}$ defined as in \eqref{prop:ten-rep-1dim8.5}.
We have
\eq{thm:ten-rep5}{\phi_{C\setminus |D|}(a_1,\ldots, a_n)\in \tau_!\wt{H}(C,D).}
Indeed, this follows directly from Lemma \ref{lem:L3} and Lemma \ref{lem:ten-cond}.
Thus
\[\hat{\phi}(\alpha)= (i_0^*-i_\infty^*)(f^t)^*\phi_{C\setminus |D|}(a_1,\ldots, a_n)=0,\]  
where $f^t\in \ulMCor^{\pro}((\P^1_K, 1), (C,D))$ denotes the transpose of the graph of $f:C\to \P^1_K$.
This proves  \eqref{thm:ten-rep4}. 
\end{proof}

\begin{cor}\label{cor:assoc}
Let $F_1, \ldots, F_n\in \RSC_{\Nis}$ and $\bullet\in\{\te, \tte\}$. 
Let $0=m_0< m_1<m_2<\ldots< m_r=n$,  $r\ge 1$.
For $j=1,\ldots, r$ set
\[\ul{\tF}_j:=\tF_{m_{j-1}+1}\bullet\ldots\bullet \tF_{m_j}\in\MPST,\quad \text{and}\quad
H_j:= \wt{h_{0,\Nis}(\ul{\tF}_j)}\in \CI.\]
Denote by 
\[\tau_{j}\in \Lin_{\RSC}^\bullet(F_{m_{j-1}+1},\ldots,F_{m_j};\; h_{0,\Nis}(\ul{\tF}_j))\]
and
\[\tau_{1,\ldots, r}\in 
\Lin_{\RSC}^\bullet\Big(h_{0,\Nis}(\ul{\tF}_1),\ldots, h_{0,\Nis}(\ul{\tF}_r)\,;\, 
h_{0,\Nis}(H_1\bullet\ldots\bullet H_r)\Big)\]
 the maps corresponding via Theorem \ref{thm:ten-rep} to the identity
on $h_{0,\Nis}(\ul{\tF}_j)$ and on 
$h_{0,\Nis}(H_1\bullet\ldots\bullet H_r)$, respectively.
Then the composition 
\[\tau:=\tau_{1,\ldots,r}\circ(\tau_{1}\times\ldots\times \tau_r): 
F_1\times \ldots \times F_n\to h_{0,\Nis}(H_1\bullet\ldots\bullet H_r)\]
lies in $\Lin_{\RSC}^{\bullet}(F_1,\ldots,F_n; h_{0,\Nis}(H_1,\ldots, H_r))$
and the induced map 
\eq{cor:assoc1}{h_{0,\Nis}(\tF_1\bullet\ldots\bullet \tF_n)\to
h_{0,\Nis}(H_1\bullet\ldots\bullet H_r)}
is compatible with the natural surjection of $F_1\ten{\NST}\ldots\ten{\NST} F_n$ to either side (see \eqref{subsec:omegaCI2});
in particular it is surjective.
\end{cor}
\begin{proof}
It is direct to see that $\tau$ satisfies \ref{L1} and \ref{L2} of Definition \ref{defn:Lin}.
Hence by \eqref{thm:ten-rep3} it factors via the natural surjection 
\[F_1\ten{\NST}\ldots \ten{\NST}F_n\surj h_{0,\Nis}(H_1\bullet\ldots\bullet H_r)=:G.\]
Thus for $L\in \Phi$ and $a_i\in F_i(L)$ we have 
\[\tau_L(a_1,\ldots,a_n)=[ [\ul{a}_1]_L,\ldots, [\ul{a}_r]_L ]_L\in G(L), \]
where $\tau_L$ is defined as in \eqref{defn:Lin1}, $\ul{a}_j=(a_{m_{j-1}+1},\ldots, a_{m_j})$, and 
the bracket notation on the right is as in Corollary \ref{cor:mod-genRSC}.
Therefore, Corollary \ref{cor:cond-t} yields in the case $\bullet=\te$
\[c_L^G(\tau_L(a_1, \ldots, a_n))\le \sum_{j=1}^{r}c_L^{h_{0,\Nis}(\ul{F}_j)}([\ul{a}_j]_L) 
\le \sum_{i=1}^n c^{F_i}_L(a_i),\]
and in the case $\bullet=\tte$
\[c_L^G(\tau_L(a_1, \ldots, a_n))\le \max_{j}\{ c_L^{h_{0,\Nis}(\ul{F}_j)}([\ul{a}_j]_L)\} 
\le \max_i \{c^{F_i}_L(a_i)\},\] 
i.e., $\tau$ satisfies \ref{L3}; hence the statement.
\end{proof}

\begin{remark}
Corollary \ref{cor:assoc}, shows that 
\[\RSC_\Nis^{\times n}\to \RSC_\Nis,\quad (F_1, \ldots, F_n)\mapsto h_{0,\Nis}(F_1\bullet\ldots\bullet F_n),\]
 $\bullet\in\{\te,\tte\}$,
is a lax monoidal structure on $\RSC_\Nis$, in the sense that there is only a weak form
of associativity. See also Corollary \ref{cor:assocII} below.
\end{remark}

\begin{lem}\label{lem:RSC-mod-twist}
Let $H\in \RSC_{\Nis}$ and $\cY\in\ulMCor$. Set $Y:=\cY^\o$.
Let $H_{\cY}$ be the presheaf on $\Sm$ defined by 
(see \eqref{subsec:omegaCI3}, \eqref{subsec:omegaCI4})
\[
H_{\cY}(X):= \tau_!\tH((X,\emptyset)\otimes \cY) \subset H(X \times Y).
\]
Then $H_{\cY}\in\RSC_{\Nis}$. 
Moreover, 
given $a\in H_{\cY}(X)$ and $\bullet\in\{\te,\tte\}$,
if $M$ and $N\in\MCor$ are modulus compactifications of 
$X$ and  $Y$, respectively, 
such that $a$ has modulus $M\bullet N$ as an element of $H(X\times Y)$,
then $M$ is a modulus for $a$ as an element of $H_{\cY}(X)$.
\end{lem}
\begin{proof}
For any $X'\in \Sm$ the inclusion 
\mlnl{\Cor(X', X)=\ulMCor((X',\emptyset), (X,\emptyset))\\
\inj \ulMCor((X',\emptyset)\te \cY, (X,\emptyset)\te \cY),}
given by $f\mapsto f\times\id_{Y}$,
endows $H_{\cY}$ with the structure of a presheaf with transfers, which is also a Nisnevich sheaf.
Let $a\in H_{\cY}(X)$ and assume there exist modulus compactifications $M$ and $N$
of $X$ and $Y$, respectively, such that $a\in H(X\times Y)$ has modulus $M\bullet N$,
 i.e., the Yoneda map $\Z_\tr(X\times Y)\to H$ induced by $a$ factors via $h_0(M\bullet N)$ (see \ref{subsec:RSC}).
The exterior product with $\id_Y$ also induces an inclusion 
\eq{lem:RSC-mod-twist1}{\ulMCor((X', \emptyset)\otimes \bcube,M)\inj 
\ulMCor((X',\emptyset)\otimes (Y,\emptyset) \otimes\bcube, M\bullet N).}
Indeed, for $\bullet=\te$ observe that if $Z\subset X'\times \P^1\setminus\{\infty\}\times X$,
then the closure of $Z\times \id_Y$ in $X'\times Y \times \P^1\times \ol{M}\times \ol{N}$
is contained in $X'\times Y \times \P^1\times \ol{M}\times Y$; 
for $\bullet =\tte$  it follows from the inclusion 
\mlnl{\ulMCor((X',\emptyset)\otimes (Y,\emptyset) \otimes\bcube, M\te N)\\
\inj \ulMCor((X',\emptyset)\otimes (Y,\emptyset)\otimes\bcube, M\tte N)}
induced by the natural map $M\te N\to M\tte N$, see \eqref{eq1.12}.
By definition of $\tH=\omega^{\CI}H$ we have $a\in \tH(M\bullet N)\cap H_{\cY}(X)$.
Since $\tH$ is $\bcube$-invariant, so is $\tau_!\tH$ (see \cite[Lemma 1.14]{S}); thus
\mlnl{\gamma\in \ulMCor((X',\emptyset)\otimes (Y, \emptyset) \otimes\bcube, M\bullet N)\\
 \Longrightarrow i_0^*\gamma^*a=i_1^*\gamma^* a \quad \text{in }H(X'\times Y).}
Hence the natural map $\Z_\tr(X)\to H_{\cY}$ induced by $a\in H_{\cY}(X)$ factors by \eqref{lem:RSC-mod-twist1}
via $h_0(M)$, i.e., $M$ is a modulus for $a\in H_{\cY}(X)$.  Finally, observe for any $a\in \tau_!\tH((X,\emptyset)\te \cY)$ 
we always 
find compactifications $M$ and $N$ of $X$ and $Y$, respectively, such that
$M\te N$ is a modulus for $a$. It follows that any $a\in H_{\cY}(X)$ has a modulus, thus $H_{\cY}\in \RSC_{\Nis}$.
\end{proof}

\begin{cor}\label{cor:assocII}
Let $F_1,\ldots, F_n, G_1,\ldots, G_r\in \RSC_{\Nis}$  and $\bullet\in\{\te, \tte\}$. 
Set $\ul{\tF}=\tF_1\bullet\ldots\bullet \tF_n$, $\ul{\tG}=\tG_1\bullet\ldots\bullet\tG_r\in \MPST$.
There is a natural surjection 
\[h_{0,\Nis}(\ul{\tF})\ten{\NST}h_{0,\Nis}(\ul{\tG})\surj h_{0,\Nis}(\ul{\tF}\bullet \ul{\tG})
\quad \text{in }\NST,\]
which is compatible with the surjection of 
$F_1\ten{\NST}\ldots\ten{\NST}G_r$ to either side.
\end{cor}
\begin{proof}
%
%
Let $H:=h_{0,\Nis}(\ul{\tF}\bullet \ul{\tG}) \in \RSC_{\Nis}$ and
\[\phi\in\Lin^{\bullet}_{\RSC}(F_1,\ldots, F_n, G_1,\ldots,G_r\,;\, H).\]
 Fix $b_i\in G_j(Y_j)$, set $\ul{b}:=b_{1},\ldots, b_r$, $Y:=Y_1\times\ldots \times Y_r$,
$H_Y(X):=H(X\times Y)$, for $X\in \Sm$, and define
\[\phi_{\ul{b}, X_1,\ldots, X_{n}}: F_1(X_1)\times \ldots\times F_{n}(X_{n})
                           \to H_{Y}(X_1\times\ldots\times X_{n})\]
by 
\[\phi_{\ul{b}, X_1,\ldots, X_{n}}(a_1,\ldots, a_n):=\phi(a_1,\ldots, a_n, \ul{b}).\]
By Lemma \ref{lem:RSC-mod-twist} we have $H_Y\in \RSC_{\Nis}$.
Furthermore, it is clear that $\phi_{\ul{b}}:=\{\phi_{\ul{b}, X_1,\ldots, X_{n}}\}$ satisfies \ref{L1} and \ref{L2}
of Definition \ref{defn:Lin}; that it satisfies  \ref{L3} follows from the Lemmas \ref{lem:L3} and \ref{lem:RSC-mod-twist}.
More precisely, if $a_i\in \tF_i(\cX_i)$ and $b_j\in \tG_j(\cY_j)$, for some $\cX_i,\cY_j\in \MCor$ with $\cX_i^\o=X$ and 
$\cY_i^\o=Y_i$, then 
\[\phi_{\ul{b}, X_1,\ldots, X_n}(a_1,\ldots, a_n)\in \tH(\cX_1\bullet\ldots\bullet \cX_n\bullet\cY_1\ldots\bullet\cY_r)
\subset \widetilde{(H_{\cY})}(\cX_1\bullet\ldots\bullet \cX_n),\]
where $\cY=\cY_1\bullet\ldots\bullet \cY_r$ and $H_{\cY}$ is defined as in Lemma \ref{lem:RSC-mod-twist}.

Thus $\phi_{\ul{b}}$ induces by Theorem \ref{thm:ten-rep} a well-defined map
\[\hat{\phi}_{\ul{b}}: h_{0,\Nis}(\ul{\tF})\to H_{\cY}\subset H_Y.\]
Fix $X$ and $\alpha\in h_{0,\Nis}(\ul{\tF})(X)$. 
Define
\[\phi_{\alpha,Y_1,\ldots, Y_r} :  G_1(Y_1)\times \ldots \times G_{r}(Y_r)\to H_X(Y):=H(X\times Y)\]
by 
\[\phi_{\alpha,Y_1,\ldots, Y_r}(b_1,\ldots, b_r):= \hat{\phi}_{\ul{b}}(\alpha).\]
It is direct to check that $\phi_{\alpha}=\{\phi_{\alpha, Y_1,\ldots, Y_r}\}$ satisfies \ref{L1} and \ref{L2}.
Assume that $b_j$ has modulus $\cY_j\in \MCor$ as above.
Let $\cX=(\ol{X}, X_{\infty})$ be a compactification of $(X,\emptyset)$. Set $\cX^{(m)}=(\ol{X}, m X_\infty)$
and $\cY=\cY_1\bullet\ldots\bullet\cY_r$.
Then (see e.g. \cite[Lemma 1.4(5, Remark 1.5)]{S})
\[\tau_!\tH((X,\emptyset)\otimes \cY)=\varinjlim_{m}\tH(\cX^{(m)}\otimes \cY),\]
i.e., any $c\in \tau_!\tH((X,\emptyset)\otimes \cY)\subset H_X(Y)$ has a modulus of the form 
$\cX^{(m)}\te \cY$, some $m\ge 1$.
Thus 
\[\phi_{\alpha,Y_1,\ldots, Y_r}(b_1,\ldots, b_r)\in\tau_!\tH((X,\emptyset)\otimes \cY)\subset \widetilde{(H_X)}(\cY),\]
where the inclusion holds by Lemma \ref{lem:RSC-mod-twist} (with the role of $X$ and $Y$ interchanged).
Therefore, $\phi_\alpha$ satisfies \ref{L3}, by Lemma \ref{lem:L3}. Hence we obtain an induced map
$\hat{\phi}_{\alpha}: h_{0,\Nis}(\ul{\tG})\to H_X$. 
It is direct to check that the induced map $\hat{\phi}=\{\hat{\phi}_{X, Z}\}$,
with 
\[\phi_{X,Z}: h_{0,\Nis}(\ul{\tF})(X) \times h_{0,\Nis}(\ul{\tG})(Z)\to H(X\times Z), \quad
(\alpha, \beta)\mapsto \hat{\phi}_{\alpha}(\beta),\]
satisfies \ref{L1} and \ref{L2} and therefore induces the map from the statement.
\end{proof}

\begin{lem}\label{lem-T-rexact}
Let $F_1,\ldots, F_n\in\RSC_{\Nis}$ and set $\ul{\tF}=\tF_1\bullet\ldots\bullet \tF_n\in\MPST$,  $\bullet\in\{\te, \tte\}$.
 Let $G'\to G\to G''\to 0$  be an exact sequence in $\RSC_{\Nis}$.
\begin{enumerate}[label= (\arabic*)]
\item\label{lem-T-rexact2}  The natural map $h_{0,\Nis}(\ul{\tF}\bullet \tG)\surj h_{0,\Nis}(\ul{\tF}\bullet \wt{G''})$ is surjective.
\item\label{lem-T-rexact3}Assume that the induced maps
  \eq{lem-T-rexact3.1}{\tau_!\tG(\sO_L, \fm_L^{-n})\surj \tau_!\widetilde{G''}(\sO_L, \fm_L^{-n}), \quad L\in \Phi,\, n\ge 0,}
are surjective. Then the following sequence is exact
\[h_{0,\Nis}(\ul{\tF}\bullet \wt{G'})\to h_{0,\Nis}(\ul{\tF}\bullet \wt{G})\to h_{0,\Nis}(\ul{\tF}\bullet \wt{G''})\to 0.\]
\end{enumerate}
\end{lem}
\begin{proof}
We write $\Lin^\bullet$ instead of $\Lin^\bullet_{\RSC}$ and 
$\ul{F}=(F_1,\ldots, F_n)$.
For \ref{lem-T-rexact2} it suffices by Theorem \ref{thm:ten-rep} to show that for all $H\in \RSC_{\Nis}$
the natural map
$\Lin^\bullet(\ul{F}, G'' \,;\, H )\to \Lin^{\bullet}(\ul{F}, G\,;\, H)$
is injective. This follows directly from the surjectivity $G\surj G''$ in $\RSC_{\Nis}$.
Similarly for \ref{lem-T-rexact3} we have to show that for all $H\in \RSC_{\Nis}$ the sequence
\[0\to \Lin^{\bullet}(\ul{F},G''\,;\, H)\to \Lin^{\bullet}(\ul{F},G\,;\, H)\to 
\Lin^{\bullet}(\ul{F},G'\,;\, H)\]
is exact. It remains to show the exactness in the middle.
To this end, let $\phi\in \Lin^{\bullet}(\ul{F},G\,;\, H)$ map to zero in $\Lin^{\bullet}(\ul{F},G'\,;\, H)$.
For $a_j\in F_j(X_j)$ and $b\in G''(Y)$, take a Nisnevich cover
$\{ V_i\to Y\}_i$  such that there exist $\tilde{b}_i\in G(V_i)$ lifting $b_{|V_i}$.
Set $\bar{\phi}_{\ul{X}, V_i}(\ul{a}, b_{|V_i}):=\phi_{\ul{X}, Y}(\ul{a}, \tilde{b}_i)$
(with the obvious shorthand notation).
By assumption this  glues to give an element
\[\bar{\phi}_{\ul{X}, Y}(\ul{a}, b)\in H(\ul{X}\times Y).\]
It is immediate to check that this is independent of the lifts $\tilde{b}_i$ of $b_{|V_i}$ and the choice of the cover $\{V_i\to Y\}_i$.
It is also direct to check that $\bar{\phi}=\{\bar{\phi}_{\ul{X}, Y}\}_{\ul{X}, Y\in \Sm}$ satisfies 
\ref{L1} and \ref{L2} of Definition \ref{defn:Lin}. For  \ref{L3},
let $a_j\in F_j(L)$ and $b\in G''(L)$; by \eqref{lem-T-rexact3.1}  we find
$\tilde{b}\in G(L)$ with $c^G_L(\tilde{b})= c^{G''}_L(b)$; 
hence, if $\bullet =\te$
\[c^H_L(\bar{\phi}_L(\ul{a}, b))=c_L^H(\phi_L(\ul{a}, \tilde{b}))\le \sum_{i} c_L^{F_i}(a)+ c_L^{G}(\tilde{b})
=\sum_{i} c_L^{F_i}(a)+ c_L^{G''}(b).\]
If $\bullet =\tte$, we replace $+$ by $\max$.
Thus,
in either case, 
$\bar{\phi}\in \Lin^{\bullet}(\ul{F}, G''\,;\, H)$ maps to $\phi$. This completes the proof.
\end{proof}

\begin{cor}\label{cor:comp-te-tte}
Let $F_1,\ldots, F_n\in \RSC_{\Nis}$ and assume all but one of the $F_i$ are proper.
Then the natural surjection (see \ref{eq1.12}, \eqref{subsec:omegaCI2})
\[h_{0,\Nis}(\tF_1\te\ldots\te \tF_n)\xr{\simeq} h_{0,\Nis}(\tF_1\tte\ldots\tte \tF_n)\]
is an isomorphism.
(Recall that $F\in \RSC_{\Nis}$ is called {\em proper} if $c_L^F(a)=0$, for all $a\in F(L)$, $L\in \Phi$.)
\end{cor}
\begin{proof}
By assumption
$\Lin^{\tte}_{\RSC}(\ul{F}\,;\,H)= \Lin^{\te}_{\RSC}(\ul{F}\,;\,H)$, where as above $\ul{F}=F_1,\ldots, F_n$;
thus the  statement follows from Theorem \ref{thm:ten-rep}.
\end{proof}

\section{Applications}\label{sec:appl}
In this section we compute $h_{0,\Nis}(F_1\bullet\ldots\bullet F_n)$, $\bullet\in \{\te, \tte\}$,
in certain cases, where the $F_i\in\MPST$ are modulus lifts of the multiplicative - or the additive group, of abelian varieties,
or of generalized Jacobians. We also compare it to the reciprocity functors defined in \cite{IR}.

\begin{para}\label{not:GmM}
For $M \in \MCor$ we write 
\begin{equation}\label{eq:def-reduced-h}
h_0^\bcube(M)^0 :=
\ker(h_0^\bcube(M) \to h_0^\bcube(\Spec k, \emptyset)=\Z)
\in \CI,
\end{equation}
where the map is induced by the structure morphism of $M$.
We have 
\[\omega_{\CI}h_0^\bcube(M)^0= h_0(M)^0:=\Ker (h_0(M)\to h_0(\Spec k, \emptyset)) \in \RSC.\]
Hence we get a natural map
\[h_0^\bcube(M)^0\to \omega^{\CI}h_0(M)^0=\wt{h_0(M)^0}.\]
This map is in general not an isomorphism (e.g. not in the case $M=\G_a^M$ considered below.)

Set $\G_m^M:=(\P^1, (0)+(\infty)), ~ \G_a^M:=(\P^1, 2(\infty))$.
Fix $* \in \{ m, a \}$.
We put
\[
\G_*^\# := h_0^\bcube(\G_*^M)^0 \in \CI.
\]
By \cite[Theorem 1.1]{RY},
we have an isomorphism in $\RSC$
\begin{equation}\label{eq:Gsharp-G}
(\omega_\CI \G_*^\#)_{\Nis} \cong \G_*.
\end{equation}
For any field $K$ containing $k$,
we identify 
$\G_*^\#(K)=\G_*(K)$. 
We regard a rational function $f \in K(t)$ as
a morphism $f : \P^1_K \to \P^1_K = (\ol{\G_*^M})_K \to \ol{\G_*^M}$.
Let $D$ be an effective divisor on $\P^1_K$
such that $D \ge f^*((\G_*^M)^\infty)$.
Let $\Gamma_f \in \MCor((\P^1_K, D), \G_*^M)$
be the graph of $f$ (restricted to $\P_K^1 \setminus |D|$).
Set $\epsilon = 1$ (resp. $0$) if $*=m$ (resp. $a$),
regarded as a constant function.
Then the image of $\Gamma_f - \Gamma_\epsilon$
by the natural map
\[ \MCor((\P^1_K, D), \G_*^M) 
=\Z_\tr(\G_*^M)(\P^1_K, D) \to h_0^\bcube(\G_*^M)(\P^1_K, D)
\]
belongs to $\G_*^\#(\P^1_K, D)$,
which we denote by $f^\#$.
For a closed point $i : x \hookrightarrow \P^1_K \setminus |D|$,
we write $f(x):=i^*(f^\#) \in \G_*^\#(K(x))=\G_*(K(x))$.
\end{para}

\subsection{The case of homotopy invariant sheaves}

\begin{lem}\label{lem:bcHI}
Let $F\in \RSC_{\Nis}$. Assume there exists a natural number $n\ge 0$, such that 
$c^F_L(a)\le n$, for all  $L\in \Phi$ and all $a\in F(L)$ (see \ref{subsec:not-cond}, \ref{subsec:cond} for notation).
Then $F\in \HI_{\Nis}$ and we can take $n=1$.
\end{lem}
\begin{proof}
By \cite[Corollaries 4.33,  4.36]{RS} we have for general $F\in \RSC_{\Nis}$
\eq{lem:bcHI1}{F\in \HI_{\Nis} \Longleftrightarrow c_L^F(a)\le 1, \text{ for all } L\in \Phi, a\in F(L).}
By \cite[Theorem 4.15(2)]{RS} the motivic conductor $c^F=\{c^F_L\}_{L\in \Phi}$ is a conductor in the sense
of {\em loc. cit.}, in particular it satisfies 
\eq{c3}{\quad \quad c^F_L(f_*a)\le \lceil c^F_{L'}(a)/e\rceil,}
where $f:\Spec L'\to \Spec L$ is finite of ramification index $e$ and $a\in F(L')$.
Now assume $F$ is as in the assumptions of the lemma.
Let $L\in \Phi$ and $a\in F(L)$. Let $L_i/L$ be a  totally ramified extensions of degree $n+i$, with $i=0,1$, 
and denote by $\pi_i: \Spec L_i\to \Spec L$ the corresponding morphisms.
We compute
\begin{align*}
c_L(a)&=c_L((n+1)a-n a) & &\\
        &\le \max\{c_L((n+1)a), c_L(n a)\} & &\\
       &= \max\{c_L(\pi_{1*}\pi_1^*a), c_L(\pi_{0*}\pi_0^*a)\} & &\\
         &\le \max\{\lceil c_{L_1}(\pi_1^*a)/(n+1)\rceil, \lceil c_{L_0}(\pi_0^*a)/n\rceil\},& 
&\text{by \eqref{c3}}\\
        &\le 1, 
\end{align*}
where the last inequality holds by assumption. This proves the lemma.
\end{proof}

\begin{thm}\label{thm:tHI-tRSC}
Let $F_1,\ldots, F_n \in \HI_{\Nis}$, $\bullet\in \{\te, \tte\}$. 
Then
\[h_{0,\Nis}(\tF_1\bullet\ldots\bullet\tF_n) =F_1\ten{\HI_{\Nis}}\ldots \ten{\HI_{\Nis}} F_n.\]
\end{thm}
\begin{proof}
Denote by $F$ the left hand side. 
It follows from the Corollaries \ref{cor:mod-genRSC} and \ref{cor:cond-t}, \eqref{lem:bcHI1}, \eqref{c3},
 and Lemma \ref{lem:bcHI} that 
$F\in \HI_{\Nis}$.
We claim that there are natural surjections
\[F_1\ten{\NST}\ldots\ten{\NST} F_n\surj F \surj F_1\ten{\HI_{\Nis}}\ldots \ten{\HI_{\Nis}} F_n.\]
The first map is obtained by the Nisnevich sheafification of \eqref{subsec:omegaCI2}.
For the second,
we first notice that for any $G \in \MPST$
there is a natural surjection
\[ 
h_0(G) \twoheadrightarrow
h_0^{\A^1} h_0(G) \cong h_0^{\A^1} \omega_!(G),
\]
where the left surjection is from \eqref{eq:h0-a1}
and the right isomorphism is from \cite[Proposition 2.28]{KSY3}.
We then apply it to $G=\tF_1\bullet\ldots\bullet\tF_n$,
use the monoidality of $\omega_!$, and sheafify.
The claim is proved.

Applying $h_0^{\A^1}$ (see \ref{sec:omega-h}) and Nisnevich sheafifying
yield a factorization of the identity 
\[F_1\ten{\HI_\Nis} \ldots \ten{\HI_\Nis} F_n\surj   F \surj F_1\ten{\HI_\Nis} \ldots \ten{\HI_\Nis} F_n.\]
This proves the statement.
\end{proof}

\begin{problem}
Suppose that $G_i \in \MPST$  such that $F_i := \omega_!G_i \in \RSC$.
We have a natural surjection
\[ h_{0,\Nis}(G_1 \te \cdots \te G_n)
\surj h_{0,\Nis}(\tF_1\bullet\dots\bullet\tF_n).
\]
This is in general not an isomorphism, see, e.g., Theorem \ref{thm:Ga-te-Ga-arbCh}
and Corollary \ref{ga-ga-p}.
What happens if $G_i\in \CI$ and $F_i\in\HI$?
Is the map then also injective?
\end{problem}

\begin{para}\label{para:tenHI-Gm}
Let $\cK^M_n$ be the restriction to $\Sm$ of the improved Milnor $K$-sheaf from \cite[1.]{Ke10};
in particular it is a Zariski sheaf.
By \cite[Proposition 10, (8)]{Ke10} (see also \cite[Theorem 7.6]{Ke09}) 
\[\cK^M_n\cong \cH^n(\Z(n)),\]
where $\Z(n)$ is Voevodsky's motivic complex. In particular, $\cK^M_n\in \HI_{\Nis}$, 
by \cite[Theorem 3.1.12]{VoTmot}.
Employing $\Z(1)[1]\cong \G_m$ in $\DM$ (see \cite[Theorem 4.1]{MVW}) we obtain
\[\cK^M_n\cong \cH^0(\Z(n)[n])\cong \cH^0(\Z(1)[1]^{\ten{\DM} n})\cong \G_m^{\ten{\HI_\Nis} n}.\] 
Thus Theorem \ref{thm:tHI-tRSC} yields an isomorphism
\eq{para:tenHI-Gm1}{h_{0, \Nis}(\widetilde{\G_m}^{\bullet n})\xr{\simeq} \cK^M_n, \quad \bullet\in \{\te,\tte\}.}
It is direct to check that when evaluated at a $k$-field $K$ this map is given by 
\[(a_1\otimes\ldots\otimes a_n)_{K'/K}\mapsto \Nm_{K'/K}\{a_1,\ldots, a_n\},\]
where $K'/K$ is a finite field extension and $a_i\in (K')^\times$.
However, Theorem \ref{thm:tHI-tRSC} does not imply  the following result.
\end{para}

\begin{prop}\label{prop:Milnor-K}
We have isomorphisms in $\RSC_\Nis$
\[h_{0,\Nis}((\G^M_m/1)^{\bullet n})\cong h_{0, \Nis}((\G_m^\#)^{\bullet n}) \cong \cK_n^M,
\qquad \bullet\in\{\te,\tte\},
\]
where $\G^M_m/1= \Coker (i_{1}: \Z\to \Z_\tr(\G^M_m))\in \MPST$
and $i_1: \{1\}\inj \P^1$ (see \ref{not:GmM} for notation).
\end{prop}
\begin{proof}
The first isomorphism for $\bullet = \te$
follows from $h_0^{\bcube}(\G^M_m/1)\cong \G_m^\#$ and Lemma \ref{lem:hcube-ten}.
Thus for any $F \in \RSC_\Nis$ appearing in the statement,
we have a chain of surjective maps
\[
h_{0,\Nis}((\G_m^\#)^{\te n})
\surj
F 
\surj
h_{0,\Nis}(\widetilde{\G_m}^{\tte n})
\cong \cK_n^M.
\]
Hence by Lemma \ref{lem:ffc} it suffices to show that
$\alpha : h_{0}((\G_m^\#)^{\te n})(K)  \to K_n^M(K)$ is bijective for any finitely generated $k$-field $K$.
For this, we construct a surjection
$\beta : K_n^M(K) \surj h_{0}((\G_m^\#)^{\te n})(K)$
such that $\alpha \circ \beta = \id$.
We want to define $\beta$ by $\beta(\{a_1,\ldots, a_n\})= a_1\otimes\ldots\otimes a_n$.
If this is well defined, then $\beta$  is automatically surjective since
it is compatible with the surjection of  $\G_m^{\te{\PST} n}(K)$ to either side.
Showing well-definedness of $\beta$ amounts to showing
\begin{equation}\label{eq:vanishing-milK}
a_1 \otimes \dots \otimes a_n = 0
~\text{in}~ h_{0}((\G_m^\#)^{\te n})(K)
\end{equation}
for $a_1, \dots, a_n \in K^\times$
such that $a_i + a_j=1$ for some $i<j$.
This can be shown by a slight modification of
the proof of \cite[Proposition 5.3.1]{IR},
so we will be brief.

We may suppose $(i, j)=(1, 2)$,
and put $a:=a_1 = 1-a_2$,
$\boldsymbol{b}:=a_3 \otimes \dots \otimes a_n$.
Let $K'=K(c, \mu)$ be a finite extension of $K$ 
generated by $c, \mu \in {K'}^\times$ 
such that $c^6=a$ and such that
$\mu$ is of order $12$ (resp. $4$, resp. $3$)
if the characteristic is neither $2$ nor $3$
(resp. $3$, resp. $2$). 
Consider rational functions in $K'(t)$
\begin{align*} 
&f=
\frac{t^6-a}{t^6-(a+1)t^4+(a+1)t^2-a}
=\frac{(t^2-c^2)(t^2-\mu^2 c^2)(t^2 - \mu^4 c^2)}
{(t^2-c^6)(t^2-\mu^2)(t^2-\mu^{10})},
\\
&g_1=t, \quad g_2=1-t, \quad g_i=a_i ~(i=3, \dots, n)
\end{align*}
so that (with the notation from \ref{not:GmM})
\begin{align*}
&g_1^\# \in \G_m^\#(\P^1_{K'}, (0)+(\infty)),\\
&g_2^\# \in \G_m^\#(\P^1_{K'}, (1)+(\infty)), \\
&g_i^\# \in \G_m^\#(\P^1_{K'},\emptyset)\quad (i=3,\dots,n).
\end{align*}
Since $f \equiv 1 \mod (0)+(1)+2(\infty)$,
we may apply Theorem \ref{thm:CI-te-fields} (R2)
to get a vanishing element in $h_0((\G_m^\#)^{\te n})(K')$.
A slight modification of the computation in {\em loc. cit.} shows that
$4(a \otimes (1-a) \otimes \boldsymbol{b})=0$ in $h_0((\G_m^\#)^{\te n})(K')$.
Since $[K':K]$ is divisible by $24$, we get
$96(a \otimes (1-a) \otimes \boldsymbol{b})=0$ in $h_0((\G_m^\#)^{\te n})(K)$.
Now, \eqref{eq:vanishing-milK} follows from
exactly the same argument as \cite[Lemma 5.8]{MVW}.
\end{proof}

\subsection{Comparison with the K-group of reciprocity functors}

\begin{thm}\label{cor:tenT}
Let $F_1,\ldots, F_n\in \RSC_{\Nis}=\RSC\cap \NST$. 
Let $\mathbf{Reg}^{\le 1}\subset\ulMCor^{\pro}$ be the full subcategory 
whose objects are regular schemes of dimension $\le 1$, which are of finite type and separated
over a finitely generated field extension of $k$. 
Then the restriction (via \ref{sec: MCorpro2}) of $F_i$ to $\mathbf{Reg}^{\le 1}$ is a 
reciprocity functor in the sense of \cite[Definition 1.5.1]{IR}. Furthermore, if 
$T(F_1,\ldots, F_n)$ denotes the reciprocity functor defined in \cite[4.2.3]{IR} and $K/k$ is a finitely generated
field, then
\[h_{0,\Nis}(\tF_1\tte \ldots\tte \tF_n)(K)= T(F_1, \ldots, F_n)(K).\]
\end{thm}
\begin{proof}
For the first statement, observe that $F\in \RSC_{\Nis}$ clearly restricts to
a Nisnevich sheaf with transfers on $\mathbf{Reg}^{\le 1}$ in the sense of \cite[Definition 1.2.1]{IR} satisfying
the condition (FP) from \cite[Definition 1.3.5]{IR}. The condition (Inj) from {\em loc. cit.}, i.e., the injectivity of
the restriction $F(X)\inj F(U)$ for $U\subset X$ open dense in $\mathbf{Reg}^{\le 1}$ is satisfied by 
\cite[Theorem 6]{KSY} and \cite[Corollary 3.2.3]{KSY3}. Finally, note by 
\cite[Lemma 5.1.7]{KSY} and \cite[Corollary 3.2.3]{KSY3}, the reciprocity sheaf $F$  has weak reciprocity in the 
sense of \cite[5.1.6]{KSY}, which implies that any element $a\in F(K(C))$, 
where $K(C)$ is the function field of a regular projective curve over a finitely generated $k$-field,
has a modulus in the sense of \cite[Definition 1.4.1]{IR}. Hence $F$ defines a reciprocity functor
in the sense of \cite[Definition 1.5.1]{IR}. This shows the first statement.

For the second statement observe that by Theorem \ref{thm:CI-te-fields} we have 
a morphism of presheaves with transfers on $\mathbf{Reg}^{\le1}$
\eq{cor:tenT1}{LT(F_1,\ldots, F_n)\to h_0(\tF_1\tte\ldots\tte \tF_n):=F ,}
where $LT(F_1,\ldots, F_n)$ is defined in \cite[Definition 4.2.3(1)]{IR}, and 
it induces an isomorphism
\eq{cor:tenT2}{LT(F_1,\ldots, F_n)(K)\xr{\simeq} F(K).}
Let $\Sigma$ be the functor from \cite[Proposition 3.1.4]{IR}.
By the adjunction property of $\Sigma$, the composition of 
\eqref{cor:tenT1} with the natural map $F\to F_{\Nis}$ factors as follows in the category of presheaves with transfers on
${\rm Reg}^{\le 1}$
\[LT(F_1,\ldots,F_n)\to T(F_1,\ldots, F_n)\overset{\text{defn}}{=}\Sigma(LT(F_1,\ldots, F_n))\to F_{\Nis}.\]
Thus the isomorphism \eqref{cor:tenT2} factors as 
\[LT(F_1,\ldots,F_n)(K)\surj T(F_1,\ldots, F_n)(K)\to F_{\Nis}(K)=F(K),\]
where the first map on the left is surjective by the construction of $\Sigma$, see \cite[Proposition 3.1.4]{IR}.
This implies the statement.
\end{proof}

\subsection{Tensors of additive groups}
\begin{cor}\label{cor:prop-tten}
Let $F\in \RSC_{\Nis}$.
\begin{enumerate}[label=(\arabic*)]
\item\label{cor:prop-tten1} Assume  $\ch(k)\neq 2$.
We have
\[h_{0,\Nis}(\wt{\G_a}\tte \wt{\G_a}\tte \wt{F})=0.\]
\item\label{cor:prop-tten2} Assume the characteristic of $k$ is zero.
Let $G$ be a unipotent commutative group scheme over $k$ and $A$ an abelian $k$-variety.
Then $G, A\in \RSC_{\Nis}$ and 
\[h_{0,\Nis}(\wt{A}\te \tG\te \wt{F})=h_{0,\Nis}(\wt{A}\tte \tG\tte \wt{F})=0.\]
\end{enumerate}
\end{cor}
\begin{proof}
\ref{cor:prop-tten1}.   By \cite[Theorem 5.5.1]{IR}, we have $T(\G_a, \G_a, F)(K)=0$, for all function fields $K$.
Hence the statement follows from Theorem \ref{cor:tenT} and Lemma \ref{lem:ffc}.
\ref{cor:prop-tten2}. 
The first statement follows from \cite[Corollary 3.2.5]{KSY3}.  To show the vanishing, recall that a unipotent commutative group scheme
in characteristic zero is a product of $\G_a$. By \cite[Corollary 1.2]{RY14} we have $T(A,\G_a)(K)=0$.
Thus the statement follows from Corollary \ref{cor:assocII}, Corollary \ref{cor:comp-te-tte},
and Theorem \ref{cor:tenT}.
\end{proof}
\begin{remark}
The vanishing results above were conjectured by Bruno Kahn, even before a precise definition of the terms were
available.
\end{remark}

Next we compute $h_{0,\Nis}(\wt{\G_a}\te \wt{\G_a})$, in particular it does not vanish.

\begin{para}\label{para:omega}
Denote by $\Omega^n=\Omega^n_{-/\Z}$ the sheaf of absolute K\"ahler differentials.
By \cite[Corollary 3.2.5]{KSY3} we have $\Omega^n\in \RSC_{\Nis}$.
Note also that $\G_a=\Omega^0$ in $\RSC_{\Nis}$. 
%
%
%
%
Furthermore, for a function field $K$ and a regular projective $K$-curve $C$,
the local symbol at a closed point $x\in C$ (in the sense of \cite[Proposition 5.9]{KSY})
$(-,-)_{C/K,x}:\Omega^q_{K(C)}\times K(C)^\times\to \Omega^q_K$
is given by 
\eq{para:omega2}{(a,f)_{C/K,x}=\Res_{C/K,x}(a\dlog f), }
where $\Res_{C/K,x}$ is the residue symbol at $x$ (see e.g. \cite[17.4]{Kunz}).
If $K$ is non-perfect, this requires a small argument, see \cite[Lemma 7.11]{RS}.
\end{para}

\begin{para}\label{para:P}
Let $\cP^1$ be the sheaf on the Zariski site of all schemes given by 
\[\cP^1(X)= \Gamma(X, \Delta_X^{-1}(\sO_{X\times_\Z X}/ I_{\Delta_X}^2)),\]
where $\Delta_X: X\to X\times_{\Z} X$ is the diagonal and $I_{\Delta_X}$ is the ideal  sheaf defined by $\Delta_{X}$.
We have the isomorphisms  
\eq{para:P1}{\varphi, \varphi':\cP^1\xr{\simeq} \Omega^1\oplus\G_a}
given by 
\[\varphi(a\otimes b)= adb\oplus ab\quad \text{and}\quad \varphi'(a\otimes b)=bda\oplus ab.
\]
The inverse of $\varphi$ is given by 
\eq{para:P2}{\varphi^{-1}(adb \oplus c)=a\otimes b-ab\otimes1+c\otimes 1,}
and similarly for $\varphi'$.
The restriction of $\cP^1$ to $\Sm$ can be equipped with transfers 
via $\varphi$ or $\varphi'$,
but the following lemma shows they are the same.
We obtain $\cP^1\in \RSC_{\Nis}$.
Note that by \eqref{para:P1}
any open dense immersion $j : U \hookrightarrow X$
induces an injection:
\begin{equation}\label{eq:open-emb-inj}
j^* : \cP^1(X) \hookrightarrow \cP^1(U).
\end{equation}
\end{para}

\begin{lem}\label{lem:transferP}
The two transfer structures on $\cP^1$ induced by 
$\varphi$ and $\varphi'$ coincide.
\end{lem}
\begin{proof}
Let $Z\in \Cor(X, Y)$ be a finite prime correspondence. For $\alpha\in \cP^1(Y)$ denote by
$Z^*\alpha$ the action defined via $\varphi$, and by $Z^\bigstar\alpha$ the action defined by
$\varphi'$. We have to show $Z^*\alpha=Z^\bigstar \alpha$ in $\cP^1(X)$.
By \eqref{eq:open-emb-inj},
we may shrink $X$ and therefore assume 
that $Z$ is smooth. Denote by $f: Z\to Y$ and $g: Z\to X$ the maps induced by projection
and write $f^*=\Gamma_f^*$ and $g_*= (\Gamma_g^t)$, where $\Gamma_f$ and $\Gamma_g^t$
are the graph of $f$ and the transpose of the graph of $g$, respectively, similar with $\bigstar$.
We obtain $Z^*=g_*f^*$ and $Z^\bigstar=g_\bigstar f^\bigstar$.
Obviously we have $f^*=f^\bigstar$. It remains to show $g_*=g_\bigstar$. This is local in $X$ and we can therefore
assume $X=\Spec K$, with $K$ a field and $g$ corresponds to a finite field extension $L/K$.
Via $\varphi$ (resp. $\varphi'$) the pushforward $g_*$ (resp. $g_\bigstar$)  corresponds to $\Tr\oplus \Tr$
 on $\Omega^1\oplus \G_a$ (we write $\Tr$ for $\Tr_{L/K}$). 
 Note
 \[\varphi^{-1}(adb\oplus c)= a\otimes b- ab\otimes 1 + c\otimes 1, \quad 
 {\varphi'}^{-1}(adb\oplus c)= b\otimes a- 1\otimes ab + 1\otimes c.\]
 By transitivity it suffices to consider the cases in which $L/K$ is either separable or purley inseparable
 of degree $p$.

{\em 1st case: $L/K$ is separable:}
From the isomorphism $\varphi$ we find that any element in $\cP^1(L)$
can be written as a sum of elements
\[\alpha:= a\otimes b- ab\otimes 1 + c\otimes 1, \quad a, c\in L, b\in K.\]
Using that $\Tr$ is $K$-linear and commutes with $d$ we obtain
\begin{align*}
f_\bigstar(\alpha)& =\Tr(a)\otimes b -1\otimes \Tr(a)b + 1\otimes \Tr(a)b\\
                            &\quad - \Tr(a)b\otimes 1 +1\otimes \Tr(a)b- 1\otimes \Tr(a)b\\
                         &\quad + \Tr(c)\otimes 1 -1\otimes \Tr(c)  + 1\otimes \Tr(c)\\
                         &= \Tr(a)\otimes b- \Tr(a)b\otimes 1 + \Tr(c)\otimes 1\\
                         &=f_*(\alpha).
\end{align*}

{\em 2nd case: $L/K$ is purely inseparable of degree $p$.}
We can write $L=K[x]$ with $x\in L\setminus L^p$ and $x^p=:y\in K$.
From the isomorphism $\varphi$ we see that we can write any element as a sum of the following elements
\[\alpha=ax^i\otimes x - ax^{i+1}\otimes 1 + c\otimes 1, \quad a\in K, i\in \{0,\ldots, p-1\}, c\in L\]
\[\beta= ax^i\otimes b - abx^i\otimes 1 + c\otimes 1, \quad a, b\in K, i\in \{0,\ldots, p-1\}, c\in L.\]
Note that $0=\Tr: L\to K$, and for $a\in K$ we have 
\eq{lem:transferP1}{\Tr(x^i da)=0,\quad \text{all } i, \text{ and}\quad  
\Tr(a x^idx)=\begin{cases}0, & \text{if } i\in \{0,\ldots, p-2\}\\ ady, & \text{if }i=p-1, \end{cases}}
see, e.g., \cite[(2.2.6) Definition]{KunzTr} (or \cite[Theorem 2.6]{Ru}) for the definition of 
the trace on differential forms of a purely inseparable extension.
Using that $(a\otimes 1-1\otimes a)(1\otimes y-y\otimes 1)=0$ in $\cP^1(K)$ we obtain
\begin{align*}
f_\bigstar\alpha & = {\varphi'}^{-1}\Tr\varphi'(\alpha)\\
                       &=\begin{cases} 
                                   0, &\text{if } i\in \{0,\ldots, p-2\},\\ 
                            -(y\otimes a - 1\otimes ya), &\text{if } i=p-1
                              \end{cases}\\
                      &=\begin{cases} 0, &\text{if } i\in \{0,\ldots, p-2\},\\
                                           a\otimes y -ay\otimes 1, & \text{if } i=p-1,
                             \end{cases}\\
                       &= \varphi^{-1}\Tr\varphi(\alpha)\\
                       &= f_*\alpha.
\end{align*}
Furthermore, it follows directly from the above formulas and $\Tr d=d\Tr$ that we have 
$f_*\beta=0=f_\bigstar \beta$.
This completes the proof.
\end{proof}

\begin{prop}\label{prop:map-to-Ga-te-Ga}
Assume $\ch(k)\neq 2$. Let $K$ be a $k$-field. Then the morphism
\[\Psi_K: \cP^1(K)\to h_{0,\Nis}(\G_a^\#\te\G_a^\#)(K), \quad a\otimes b\mapsto [a,b]_K\]
is well-defined and surjective (see Corollary \ref{cor:mod-genRSC} for the bracket notation on the right).
\end{prop}
\begin{proof}
We have $\cP^1(K)= K\otimes_{\Z} K/ I_{\Delta_K}^2$.
As a group $I_{\Delta_K}^2$ is generated by elements of the form 
\eq{prop:map-to-Ga-te-Ga1}{(c\otimes1)\cdot ( 1\te ab+ab\te 1-a\te b-b\te a), \quad a,b,c\in K\setminus\{0\}.}
Thus for the well-definedness of $\Psi$ it suffices to show that this element is mapped to zero in 
$h_{0,\Nis}(\G_a^\#\te\G_a^\#)(K)$. 
To this end, consider the following functions on $\P^1_K$ (where $\P^1_K\setminus\{\infty\}=\Spec K[t]$)
\[f=\frac{(t^2-\frac{a^2}4)(t^2-b^2)(t^2-(1+\frac{ab}2)^2)}{(t^2-1)(t^2-\frac{a^2b^2}4)(t^2-(\frac{a}2+b)^2)},
\quad g_1= c\cdot t, \quad  g_2=t.\]
One can check that $f\equiv 1 \mod 4(\infty)$ and that $g_1^\#, g_2^\# \in \G_a^\#(\P^1,2\infty)$
(see \ref{not:GmM} for notation).
Therefore, by Theorem \ref{thm:CI-te-fields}, 
in $(\G_a^\# \ten{\CI} \G_a^\#)(K)$ we have  
\begin{align*}
0
&=\sum_{x\in \P^1\setminus\{\infty\}}v_x(f) \Tr_{K(x)/K}([g_1(x), g_2(x)]_{K(x)})\\
&=ca\te \frac{a}2 +2cb\te b+c(2+ab)\te (1+\frac{ab}2)\\
&\qquad -2c\te 1-cab\te\frac{ab}2-c(a+2b)\te (\frac{a}2+b)\\
&= c\te ab+cab\te 1-ca\te b-cb\te a.
\end{align*}
Thus $\Psi_K$ is well-defined.
By Corollary \ref{cor:mod-genRSC} $\Psi_K$ is surjective if we show
\eq{prop:map-to-Ga-te-Ga2}{\Tr_{K'/K}\circ \Psi_{K'}= \Psi_{K}\circ \Tr_{K'/K}: 
\cP^1(K')\to h_{0,\Nis}(\G_a^\#\te\G_a^\#)(K),}
for all finite field extensions $K'/K$.
By definition of the transfer structure on $\cP^1$ we have to show this equality after precomposing with 
\eqref{para:P2}.
By the transitivity of the trace it suffices to consider separately the cases in which  $K'/K$ is either separable or
purely inseparable of degree $p=\ch(k)$. Set $\Psi'_K:=\Psi_K\circ \eqref{para:P2}$ and $\Tr:=\Tr_{K'/K}$.

{\em 1st case: $K'/K$ separable:} In this case $\Omega^1_{K'}$ is generated by elements of the form
$a'db$ with $a'\in K'$ and $b\in K$. Let $c'\in K'$. We compute
\begin{align*}
\Psi'_K\Tr(a'db\oplus c') & = \Psi_{K}'(\Tr(a')db\oplus \Tr(c'))\\
                                & = [\Tr(a'),b]_K - [\Tr(a')b, 1]_K + [\Tr(c'), 1]_K\\
                                & = \Tr([a', b]_{K'} - [a'b,1]_{K'}+ [c',1]_{K'})\\
                               &=\Tr\Psi'_{K'}(a'db\oplus c'),
\end{align*}
where for the pre-last equality we use projection formulas.
This yields \eqref{prop:map-to-Ga-te-Ga2} in this case.
Note that the argument on $\G_a$ works for any field extension $K'/K$, thus we can assume $c' =0$ in the following.

{\em 2nd case: $K'/K$ is purely inseparable of degree $p$.}
In this case we can write $K'=K[x]$, where $x\in K'$ and  $y:=x^p\in K$. 
It follows that every element in $\Omega^1_{K'}$ can be written as a sum of elements of the form 
\[x^i a db,  \quad x^{i}adx, \quad a,b\in K, i\in \{0,\ldots, p-1\}.\]
For $x^i adb$ we argue as in the case above.
For $0\le i\le p-2$ we have $x^iadx=\frac{1}{i+1}a d x^{i+1}$ and 
thus $\Tr(x^iadx)= \frac{1}{i+1}a d \Tr(x^{i+1})$ and we can again argue similarly as in the above case.
(In fact these traces are zero, but this is not needed.)
It remains to consider the element $x^{p-1}a dx$. 
We have 
\[\Psi'_{K}\Tr(x^{p-1}a dx)= \Psi'_K(ady)=[a,y]_K- [ay,1]_K.\]
On the other hand 
\begin{align*}
\Tr \Psi'_{K'}(x^{p-1}a dx) &= \Tr ([ax^{p-1},x]_{K'}- [ay,1]_{K'})\\
                                     &= \Tr [ax^{p-1},x]_{K'},
\end{align*}
for the last equality we use the projection formula and $\Tr(1)=p=0$ (in the case under consideration).
Thus it remains to show
\eq{prop:map-to-Ga-te-Ga3}{\Tr([ax^{p-1},x]_{K'})-[a,y]_K+[ay,1]_K=0.}
To this end, consider the following functions on $\P^1_K$ (where $\P^1_K\setminus\{\infty\}=\Spec K[t]$)
\[f=\frac{(t^{2p}-y^2)(t^2-1)}{(t^2-y^2)(t^{2p}-1)}, \quad g_1=\frac{ay}{t}, \quad g_2=t.\]
Then $g_1^\#\in \G^\#_a(\P^1_K, 2\cdot 0)$, $g_2^\#\in \G_a^\#(\P^1_K, 2\cdot \infty)$ and 
$f\equiv 1$ mod $2\cdot (0+\infty)$.
Note that $V(t^p-y)$ defines a point in $\P^1_K\setminus \{0,\infty\}$ with residue field $K'$.
By Theorem \ref{thm:CI-te-fields} we obtain in $h_{0,\Nis}(\G_a^\#\te\G_a^\#)(K)$
\begin{align*}
0&=\sum_{z\in \P^1_K\setminus\{0,\infty\}} v_z(f) \cdot \Tr_{K(z)/K}([g_1(z), g_2(z)]_{K(z)})\\
  &= 2 \Tr_{K'/K}([ay/x,x]_{K'}) + 2[ay,1]_{K} - 2 [a,y]_K - 2p[ay,1]_K\\
  &=   2 (\Tr_{K'/K}([ax^{p-1},x]_{K'}) + [ay,1]_{K} -  [a,y]_K).
\end{align*}
Since $p\neq 2$, this shows \eqref{prop:map-to-Ga-te-Ga3} and completes the proof of the proposition.
\end{proof}

\begin{thm}\label{thm:Ga-te-Ga-arbCh}
Assume $\ch(k)\neq 2$.
Then there are canonical isomorphisms in $\RSC_{\Nis}$
\[h_{0,\Nis}(\G_a^M/0\te \G_a^M/0)\cong 
h_{0,\Nis}(\G_a^\# \te \G_a^\#)\cong  \cP^1,\]
where $\G_a^M/0= \Coker (i_0: \Z \to \Z_{\tr}(\G_a^M))$ (see \ref{not:GmM} for notation).
\end{thm}
\begin{proof}
The first isomorphism follows from $h_0^{\bcube}(\G^M_a/0)\cong \G_a^\#$ and Lemma \ref{lem:hcube-ten}.
For a prime correspondence 
$Z\in \Cor(X, \P^1_t\setminus\{\infty\}\times \P^1_s\setminus\{\infty\})$ (here $t$, $s$ are the coordinates away from $\infty$)
we define
\[\theta(Z):= Z^* (s\otimes t)\in \cP^1(X),\]
where we view $s\otimes t\in \cP^1(\P^1_s\setminus \{\infty\}\times \P^1_s\setminus\{\infty\})$ and 
$Z^*$ denotes the transfer action. Clearly $\theta$ extends to a well-defined morphism in $\PST$
\[\theta: \frac{\uomega_!\Z_\tr(\G_a^M\te \G_a^M)}{(i_0\times \id)\uomega_!\Z_\tr(\G_a^M)
+ (\id\times i_0)\uomega_!\Z_\tr(\G_a^M)}\to \cP^1,\]
where $i_0: \Spec k\to (\P^1, 2\cdot \infty)$ is induced by the zero-section.
We want to show that $\theta$
factors via
\eq{thm:Ga-te-Ga-arbCh1}{\bar{\theta}: h_{0,\Nis}(\G_a^M/0\te \G_a^M/0)\to \cP^1.}
To this end it suffices to show that if  $Z\in \uMCor((\P^1_X, 1_X), \G^M_a\te\G^M_a)$ is a prime correspondence
then
\[\theta(Z\circ i_0 - Z\circ i_\infty)=0 \quad \text{in }\cP^1(X).\]
It suffices to show this in the case where $X=\Spec K$ is a function field
by \eqref{eq:open-emb-inj}.
Hence we can assume that the correspondence $Z$ factors via
\[Z= \pi\circ (f^t),\]
where $\pi: C\to \P^1\times \P^1$ is a morphism from a proper regular $K$-curve,  
$f\in K(C)^\times$ is a function satisfying $f\equiv 1$ mod $\pi^*(2\cdot \infty\times \P^1 + \P^1\times 2\cdot \infty):=D$,
and $f^t\subset \P^1\times C$ is the transpose of the graph of $f$.
Using the identification \eqref{para:P1} and \eqref{para:omega2} we obtain
\begin{align*}
\theta(Z\circ i_0 - Z\circ i_\infty)& = \Div_C(f)^* \pi^*(sdt\oplus st)\\
                                         & =\sum_{x\in C\setminus |D| } \Res_{C/K,x}(\pi^*(sdt)\dlog f)\\
                                          &\quad \oplus \sum_{x\in C\setminus |D| } \Res_{C/K,x}(\pi^*(st)\dlog f).
\end{align*}
Now $s\in \sO(\P^1_s\setminus \{\infty\})$ has pole divisor equal to $\infty$
thus $\pi^*s\in \sO(C\setminus|D|)$ has  pole  equal to $\pi^*(p_1^*\infty)$, where $p_1:\P^1_s\times \P^1_t\to \P^1_s$
is the first projection;
similar with $\pi^*t$. 
Thus if $x\in |D|$ and $n_i$ is the multiplicity of $x$ in $\pi^*(p_1^*\infty)$,
then $\pi^*(sdt)$ and $\pi^*(st)$ have a pole of order at most $n_1+n_2+1$;
on the other hand we have $f\in 1+\fm_x^{2n_1+2n_2}$. 
We find
\[\pi^*(sdt)\dlog f\in \Omega^2_{C,x}, \quad \pi^*(st)\dlog f\in \Omega^1_{C,x}, \quad \text{for }x\in |D|.\]
We obtain
\[
\theta(Z\circ i_0 - Z\circ i_\infty) = \sum_{x\in C} \Res_{C/K,x}((\pi^*(sdt)\oplus \pi^*(st))\dlog f)=0,
\]
by the reciprocity law for the residual symbol. 
This shows the existence of $\bar{\theta}$. 
By Lemma \ref{lem:ffc} it suffices to show that this is an isomorphism on function fields $K/k$.
It is direct to check that the composition
\[\cP^1(K)\xr{\Psi_K} h_{0,\Nis}(\G_a^\#\te \G_a^\#)(K)\cong 
h_{0,\Nis}(\G_a^M/0\te \G_a^M/0)(K)\xr{\bar{\theta}} \cP^1(K) \]
is the identity, where $\Psi_K$ is the map from Proposition \ref{prop:map-to-Ga-te-Ga}.
Since $\Psi_K$ is surjective it follows that $\bar{\theta}(K)$ is an isomorphism.
This completes the proof.
\end{proof}

\begin{thm}\label{thm:Ga-te-Ga}
Assume $k$ is of characteristic zero.
Then, in $\RSC_\Nis$,  we have isomorphisms
\[ h_{0,\Nis}(\G_a^M/0\te \G_a^M/0)\cong 
h_{0,\Nis}(\G_a^\# \te \G_a^\#)\cong h_{0,\Nis}(\wt{\G_a}\te \wt{\G_a})\cong \cP^1.
\]
\end{thm}
\begin{proof}
The first isomorphism follows from $h_0^{\bcube}(\G^M_a/0)\cong \G_a^\#$ and Lemma \ref{lem:hcube-ten}.
For $X, Y\in \Sm$ define
\[\phi_{X,Y}: \G_a(X)\times \G_a(Y)\to \cP^1(X\times Y)\]
by
\[\phi_{X,Y}(a,b):= p^*(a)\otimes q^*(b) \text{ mod } I_{\Delta_{X\times Y}}^2,\]
where $p: X\times Y \to X$ and $q: X\times Y\to Y$ are the projections
and set $\phi:=\{\phi_{X,Y}\}_{X,Y}$. Clearly $\phi$ satisfies \ref{L1} from Definition \ref{defn:Lin}.
For \ref{L2} consider $f\in \Cor(X', X)$, $a\in\G_a(X) $, $b\in \G_a(Y)$. We have to show
\eq{thm:Ga-te-Ga1}{(f\times \id_Y)^*\phi_{X,Y}(a,b)= \phi_{X', Y}(f^*a, b) \quad \text{in } \cP^1(X'\times Y).}
It suffices to show this after shrinking $X'$ arbitrarily around the generic points of $X'$ by \eqref{eq:open-emb-inj}.
Thus we can assume that $f$ factors as $f=h\circ g^t$, 
where $h: Z\to X$ is a morphism in $\Sm$ and $g: Z\to X'$ is finite and surjective and $g^t\subset X'\times Z$
is the transpose of the graph. Clearly \eqref{thm:Ga-te-Ga1} holds for $f=h$.
Thus we can assume $f=g^t$ with $g: X\to X'$ finite and surjective.
Let $p':X'\times Y\to X'$ and $q': X'\times Y\to Y$ be the projections and set $g_Y:= g\times \id_Y$.
Note 
\eq{thm:Ga-te-Ga2}{ q= q'\circ g_Y.}
Recall that the transfer action is induced by \eqref{para:P1}. 
We compute
\begin{align*}
(f\times \id_Y)^*\phi_{X,Y}(a,b) & =g_{Y*} \phi_{X,Y}(a,b) \\
                                            &= g_{Y*}( p^*(a)dq^*(b))\\
                                           &\quad \oplus g_{Y*}(p^*(a)\cdot q^*(b)), &
\text{via }\eqref{para:P1}, \\
                           & =g_{Y*}( p^*(a)\cdot  g_{Y}^*d{q'}^*(b))\\
                           &\quad  \oplus    g_{Y*}(p^*(a)\cdot g_Y^*{q'}^*(b)),& \text{by }\eqref{thm:Ga-te-Ga2},\\
                          & = (g_{Y*}p^*(a))\cdot  d{q'}^*(b)\\
                             & \quad \oplus    (g_{Y*}p^*(a))\cdot {q'}^*(b), & \text{proj. formula},\\
                         & ={p'}^* (g_*a) \cdot d {q'}^*b \oplus {p'}^* (g_*a) \cdot {q'}^*b\\
                         &= \phi_{X', Y}(g_*a, b), & \text{via }\eqref{para:P1},\\
                         &= \phi_{X',Y}(f^*a, b).
\end{align*}
This shows \eqref{thm:Ga-te-Ga1}. The corresponding formula with $f\in \Cor(Y',Y)$ involving
$(\id_X\times f)^*$ is proved similarly. (Use that the differential $d: \sO\to \Omega^1$
 is a morphism in $\PST$ hence $g_*d=d g_*$.) Thus $\phi$ satisfies \ref{L2}.
Let $L\in \Phi$ and $a,b\in L$. For \ref{L3} we have to show
(see \ref{subsec:not-cond}, \ref{subsec:cond} for notation)
\eq{thm:Ga-te-Ga2.5}{c^{\cP^1}_L(\phi_L(a,b))\le c_L^{\G_a}(a)+c_L^{\G_a}(b).}
For $L\in \Phi$ and $a\in \Omega^n_L$ we have by \cite[Theorem 6.4]{RS} 
\eq{cor:DR01}{ c^{\Omega^n}_L(a)\le r \Longleftrightarrow a\in \begin{cases}
\Omega^n_{\cO_L}, &\text{if } r=0,\\
\frac{1}{t^{r-1}}\cdot \Omega^{n}_{\cO_L}(\log), & \text{if }r\ge 1.
\end{cases}}
By \eqref{para:P1} we have 
\[c^{\cP^1}_L(\phi_L(a,b))=\max\{c^{\Omega^1}_L(adb), c_L^{\G_a}(a\cdot b)\}.\]
Thus \ref{L3} follows from \eqref{cor:DR01} (with $n=0,1$).
By Theorem \ref{thm:ten-rep} we obtain a morphism in $\RSC_{\Nis}$
\[\phi: h_{0,\Nis}(\wt{\G_a}\te \wt{\G_a})\to \cP^1.\]
Composing with the natural surjection $h_{0,\Nis}(\G_a^\#\te \wt\G_a^\#)\surj h_{0,\Nis}(\wt{\G_a}\te \wt{\G_a})$
we obtain the morphism
\eq{thm:Ga-te-Ga3}{h_{0,\Nis}(\G_a^\#\te \G_a^\#)\to \cP^1.}
By Lemma \ref{lem:ffc} it suffices to show that this is an isomorphism on function fields $K/k$.
On such a $K$ the morphism $\Psi_K$ from Proposition \ref{prop:map-to-Ga-te-Ga}
is clearly inverse to \eqref{thm:Ga-te-Ga3}; this proves the theorem.
\end{proof}

\begin{remark}\label{rmk:Ga}
The isomorphism $h_{0,\Nis}(\wt{\G_a}\te \wt{\G_a})\cong \cP^1$ from
Theorem \ref{thm:Ga-te-Ga} does not generalize to positive characteristic, see Corollary \ref{ga-ga-p} below.
One reason is that formula \eqref{cor:DR01} 
does not hold in this case.
Indeed, if ${\rm char}(k)=p>0$, then  for $L\in \Phi$ with local parameter $t$ and  $r\ge 2$
\[\wt{\G_a}(\cO_L,\fm_L^{-r})=\begin{cases}
\sum_{j\ge 0}F^j(\frac{1}{t^{r-1}}\cdot\cO_L),&\text{if } (r,p)=1,\\
\sum_{j\ge 0}F^j(\frac{1}{t^r}\cdot \cO_L),&\text{if } p\mid r,
\end{cases}
\]
where $F$ denotes the Frobenius, see \cite[Theorem 7.15]{RS}. 
Hence in positive characteristic 
\eqref{cor:DR01} does not hold.
(For example, $F(1/t^{r-1})\cdot 1/t^{s-1}\not\in \wt{\G_a}(\cO_{L}, \fm_L^{-(r+s)})$ in general.)
\end{remark}

\begin{cor}\label{ga-ga-p}
Suppose the characteristic of $k$ is $p \ge 3$ and let $K$ be a perfect $k$-field.
Then  $h_{0,\Nis}(\wt{\G_a}\te\wt{\G_a})(K)=0$. 
\end{cor}
\begin{proof}
Since $\Omega^1_K=0$, for perfect $K$, Proposition \ref{prop:map-to-Ga-te-Ga} yields a surjection
\eq{ga-ga-p1}{K\xr{\simeq}  h_{0,\Nis}(\G_a^\#\te \G_a^\#)(K)
                          \surj h_{0,\Nis}(\wt{\G_a}\te\wt{\G_a})(K), \quad a\mapsto [a,1]_K.}
Take $a \in K$ and
consider the following functions on $\P^1_K$
\begin{align*}
g_1=at,\ 
g_2=t^p,\
f=\frac{t^{p+1}-1}{t^{p+1}} \in K(t).
\end{align*}
Set
\[
\alpha_a := \sum_{x\in \P^1_K \setminus\{\infty\}}v_x(f) \Tr_{K(x)/K}([g_1(x), g_2(x)]_{K(x)})
\in h_{0,\Nis}(\G_a^\#\te\G_a^\#)(K).
\]
Let $x_1,\ldots, x_n\in\P^1_K$ be the points corresponding to the irreducible factors of $t^{p+1}-1$ and denote by
$K_i=K(x_i)$ their residue fields.
Then 
\[\alpha_a= \sum_{i=1}^n \Tr_{K_i/K}([a t(x_i), t(x_i)^p]_{K_i}).\]
Under the map $h_{0,\Nis}(\G_a^\#\te\G_a^\#)(K)\to K$ induced from the isomorphism in Theorem \ref{thm:Ga-te-Ga-arbCh}, 
$\alpha_a$ is sent to
\[\sum_{i=1}^n \Tr_{K_i/K}(a t(x_i)^{p+1})= \sum_{i=1}^n [K_i:K] a=(p+1)a=a.\]
Since this last map is the inverse of the first map in \eqref{ga-ga-p1} we obtain 
\[[a,1]_K= \alpha_a\quad \text{in }h_{0,\Nis}(\G_a^\#\te\G_a^\#)(K).\]
On the other hand,
since $f\equiv 1 \mod 4(\infty)$
and $g_1, g_2 \in \wt{\G_a}(\P^1_K,2(\infty))$, by Remark \ref{rmk:Ga},
the image of $\alpha_a$ in 
$h_{0,\Nis}(\wt{\G_a}\te \wt{\G_a})(K)$ is zero by Theorem \ref{thm:CI-te-fields}.
This proves the statement.
\end{proof}

\subsection{Tensors of the additive - with the multiplicative group}

\begin{prop}\label{prop:DR-to-T}
Assume $\ch(k)\neq 2,3,5$ and let $K$ be a field containing $k$.
Then the morphism 
\eq{prop:DR-to-T0}{
\theta: \Omega^n_K \to h_{0,\Nis}(\G_a^\# \te (\G_m^\#)^{\te n})(K), }
\[\quad a\dlog b_1\cdots \dlog b_n\mapsto [a, b_1,\ldots, b_n]_K,
\]
is well-defined and surjective.
\end{prop}
\begin{proof}
First we show the well-definedness.
We employ the same strategy of Hiranouchi \cite{Hira} and Ivorra-R\"ulling \cite{IR},
but with different  functions, since the modulus condition is different.
According to \cite[Lemma 4.1]{BE03},
our task is to show that the elements of the forms 
\begin{enumerate}
\item[(i)]
$
[ca, a, b_2, \cdots, b_n]_K + [c(1-a), (1-a), b_2, \cdots, b_n]_K\ \  
$
\item[(ii)]
$
[a, b_1,\cdots,b_n]_K \quad 
\text{($b_i=b_j$ for some $i<j$)}
$
\end{enumerate}
are zero in $h_{0,\Nis}(\G_a^\# \te (\G_m^\#)^{\te n})(K)$,
where $a\in K\setminus\{0,1\}, b_i\in K^\times, c\in K$.
We will use the notations introduced in \S \ref{not:GmM}.

(i)
Set $\boldsymbol{b}:=b_2,\ldots, b_n$.
If $a$ is a root of $t^2-t+1=0$, then so is $1-a$; furthermore, $a^3=-1$ and hence
\[[ca, a, \boldsymbol{b}]_K=0=[c(1-a),(1-a), \boldsymbol{b}]_K.\]
Given $a\in K\setminus\{0,1\}$ with $a^2-a+1\neq 0$ and $b_i\in K^\times$ for $i=2,\dots,n$, 
we consider the following rational functions
\begin{align*}
&f=\frac{(t-a)(t-(1-a))(t+1)(t^2+a^2-a+1)}{t^5+a(1-a)(a^2-a+1)},
\\
&g_0=ct,\qquad g_1=t, \qquad g_i=b_i ~(i=2,\dots,n)
\end{align*}
so that
\begin{align*}
&g_0^\# \in \G_a^\#(\P^1,2(\infty)),\\
&g_1^\# \in \G_m^\#(\P^1,(0)+(\infty)),\\
&g_i^\# \in \G_m^\#(\P^1,\emptyset)\quad (i=2,\dots,n).
\end{align*}
One can easily check that $f\equiv 1\mod (0)+3(\infty)$.
Therefore (R2) shows that
\eq{prop:DR-to-T1}{
\sum_{x\in \P^1\setminus\{0,1,\infty\}}v_x(f) \Tr_{K(x)/K}([g_0(x), g_1(x),\boldsymbol{b}]_{K(x)})=0.
}
We write $\alpha_i, \beta_j$ for a root of $t^2+a^2-a+1=0$ and 
$t^5+a(1-a)(a^2-a+1)=0$, respectively.
Put $K'=K(\alpha_i,\beta_j)$.
By the assumption on the characteristic, $[K':K]$ is prime to $\ch(K)$, in particular $K'/K$ is separable.
Denote by $\pi: \Spec K'\to \Spec K$ the induced map.
For $x\in \P^1_K$ with $v_x(f)\neq 0$ we have $\pi^*\Tr_{K(x)/K}= \sum_\sigma \sigma$,
where $\sigma: K(x)\inj K'$ runs through all the $K$-embeddings.
Thus $\pi^*\eqref{prop:DR-to-T1}$ yields the following equality 
in $h_{0,\Nis}(\G_a^\# \te (\G_m^\#)^{\otimes n})(K')$  (we write $[-]$ instead of $[-]_{K'}$)
\begin{align*}
0 &=[ca, a,\boldsymbol{b}]
+[c(1-a), 1-a,\boldsymbol{b}]\\
&\quad
+[-c, -1,\boldsymbol{b}]+\sum_i [c\alpha_i, \alpha_i,\boldsymbol{b}]
-\sum_j [c\beta_j, \beta_j,\boldsymbol{b}]\\
&=[ca, a,\boldsymbol{b}]
+[c(1-a), 1-a,\boldsymbol{b}]
+\sum_i \left[\cfrac{c\alpha_i}2,\alpha_i^2,\boldsymbol{b}\right]
-\sum_j\left[\frac{c\beta_j}5, \beta_j^5,\boldsymbol{b}\right]\\
&=[ca, a,\boldsymbol{b}]
+[c(1-a), 1-a,\boldsymbol{b}]
+\left[\cfrac{c}2\cdot \sum_i \alpha_i , -a^2+a-1,\boldsymbol{b}\right]\\
&\quad
-\left[\frac{c}5 \cdot \sum_j \beta_j, a(1-a)(-a^2+a-1), \boldsymbol{b}\right]\\
&=[ca, a,\boldsymbol{b}]_{K'}+[c(1-a), 1-a,\boldsymbol{b}]_{K'}.
\end{align*}
Since $[K':K]$ is prime to $\ch(K)$
taking the norm yields the desired vanishing in $h_{0,\Nis}(\G_a^\# \te (\G_m^\#)^{\otimes n})(K)$.

(ii)
We know that the element of the form
\[
[b_1,\cdots, b_n]_K \qquad \text{with $b_i=b_j$ for some $i<j$}
\]
is 2-torsion in $h_0((\G_m^\#)^{\otimes n})(K)\cong K_n^M(K)$ (see Proposition \ref{prop:Milnor-K}).
Thus the element $[(a/2), b_1,\cdots, b_n]_{K}$ 
is also 2-torsion in $h_0(\G_a^\# \te (\G_m^\#)^{\otimes n})(K)$,
and we have
\[
[a, b_1, \cdots, b_r]=2\cdot [\frac{a}2, b_1,\cdots, b_r]_K=0.
\]
Hence we have a well-defined map $\theta$ as in the statement. 
The surjectivity of $\theta$ follows once we know that it is compatible with traces, i.e.,
it suffices to show
\eq{prop:DR-to-T2}{\Tr_{K'/K}\circ \theta_{K'}= \theta_K\circ \Tr_{K'/K},}
for a finite field extension $K'/K$. If $K'/K$ is separable,
this follows from the fact that $\Tr_{K'/K}$ satisfies a projection formula
on both sides of \eqref{prop:DR-to-T0} (cf. the proof of Proposition \ref{prop:map-to-Ga-te-Ga}).
If $K'=K[x]$ for some $x\in K'$ with $x^p=:y\in K$, where $p=\ch(K)$,
then we can write any element in $\Omega^1_{K'}$ as a sum of elements of the form
\[\beta_i=a x^i \dlog \boldsymbol{b}, \quad \gamma_i=ax^i\dlog x\dlog\boldsymbol{c}, \quad i=0,\ldots, p-1,\]
where $a\in K$ and $b_i, c_j\in K^\times$.  The equality \eqref{prop:DR-to-T2} holds on the elements $\beta_i$ 
again by the projection formula. 
We have $\Tr_{K'/K}(\gamma_0)=a \dlog y\dlog\boldsymbol{c}$, cf. \eqref{lem:transferP1}.
Thus on $\gamma_0$ the equality \eqref{prop:DR-to-T2} translates into
\[\Tr_{K'/K}([a, x, \boldsymbol{c}]_{K'})= [a, y, \boldsymbol{c}]_K,\]
which  holds by the projection formula. 
For $i\in [1,p-1]$, we can write $\gamma_i= \frac{1}{i} ad (x^i)\dlog\boldsymbol{c}$.
Thus $\Tr_{K'/K}(\gamma_i)=0$ and it remains to show
\eq{prop:DR-to-T3}{\Tr_{K'/K}([ax^i, x^i, \boldsymbol{c}]_{K'})=0\quad \text{in } h_0(\G_a^\#\te (\G_m^\#)^{\te n})(K).}
To this end define the functions
\[f= \frac{(t^p-y^i)^2}{t^{2p}+y^{2i}}, \quad g_0:=at, \quad g_1:=t,\quad g_j:=c_j,\quad j\ge 2.\]
Then $g_i^\#$ are as in (i) at the beginning of this proof and 
$f\equiv 1$ mod $(0)+3(\infty)$. Let $\e \in \bar{K}$ with $\e^2=-1$ and set $K_1:= K[\e]$ and 
$K_1':= K_1[x]= K_1\otimes_K K'$ and denote by $\pi: \Spec K_1\to \Spec K$ the induced map.
Then by (R2) we obtain in $h_0(\G_a^\#\te (\G_m^\#)^{\te n})(K_1)$
\begin{align*}
0 & = 2\Tr_{K_1'/K_1}([ax^i, x^i,\boldsymbol{c}]_{K_1'}) \\
     &\quad  - \Tr_{K_1'/K_1}([a\e x^i, \e x^i, \boldsymbol{x}]_{K_1'} + [-a\e x^i, - \e x^i, \boldsymbol{c}]_{K_1'}) \\
   &  = 2\cdot \pi^*\Tr_{K'/K}([ax^i, x^i,\boldsymbol{c}]_{K'}).
\end{align*}
Since $[K_1: K]$ divides 2 applying $\pi_*$ yields the wanted vanishing .
This completes the proof.
\end{proof}

\begin{thm}\label{thm:DR-arbCh}
Assume $\ch(k)\neq 2,3,5$. Then there are canonical isomorphisms in $\RSC_{\Nis}$
\[h_{0,\Nis}(\G_a^M/0\otimes (\G_m^M/1)^{\otimes n})\cong h_{0,\Nis}(\G_a^\#\te (\G_m^\#)^{\te n}) \cong \Omega^n.\]
\end{thm}
\begin{proof}
The proof is similar to the proof of Theorem \ref{thm:Ga-te-Ga-arbCh}, so we will be brief.
For a prime correspondence $Z\in \Cor(X, \P^1_s\setminus\{\infty\}\times \prod_{i=1}^n (\P^1_{t_i}\setminus\{0,\infty\}))$
define 
\[\eta(Z):= Z^*(s\dlog t_1\cdots \dlog t_n)\in \Omega^n(X).\]
We can extend $\eta$ to a morphism in $\PST$
\[\eta: \frac{\uomega_!\Z_\tr(\G_a^M\te (\G_m^M)^{\te n})}{
  i_{0,0} (\uomega_!\Z_\tr((\G_m^M)^{\te n}))+\sum_{j} i_{1, j}(\uomega_!\Z_\tr(\G_a^M\te (\G_m^M)^{\te n-1}))}
\to \Omega^n,\]
where $i_{\e, j}$ denotes the obvious closed immersion which inserts $\e$ at the $j$th position.
To show that $\eta$ induces a well-defined map in $\RSC_{\Nis}$
\[\bar{\eta}: h_{0,\Nis}(\G_a^M/0\otimes (\G_m^M/1)^{\otimes n})\to \Omega^n \]
we can proceed as in the proof of Theorem \ref{thm:Ga-te-Ga-arbCh} to see that  it suffices to show the following:
Let $K/k$ be a function field, $\pi: C\to \P^1_k\times_k (\P^1_k)^n$ be a $k$-morphism from a regular $K$-curve and 
$f\in K(C)$ be a function with 
\[f\equiv 1 \text{ mod } \pi^* (2\cdot p_0^*\infty +\sum_{i=1}^n p_i^*(0+\infty))=:D,\]
where $p_i: \P^1\times (\P^1)^n\to \P^1$ is the projection to the $i$th factor; then we have to show
\eq{thm:DR-arbCh1}{0=\sum_{x\in C\setminus |D|} \Res_{C/K,x}(\pi^*(s\dlog t_1\cdots \dlog t_n)\dlog f).}
As in {\em loc. cit.} we see that $\Res_{C/K,x}(\pi^*(s\dlog t_1\cdots \dlog t_n)\dlog f)=0$, for $x\in |D|$.
Thus the vanishing \eqref{thm:DR-arbCh1} follows from the reciprocity law for the residual symbol.
Now it follows from Proposition \ref{prop:DR-to-T} that $\bar{\eta}$ is an isomorphism in the same way as 
Theorem \ref{thm:Ga-te-Ga-arbCh} follows from Proposition \ref{prop:map-to-Ga-te-Ga}.
\end{proof}

\begin{thm}\label{thm:DR}
Assume $\ch(k)=0$ and let $\bullet\in\{\te,\tte\}$. There are isomorphisms in $\RSC_{\Nis}$
\[h_{0,\Nis}(\G_a^\# \bullet (\G_m^\#)^{\bullet n})
\cong
h_{0,\Nis}(\wt{\G_a}\bullet (\wt{\G_m})^{\bullet n})
\cong h_{0,\Nis}(\wt{\G_a}\bullet \wt{\cK^M_n})
\cong\Omega^n.
\]
\end{thm}
\begin{proof}
Using \eqref{cor:DR01} it is direct to check that the collection of maps for $X,Y\in \Sm$
\[\phi_{X,Y}: \G_a(X)\times \cK^M_n(Y)\to \Omega^n(X\times Y), \quad (a,b)\mapsto p_X^*(a)\cdot p_Y^*\dlog(b)\]
define an element in $\Lin^{\bullet}_{\RSC}(\G_a, \cK^M_n\,;\,\Omega^n)$,
where $p_X : X \times Y \to X$ and $p_Y : X \times Y \to Y$ are projections. 
By Theorem \ref{thm:ten-rep}, $\{\phi_{X,Y}\}$ induces a morphism in $\RSC_{\Nis}$
\[\phi:  h_{0,\Nis}(\wt{\G_a}\bullet \wt{\cK^M_n}) \to \Omega^n.\]
It follows from  Corollary \ref{cor:assoc} and  the natural maps $\te\to \tte$ and $\G_*^\#\to \wt{\G_*}$ that we
can arrange the other reciprocity sheaves  in the statement in the following diagram:
\[
\xymatrix{
h_{0,\Nis}(\G_a^\# \te (\G_m^\#)^{\te n})
\ar[r]\ar[d]
&
h_{0,\Nis}(\G_a^\# \tte (\G_m^\#)^{\tte n})
\ar[d] &
\\
h_{0,\Nis}(\wt{\G_a}\te (\wt{\G_m})^{\te n})
\ar[r] \ar[d]
&
h_{0,\Nis}(\wt{\G_a}\tte (\wt{\G_m})^{\tte n})
\ar[d] &
\\
h_{0, \Nis}(\wt{\G_a}\te \wt{\cK^M_n}) 
\ar[r]
&
h_{0,\Nis}(\wt{\G_a}\tte \wt{\cK^M_n}) 
\ar[r]^-\phi
& \Omega^n. 
}
\]
By Lemma \ref{lem:ffc}, we only have to show that
these maps are isomorphisms on all function fields $K$ over $k$.
Since all maps are surjective by construction,
it suffices to show the injectivity of the map
\[\phi_1: h_{0,\Nis}(\G_a^\# \te (\G_m^\#)^{\te n})(K) \to \Omega_K^n.\]
But $\phi_1\circ\theta=\id_{\Omega^n}$, where $\theta$ is the surjective map from Proposition \ref{prop:DR-to-T}.
This completes the proof.
\end{proof}

\subsection{Twists of reciprocity sheaves} 
\begin{para}\label{para:twist}
Let $F\in \RSC_{\Nis}$. We define the twist $F\tw{n}$, for $n\ge 0$, as follows
\[F\tw{0}:=F,\quad F\tw{n}:=h_{0,\Nis}(\wt{F\tw{n-1}}\te\wt{\G_m}), \,n\ge 1.\]
Thus $F\tw{1}=h_{0,\Nis}(\wt{F}\te \wt{\G_m})$, 
$F\tw{2}=h_{0,\Nis}((h_{0,\Nis}(\wt{F}\te\wt{\G_m}))^{\widetilde{\,}} \, \te\wt{\G_m})$,
 etc.
(There is also a version with $\tte$ for which we don't introduce an extra notation.)
We obtain the following properties:
\begin{enumerate}[label= (\arabic*)]
\item\label{para:twist1} $\tw{n}: \RSC_{\Nis}\to \RSC_{\Nis}$ is a functor which preserves surjections;
\item\label{para:twist2} $\tw{n}\circ \tw{m}=\tw{n+m}: \RSC_{\Nis}\to \RSC_{\Nis}$;
\item\label{para:twist3}for $F\in \HI_{\Nis}$
 \[ F\tw{n}= F\ten{\HI_{\Nis}} (\G_m^{\ten{\HI_{\Nis}} n}); \quad 
\text{in particular}\quad \Z\tw{n}=\cH^n(\Z(n))\cong\cK^M_n,\]
see \ref{para:tenHI-Gm} for notation;
\item\label{para:twist4} for all $F\in \RSC_{\Nis}$ there are natural surjections 
\[F\ten{\NST} \cK^M_n\surj F\tw{n}\quad \text{in }\NST\]
and 
\eq{para:twist41}{h_{0,\Nis}(\wt{F}\te (\wt{\G_m})^{\te n})\surj F\tw{n} \quad \text{in }\RSC_{\Nis};}
\end{enumerate}
Indeed, the second statement of \ref{para:twist1} follows from Lemma \ref{lem-T-rexact};
the first statement of \ref{para:twist3} follows from  Theorem \ref{thm:tHI-tRSC}, 
for the second statement see \ref{para:tenHI-Gm};
finally \ref{para:twist4} follows from this and the Corollaries \ref{cor:assoc} and \ref{cor:assocII}.
\end{para}

\begin{cor}\label{cor:DR-twist}
Assume $\ch(k)=0$. Then the isomorphism from Theorem \ref{thm:DR} factors as isomorphisms
\[h_{0,\Nis}(\wt{\G_a}\te \wt{\cK^M_n})\xrightarrow[\eqref{para:twist41}]{\simeq} \G_a\tw{n}\xr{\simeq} \Omega^n.\]
\end{cor}
\begin{proof}
It suffices to show that we have a factorization as in the statement. Assume it is proven for $n-1$.
Then the maps for $X,Y\in \Sm$
\[\G_a\tw{n-1}(X)\times \G_m(Y)= \Omega^{n-1}(X)\times \G_m(Y)\to \Omega^n(X\times Y),\]
\[
(\alpha, b)\mapsto p_1^*\alpha \cdot p_2^*\dlog b,
\]
form an element in $\Lin^\te_{\RSC}(\G_a\tw{n-1}, \G_m; \Omega^n)$ (use \eqref{cor:DR01}).
It is direct to check that the induced map $\phi: \G_a\tw{n}\to \Omega^n$ factors the isomorphism from
Theorem \ref{thm:DR} as stated.
\end{proof}

\begin{para}\label{para:DR}
We continue to assume $\ch(k)=0$ and explain how to reinterpret classical maps related
to the de Rham complex in terms of twists of $\G_a$:
\begin{enumerate}[label=(\arabic*)]
\item\label{para:DR1} 
We can define a morphism 
\[\partial:\G_a\to \G_a\tw{1}\quad \text{in }\RSC_{\Nis}\]
as follows:
Let $U=\Spec A$ be a smooth affine $k$-scheme.
For $a\in A$ and $\epsilon\in \{0,1\}$ define
\[\partial_\epsilon (a)= \text{image of } (a-\epsilon)\otimes (a-\epsilon) \in h_{0,\Nis}(\wt{\G_a}\te \wt{\G_m})(U_{a-\e}),\]
where $U_{a-\e}=\Spec A_{a-\e}$.
By Theorem \ref{thm:DR} we have $\partial_0(a)=\partial_1(a)$ when restricted to $U_{a}\cap U_{a-1}$ and hence
we obtain a well-defined element $\partial_U(a)\in h_{0,\Nis}(\wt{\G_a}\te\wt{\G_m})(U)= \G_a\tw{1}(U)$.
It follows from Theorem \ref{thm:DR} that $\partial_U : \G_a(U)\to \G_a\tw{1}(U)$ is a group homomorphism and that
$\partial_V=\partial_U$ on $U\cap V$. Hence these glue to give a morphism $\partial_X: \G_a(X)\to \G_a\tw{1}(X)$, for all 
$X\in \Sm$. Under the isomorphism $\G_a\tw{1}\cong\Omega^1$  this map coincides with the differential 
$d:\sO(X)\to \Omega^1(X)$, thus $\partial=\{\partial_X\}$ is a morphism of sheaves with transfers.
By Theorem \ref{thm:DR} the de Rham-complex is isomorphic to
\[\Omega^\bullet\cong 
\G_a\xr{\partial} \G_a\tw{1}\xr{\partial\tw{1}} \G_a\tw{2}\xr{\partial\tw{2}}\ldots \G_a\tw{n}
\xr{\partial\tw{n}}\G_a\tw{n+1}\to\ldots.\]
\item\label{para:DR2} The twists of the natural inclusion $\iota:\Z\inj \G_a$ give rise to the $\dlog$-map, i.e.,
\[ [\dlog :\cK^M_n\to \Omega^n]= [\iota\tw{n}: \Z\tw{n}\to \G_a\tw{n}].\]
\item\label{para:DR3} Consider the exact sequence in $\RSC_{\Nis}$
\[0\to \Z\to \G_a\to \G_a/\Z\to 0.\]
Since $\Z$ is proper we have a surjection 
\[\tau_!\widetilde{\G_a}(\sO_L,\fm_L^{-n})\surj \tau_!\widetilde{(\G_a/\Z)}(\sO_L,\fm_L^{-n}),\quad L\in \Phi, n\ge 0,\]
see  \cite[Lemma 4.31]{RS}.
Thus \ref{para:DR2} and Lemma \ref{lem-T-rexact}\ref{lem-T-rexact3} yield  (cf \cite[6.10]{RS})
\[[X\mapsto H^1(X_{\Zar}, \sO_X^\times\xr{\dlog}\Omega^1_{X/\Z})]= (\G_a/\Z)\tw{1}.\]
\end{enumerate}
\end{para}

\subsection{Some applications to zero-cycles with modulus}

\begin{cor}\label{cor:ch-prod-curves}
Let $M_i:=(C_i,D_i)$ be a proper modulus pair, 
where $C_i$ is a smooth (proper) absolutely integral curve over $k$,
for $i=1,2$.
Let $J_{M_i}$ be the generalized Jacobian for $M_i$.
Put $J_{M_i}^\#=h_0^\bcube(M_i)^0$ (see \eqref{eq:def-reduced-h}).
Assume that $C_i$ admits a degree one divisor.
Then for any $k$-field $K$ we have an isomorphism
\[
\CH_0((M_1\te M_2)_K)\cong \Z\oplus J_{M_1}(K)\oplus J_{M_2}(K)\oplus h_0(J_{M_1}^\#\te J_{M_2}^\#)(K).
\]
(See \ref{para:Chow-mod} for the notation.)
In particular, for $\G_a^M$ and $\G_m^M$ 
we have the following (see notations in \ref{not:GmM}).
\begin{enumerate}
\item
In $\ch(k)\neq 2,3,5$, we have
$$
 \CH_0((\G_a^M\te \G_m^M)_K)\cong \Z\oplus K \oplus K^\times \oplus \Omega_K^1.
$$
\item
In $\ch(k)\neq 2$, we have
(see \ref{para:P} for the notation $I_\Delta(k)$)
$$\CH_0((\G_a^M\te \G_a^M)_K)\cong \Z\oplus K^{\oplus 2}
\oplus (K\te_\Z K)/I_{\Delta}(K)^2.$$
\end{enumerate}
\end{cor}
\begin{proof}
By the approximation lemma, 
we may find a degree one divisor on $C_i \setminus |D_i|$,
which gives rise to a decomposition
$
h_0^{\bcube}(M_i)\cong \Z\oplus J_{M_i}^\#.
$
By Lemma \ref{lem:hcube-ten} we obtain
\[
h_0(M_1\te M_2)\cong \Z\oplus \omega_!(J_{M_1}^\#)\oplus \omega_!(J_{M_2}^\#)\oplus 
\omega_!(J_{M_1}^\#\te J_{M_2}^\#).
\]
Taking the value at $\Spec(K)$, we have by \cite[Theorem 1.1]{RY}
\[
\CH_0((M_1\te M_2)_K)\cong \Z\oplus J_{M_1}(K)\oplus J_{M_2}(K)\oplus h_{0,\Nis}(J_{M_1}^\#\te J_{M_2}^\#)(K).
\]
Now the isomorphisms (1) and (2) follows from 
Theorem \ref{thm:DR-arbCh} and Theorem \ref{thm:Ga-te-Ga-arbCh}.
\end{proof}

\begin{remark}
Let us take $M_1=M_2=\G_a^M$ and suppose $k=\C$.
Binda and Krishna showed that there is a surjective homomorphism
$\CH_0(M_1\otimes M_2) \to \Z \oplus k \oplus k\oplus k$
with non-trivial kernel \cite[Theorem 10.10]{Binda-Krishna}.
Corollary \ref{cor:ch-prod-curves} shows that 
this kernel is isomorphic to $\Omega_{\mathbb{C}}^1$.

\end{remark}

\begin{prop}\label{prop:surj-hi-ch}
Let $M=(X,D)$ be a proper modulus pair
such that $X$ is of pure dimension $d$,
and let $n$ be a positive integer.
Then for any $k$-field $K$ there is a surjection
\[
h_0(h_0^\bcube(M)\te \wt{\G_m}^{\te n})(K)\lra \CH^{d+n}(M_K;n)_{\rm ssup}.
\]
Here $\CH^{d+n}(M_K;n)_{\rm ssup}$ is the higher Chow group with strong-sup-modulus
defined in \cite[Definition 2.6]{MR3846053}.
\end{prop}
\begin{proof}
We know the isomorphism
\[
K^\times\cong \CH^1(K;1).
\]
By Proposition \ref{prop:ext-prod} below, the external product
\[
\boxtimes : \CH_0(M_K)\te (K^\times)^{\te n}\lra \CH^{d+n}(M_K;n)_{\rm ssup}
\]
induces a surjection
\[
(h_0(M)\ten{\PST}\G_m\ten{\PST}\cdots\ten{\PST}\G_m)(K)\lra \CH^{d+n}(M_K;n)_{\rm ssup}.
\]
We now consider the element as in (R2) of Theorem \ref{thm:CI-te-fields}
\begin{equation}\label{rel-el}
\sum_{c\in C^\o}v_c(f)\alpha(c)\te g_1(c)\te\cdots\te g_n(c)
\end{equation}
where $(C,E)$ is a proper modulus pair with $C$ regular projective curve over $K$ and $\alpha\in h_0^\bcube(M_K)(C,E)$, $g_i\in \wt{\G_m}(C, E)$ $(i=1,\dots,n)$ and $f\in K(C)^{\times}$ such that $f\equiv 1\mod E+n(E)_\mathrm{red}$.
Our task is to show that the element \eqref{rel-el} goes zero in $\CH^{d+n}(M_K;n)_{\rm ssup}$ by the external product map.
We may assume that $\alpha$ is elementary, i.e.  it is given by one integral closed subscheme $Z$ in $C^\o\times X_K^\o$ where $X_K^\o=X_K\setminus |D_K|$.

Let $\ol{Z}^N$ be the normalization of the closure of $Z$ in $C\times X_K$ and let $p: \ol{Z}^N\to C$ and $q:\ol{Z}^N\to X_K$ be the projection.
Then we have a proper map
\[
\varphi:=q\times p^*f\times p^*g_1\times\cdots\times p^*g_n: \ol{Z}^N\lra X_K\times (\P_K^1)^{n+1}.
\]
Let $W$ be the restriction of $\varphi_*(\ol{Z}^N)$ to $X_K^\o\times (\P_K^1\setminus\{1\})^{n+1}$.
We will show that $W$ is in good position.
For a face $F$ of $\codim F=1$, it is enough to show that $W\not\subset X_K^\o\times F$.
If $W\subset X_K^\o\times F$, then by the construction of $W$, one of $f, g_i~(i=1,\dots,n)$ is a constant function on $C$ with value $0$ or $\infty$.
But this cannot happen since $f=1$ on $|E|$ and $g_i$ are non-zero regular functions on $C\setminus|E|$.
These conditions on $g_i~(i=1,\dots,n)$ and $f$ also yield that $W\cap X_K^\o\times F=\emptyset$ for $\codim F\geq 2$.
Thus $W$ is in good position.
By the conditions on $\alpha$, $g_i~(i=1,\dots,n)$ and $f$, we have
\mlnl{
\varphi^*(D_K\times (\P_K^1)^{n+1})=q^*(D_K)\leq p^*(E)\leq (p^*f)^*(\{1\})\\
= \varphi^*(X_K\times \{1\}\times(\P_K^1)^n).
}
Thus $W\in z^{d+n}(M_K;n+1)_\mathrm{ssup}$. The base change formula yields
\begin{align*}
\partial (W)
&=(\alpha\times g_1\times \cdots \times g_n)_*(\div(f))\\
&=\sum_{c\in C^\o}v_c(f)\mathrm{Tr}_{K(c)/K}[(\alpha(c),g_1(c),\dots,g_n(c))]
\end{align*}
which is the image of \eqref{rel-el} by the external product map.
This completes the proof.
\end{proof}

\begin{prop}[\text{\cite[Lemma 3.6, Proposition 3.7]{MR3723807}}]\label{prop:ext-prod}
Let $M=(X,D)$ be a modulus pair and let $Y$ be a smooth variety over $k$.
Denote a modulus pair $(X\times Y, D\times Y)$ by $M\te Y$.
There is an external product
\[
\boxtimes: \CH^r(M;n_1)_\mathrm{ssup}\te \CH^s(Y;n_2)\lra \CH^{r+s}(M\te Y;n_1+n_2)_\mathrm{ssup}
\]
compatible with flat pull-back and proper push-forward.
\end{prop}

Concerning a conjecture and a result \cite[Conjecture 2.8, Corollary 3.5]{KP} for a relation among modulus conditions, we give an example for higher 0-cycles with modulus.

\begin{cor}\label{Cor:hCH-GaM}
Suppose that $\ch(k)\neq 2,3,5$.
Let $n$ be a positive integer. 
Then for any $k$-field $K$ there are the following isomorphisms
\[
\CH^{n+1}(\P_K^1,2\infty;n)_{\rm ssup}\cong\CH^{n+1}(\P_K^1,2\infty;n)\cong K_n^M(K)\oplus \Omega_K^n.
\]
\end{cor}

\begin{proof}
We use the notation $(\G_{a}^M)_K=(\P_K^1,2\infty)$.
From Proposition \ref{prop:Milnor-K} and Proposition \ref{prop:surj-hi-ch}, we obtain the following commutative diagram with 
split rows
\begin{align}\label{diagram}
\xymatrix{
h_0(\G_a^\#\te \wt{\G_m}^{\te n})(K)\ar[r]\ar@{->>}[d]&h_0(\G_a^M\te \wt{\G_m}^{\te n})(K)\ar[r]^{\qquad \pi_*^M}\ar@{->>}[d]&\ar[d]^{\cong}h_0(\wt{\G_m}^{\te n})(K)\ar@/^1pc/[l]^{i_{0*}}\\
\ker(\pi_*^\mathrm{ssup})\ar[r]\ar@{->>}[d]&\CH^{n+1}((\G_{a}^M)_K;n)_{\rm ssup}\ar[r]^{\qquad \pi_*^\mathrm{ssup}}\ar@{->>}[d]&\ar@{=}[d]K_n^M(K) \ar@/^1pc/[l]^{i_{0*}}\\
\ker(\pi_*)\ar[r]&\CH^{n+1}((\G_{a}^M)_K;n)\ar[r]^{\qquad \pi_*}&K_n^M(K) \ar@/^1pc/[l]^{i_{0*}}
}
\end{align}
where the map $\pi_*^M$ is induced by the structure map $\pi:\G_a^M\to \Spec(k)$, and the maps $\pi_*^\mathrm{ssup}, \pi_*$ are the proper push-forward to $\CH^n(K;n)\simeq K_n^M(K)$ via $\pi$.
Since the map $\pi$ has a section given by the closed immersion $i_0:\Spec(k)\to \G_a^M$ into $0$, 
our task is to show that the left terms of \eqref{diagram} are isomorphic to $\Omega_k^n$.

We have a surjection
\[
\CH^{n+1}((\G_{a}^M)_K;n)\lra \CH^{n+1}(\P_K^1\setminus\{0\}, 2\infty;n)
\]
given by the pull-back along the open immersion $\P_K^1\setminus\{0\}\to \P_K^1.$
The unique automorphism of $\P_K^1$ which switches $0$ and $\infty$ and fixes $1$ induces an isomorphism
\[
\CH^{n+1}(\P_K^1\setminus\{0\}, 2\infty;n)\cong  \CH^{n+1}(\A_K^1, 2[0];n).
\]
The right hand side is the additive Chow group  and Bloch--Esnault \cite{Bloch-Esnault} 
 proved the following isomorphism
\[
\CH^{n+1}(\A_K^1, 2[0];n)\cong \Omega_K^n,
\]
which sends a $K$-rational point $(a,b_1,\dots, b_n)$ to $\frac1a \dlog b_1\cdots \dlog b_n$.
Thus we have a surjection
\[
\CH^{n+1}((\G_{a}^M)_K;n)\lra \Omega_K^n,
\]
which sends a $K$-rational point $(a,b_1,\dots, b_n)$ to $a \dlog b_1\cdots \dlog b_n$.

We now consider the following composite map 
\[
\Omega_K^n\lra h_0(\G_a^\#\te\wt{\G_m}^{\te n})(K)\lra \ker(\pi_*)\hookrightarrow \CH^{n+1}((\G_{a}^M)_K;n)\lra  \Omega_K^r,
\]
where the first map is the surjective map in Proposition \ref{prop:DR-to-T}.
The composite map is the identity of $\Omega_K^n$, which yields that the left terms of \eqref{diagram} are isomorphic to $\Omega_K^n$.
This completes the proof.
\end{proof}

\bibliographystyle{amsalpha}

\providecommand{\bysame}{\leavevmode\hbox to3em{\hrulefill}\thinspace}
\providecommand{\MR}{\relax\ifhmode\unskip\space\fi MR }
\providecommand{\MRhref}[2]{%
  \href{http://www.ams.org/mathscinet-getitem?mr=#1}{#2}
}
\providecommand{\href}[2]{#2}

\end{document}